\documentclass[11pt,reqno]{article}   %style-phys/cs 
\usepackage{fullpage}

\usepackage{microtype}
\usepackage[unicode=true]{hyperref}
\usepackage[british]{babel}
\usepackage{amsmath}
\usepackage{cite}
\usepackage{amsfonts}
\usepackage{amssymb}
\usepackage{amsthm}
\usepackage{cleveref}  
\usepackage{authblk}
\usepackage{mypack} %custom commands
\usepackage[pdftex]{color,graphicx}
\usepackage{adjustbox}
\usepackage{soul}
\usepackage{comment}
\usepackage{bbold}
\usepackage{tikz} %for diagrams
\usetikzlibrary{decorations.pathreplacing,angles,quotes}
\usetikzlibrary{matrix,arrows,decorations.pathmorphing}
\usepackage{tikz-cd}
\usetikzlibrary{arrows}
\usetikzlibrary{decorations.markings}

\usepackage{bbold}
\usepackage{diagbox}
\usepackage{caption}
\usepackage{subcaption}
\usepackage{pdfpages} 
\usepackage{blkarray}
\usepackage{centernot}
\usepackage{mathtools}
\usepackage{stmaryrd}
\usepackage{soul}  %strikethrough text: \st{}
\usepackage{enumitem}

\definecolor{bluegray}{rgb}{0.4, 0.6, 0.8}
\definecolor{turquoise}{rgb}{0.2, 0.7, 0.6}
\usepackage{todonotes}
\newcounter{commcount}

\newcommand{\on}[1]{\operatorname{#1}}

\usepackage{todonotes}
\usepackage{xcolor}
\title{Vertex structure of fiber products of probability polytopes}

\author{Aziz Kharoof\thanks{aziz.kharoof@bilkent.edu.tr} }
\author{Cihan Okay\thanks{cihan.okay@bilkent.edu.tr}} 
\affil{{\small{Department of Mathematics, Bilkent University, Ankara, Turkey}}}

\date{\today}

\begin{document}

\maketitle

\begin{abstract}
We develop tools for characterizing vertices of fiber products of polytopes and apply them to simplicial distribution polytopes, a class of probability polytopes arising in quantum foundations and quantum information. In the theory of simplicial distributions, a pair of simplicial sets encoding measurement and outcome spaces determines a convex polytope of compatible probability assignments. Our first results give geometric criteria for detecting vertices of fiber products in terms of support data. These results are obtained in the more general framework of inverse limits of diagrams of polytopes in standard form, and they translate to corresponding criteria for simplicial distributions on arbitrary colimits of measurement spaces. We then focus on one-dimensional measurement spaces, where simplicial distributions recover and generalize local marginal polytopes in graphical models. In this setting, our sharpest results concern dipole graphs, for which we obtain a complete characterization of vertices and refine it to a graph-theoretic criterion. These characterizations are reminiscent of the classical support-graph criteria for transportation polytopes, but they arise in a richer class of polytopes in which vertex structure depends not only on support acyclicity but also on additional geometric compatibility data. Using the collapsing method from simplicial topology, we transfer the dipole characterization to rose graphs and obtain analogous results there. Finally, we apply collapsing to complete bipartite graphs, which encode physically relevant bipartite Bell scenarios, and more generally to arbitrary connected graphs, yielding lower bounds on the number of vertices.
\end{abstract}

\tableofcontents

\section{Introduction}

The vertex structure of convex polytopes is a central object of study in combinatorics and convex geometry. When a polytope is presented as the intersection of an affine subspace with the nonnegative orthant—that is, as a \emph{polytope in standard form}—its vertices are the extreme points of a linear feasibility region, and their characterization arises naturally in combinatorial optimization. Our motivation for studying such polytopes comes from probability polytopes that model measurement statistics in quantum mechanics, where they appear as non-signaling polytopes in quantum foundations \cite{pitowsky1989quantum,barrett2005nonlocal}. At the same time, this class includes important families such as transportation polytopes \cite{de2013combinatorics} and marginal polytopes in graphical models \cite{wainwright2008graphical}.

The principal class of examples considered in this paper is that of
\emph{simplicial distribution polytopes}~\cite{okay2022simplicial}. These are a recently introduced family of probability polytopes that generalize the non-signaling framework, traditionally modeled using presheaves of distributions~\cite{abramsky2011sheaf}, to distributions on measurement spaces modeled by simplicial sets, which are combinatorial models for spaces in modern homotopy theory~\cite{goerss2009simplicial}. In this setting, a measurement scenario is encoded by a simplicial set $X$ together with an outcome space $Y$, and the polytope $\sDist(X,Y)$ consists of probability assignments to the simplices of $X$, where each assignment is a distribution on the simplices of $Y$ and is compatible with the face and degeneracy maps of the simplicial structures. In the one-dimensional case, these polytopes recover the local marginal polytopes of graphical models~\cite{wainwright2008graphical}. To illustrate the generality of this family, every rational polytope can be represented by a local marginal polytope~\cite{prusa2013universality}.  
More general variants include twisted distributions whose support is modified by a cohomology class~\cite{okay2024twisted} and bundle distributions with varying outcome spaces~\cite{barbosa2023bundle}.
An important class of examples of the former arises in polyhedral classical simulation of quantum computation; see~\cite{okay2025polyhedral} for a recent survey.

A particularly structured source of polytopes in standard form is given by \emph{fiber products}. Given polytopes $L^{(1)},\dots,L^{(n)}$ and affine maps $f_i\colon L^{(i)}\to M$ to a common base polytope $M$, the fiber product
\[
L^{(1)}\times_M\cdots\times_M L^{(n)}
\]
is the subpolytope of the Cartesian product consisting of tuples $(x^{(1)},\dots,x^{(n)})$ satisfying the compatibility conditions
\[
f_1(x^{(1)})=\cdots=f_n(x^{(n)}).
\]
The vertex structure of such a polytope is governed by the interaction between the vertex structures of the factors $L^{(i)}$ and the geometry of the maps $f_i$, and in many examples of interest this interaction is reflected in the affine independence of certain images and the geometry of their convex hulls. This perspective arises naturally for simplicial distribution polytopes: when a measurement space $X$ is obtained as a colimit of simpler spaces, in particular by gluing spaces along a common subspace, the polytope $\sDist(X,Y)$ is naturally realized as a fiber product over the simplicial distribution polytope of the common subspace. The vertices of these polytopes represent the extremal probabilistic models of the measurement scenario and include both classical (deterministic) and genuinely nonclassical (contextual) distributions. 
Characterizing such contextual vertices is therefore a natural and important problem not only in quantum foundations \cite{barrett2005nonlocal,abramsky2016possibilities,kharoof2024extremal,kharoof2026geometry}, but also in quantum information, where contextuality has been identified as a computational resource \cite{raussendorf2009contextuality,
howard2014contextuality,bermejo2017contextuality,
budroni2022kochen}.

This paper develops tools to detect and characterize vertices of polytopes in standard form, with particular emphasis on fiber products and on the simplicial distribution polytopes for one-dimensional measurment spaces, where we obtain sharp combinatorial characterizations. Our main technical tool is a preorder relation $\preceq$ on a polytope $L$ in standard form, defined by
\[
y \preceq x \quad\Longleftrightarrow\quad \operatorname{supp}(y)\subseteq \operatorname{supp}(x),
\]
where $\operatorname{supp}(x)=\{i:\;x_i\neq 0\}$ denotes the support of $x$. For each $x\in L$, the set of vertices below $x$ in this preorder is called the \emph{vertex support} of $x$, and is denoted by $\Vsupp(x)$. In the fiber product
$
L^{(1)}\times_M\cdots\times_M L^{(n)},
$
our main sufficient condition (Theorem~\ref{thm:VertofGluingmain}) for a point $x=(x^{(1)},\dots,x^{(n)})$ to be a vertex is formulated in terms of the sets
\[
A_i=\{f_i(y):\;y\in \Vsupp(x^{(i)})\}.
\]  

\begin{thm*} 
Let $x=(x^{(1)},\dots,x^{(n)})$ be a point in the fiber product
$
L^{(1)} \times_M \cdots \times_M L^{(n)}$.
If, for every $1 \leq i \leq n$, the map $f_i$ is injective on $\Vsupp(x^{(i)})$ and the set $A_i$ is affinely independent, and if the intersection $\bigcap_{i=1}^n \Conv(A_i)$ of their convex hulls 
consists of a single point, then $x$ is a vertex.
\end{thm*}

We also establish partial converses: if $x$ is a vertex and
$\Vsupp(x^{(j)})$ is affinely independent for some $j$, then $A_j$ is
affinely independent (Proposition~\ref{pro:affinXimpliesaffinWmain}); and
if $x$ is a vertex, then $\bigcap_{i=1}^n\Conv(A_i)$ is a single point
with all affine coefficients nonzero
(Proposition~\ref{pro:VertofGluihalfsecondirmain}).
The theorem above comes from a more general result on inverse limits (Proposition \ref{pro:Mainresultpoly}).

These results translate naturally to simplicial distributions. If a measurement space is obtained by gluing simplicial sets along a common simplicial subset,
\[
X = X^{(1)}\cup_W\cdots\cup_W X^{(n)},
\]
then the simplicial distribution polytope is the corresponding fiber product
\[
\sDist(X,Y) = \sDist(X^{(1)},Y) \times_{\sDist(W,Y)} \cdots \times_{\sDist(W,Y)}
\sDist(X^{(n)},Y).
\]
In this setting, for a simplicial distribution $p\in \sDist(X,Y)$, the sets $A_i$ are given by the restrictions to $W$ of the vertices in the vertex support of the restricted distribution $p|_{X^{(i)}}$. If each $A_i$ is affinely independent and the convex hulls $\Conv(A_i)$ have a unique common point, then $p$ is a vertex (Corollary~\ref{cor:VertofGluing}). A more general vertex characterization holds for arbitrary colimits of simplicial sets (Theorem~\ref{thm:GenGluing}), of which the gluing case above is a special instance. 

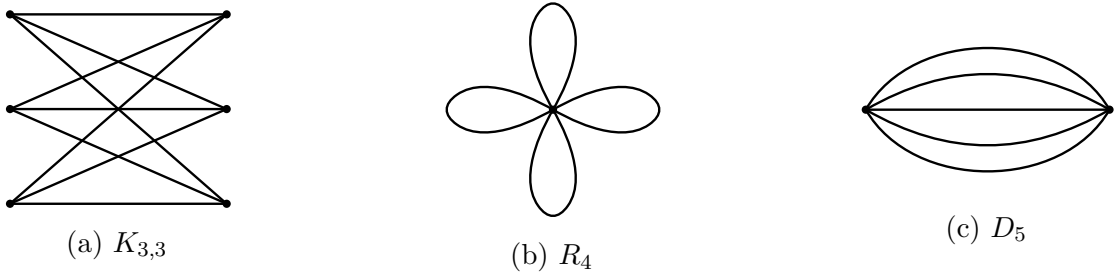
\begin{figure}[ht]
\centering

\begin{minipage}{0.3\textwidth}
\centering
\begin{tikzpicture}[x=0.9cm,y=0.9cm]
  \coordinate (L1) at (0,  1.4);
  \coordinate (L2) at (0,  0);
  \coordinate (L3) at (0, -1.4);
  \coordinate (R1) at (3.2,  1.4);
  \coordinate (R2) at (3.2,  0);
  \coordinate (R3) at (3.2, -1.4);

  \foreach \i in {1,2,3}{
    \foreach \j in {1,2,3}{
      \draw[line width=0.9pt] (L\i)--(R\j);
    }
  }

  \foreach \v in {L1,L2,L3,R1,R2,R3}{
    \node[circle,fill=black,inner sep=1.1pt] at (\v) {};
  }
\end{tikzpicture}

\smallskip
(a) \(K_{3,3}\)
\end{minipage}
\hfill
\begin{minipage}{0.3\textwidth}
\centering
\begin{tikzpicture}[x=0.9cm,y=0.9cm]
  \coordinate (O) at (0,0);

  % top petal
  \draw[line width=0.9pt]
    (O) .. controls (-0.08,0.14) and (-0.55,0.95) .. (-0.22,1.42)
        .. controls (-0.08,1.62) and ( 0.08,1.62) .. ( 0.22,1.42)
        .. controls ( 0.55,0.95) and ( 0.08,0.14) .. (O);

  % right petal
  \draw[line width=0.9pt]
    (O) .. controls (0.14, 0.08) and (0.95, 0.55) .. (1.42, 0.22)
        .. controls (1.62, 0.08) and (1.62,-0.08) .. (1.42,-0.22)
        .. controls (0.95,-0.55) and (0.14,-0.08) .. (O);

  % bottom petal
  \draw[line width=0.9pt]
    (O) .. controls ( 0.08,-0.14) and ( 0.55,-0.95) .. ( 0.22,-1.42)
        .. controls ( 0.08,-1.62) and (-0.08,-1.62) .. (-0.22,-1.42)
        .. controls (-0.55,-0.95) and (-0.08,-0.14) .. (O);

  % left petal
  \draw[line width=0.9pt]
    (O) .. controls (-0.14,-0.08) and (-0.95,-0.55) .. (-1.42,-0.22)
        .. controls (-1.62,-0.08) and (-1.62, 0.08) .. (-1.42, 0.22)
        .. controls (-0.95, 0.55) and (-0.14, 0.08) .. (O);

  \node[circle,fill=black,inner sep=1.1pt] at (O) {};
\end{tikzpicture}

\smallskip
(b) \(R_4\)
\end{minipage}
\hfill
\begin{minipage}{0.3\textwidth}
\centering
\begin{tikzpicture}[x=0.9cm,y=0.9cm]
  \coordinate (A) at (0,0);
  \coordinate (B) at (3.6,0);

  \draw[line width=0.9pt] (A) to[bend left=60] (B);
  \draw[line width=0.9pt] (A) to[bend left=30] (B);
  \draw[line width=0.9pt] (A) -- (B);
  \draw[line width=0.9pt] (A) to[bend right=30] (B);
  \draw[line width=0.9pt] (A) to[bend right=60] (B);

  \node[circle,fill=black,inner sep=1.1pt] at (A) {};
  \node[circle,fill=black,inner sep=1.1pt] at (B) {};
\end{tikzpicture}

\smallskip
(c) \(D_5\)
\end{minipage}

\caption{Basic graph families: the complete bipartite graph \(K_{3,3}\), the rose graph \(R_4\), and the dipole graph \(D_5\).}
\label{fig:intro-graphs}
\end{figure}

Our sharpest results concern the one-dimensional case, where the measurement space is a directed multigraph. The outcome space is $\Delta_{\ZZ_m}$, the natural simplicial model for measurements with outcomes in $\ZZ_m=\{0,1,\cdots,m-1\}$. As a simplicial set, it may be viewed as an unoriented version of the standard $(m-1)$-simplex. This case is of particular interest because simplicial distributions coincide with \emph{distributions on graphs}. A simplicial distribution $p\in \sDist(X,\Delta_{\ZZ_m})$, where $X$ is a one-dimensional simplicial set, that is, a graph, consists of a family of probability matrices associated to the edges, whose row and column sums agree at common vertices. More precisely, if $p_\sigma$ is the probability matrix associated to an edge $\sigma$ with source vertex $x$ and target vertex $y$, then the probability vectors at $x$ and $y$ are given by the row sums and column sums of $p_\sigma$, respectively. We refer to such an assignment as a graph distribution. Writing $\Dist(X,m)$ for graph distributions on $X$ with $m$ outcomes, we obtain a natural isomorphism
\[
\Dist(X,m) \cong \sDist(X,\Delta_{\ZZ_m}).
\]
These graph distribution polytopes are also known as local marginal polytopes in the graphical models literature \cite{wainwright2008graphical}.
They are closely related to transportation polytopes: for a graph consisting of a single edge, fixing the endpoint distributions gives precisely a transportation polytope \cite{de2013combinatorics}; from this perspective,  distributions on a general graph may be viewed as compatible families of such edgewise transportation polytopes coupled through their common vertex marginals.
In this paper, we focus on three classes of graphs: \emph{complete bipartite graphs}, \emph{rose graphs}, and \emph{dipole graphs}; see Figure~\ref{fig:intro-graphs}. Among these, dipole graphs turn out to have the most tractable vertex structure, and we obtain a complete characterization of their vertices (Theorem~\ref{thm:afinoneintersect}).

Let $\tau_1,\dots,\tau_n$ denote the edges of the dipole graph. For a vertex $p \in \sDist(D_n, m)$, the associated set obtained by restricting the vertices in the vertex support of $p|_{\tau_i}$ to the boundary $\Delta^{0}\sqcup \Delta^{0}$ is
\begin{equation*}
A_i
=\set{(e_a,e_b)^T:\; p^{ab}_{\tau_i}\neq 0}.
\end{equation*}
The pairs $(e_a,e_b)$ of canonical basis vectors of $\RR^m$ will be referred to as \emph{product-simplex vertices} in $\RR^{2m}$.
Then our fiber product theorem gives immediately full vertex characterization for the dipole graph (Theorem \ref{thm:afinoneintersect}).
This characterization admits a further graph-theoretic refinement (Theorem~\ref{thm:bybipartgraphs}). To each set $A_k$ we associate a bipartite graph $H_k$ with left vertex set $\{u_0,\dots,u_{m-1}\}$ and right vertex set $\{w_0,\dots,w_{m-1}\}$, where each product-simplex vertex $(e_i,e_j)^T\in A_k$ gives an edge between $u_i$ and $w_j$. This is closely related to the classical support-graph description of vertices of transportation polytopes, where acyclicity of the support graph characterizes extremality \cite{bolker1976simplicial,brualdi2006combinatorial}. Our criterion may be viewed as a refinement of this picture adapted to the present setting. In addition to the acyclicity of the graphs $H_k$, one must also keep track of how their connected components interact across different edges of the dipole. This information is encoded in a matrix $Q(H_1,\dots,H_n)$ assembled from the connected components of the graphs $H_k$: its entries are $0$ or $\pm 1$, according to whether a given vertex belongs to a given connected component and, if so, according to whether it lies on the left or right side of the bipartition.

\begin{thm*} 
Let $p$ be a distribution on $D_n$ with $m$ outcomes. 
Then $p$ is a vertex if and only if the following conditions are satisfied:
\begin{itemize}
    \item for every $1 \leq i \leq n$, the graph $H_i$ contains no cycle;
    \item the rank of $Q(H_1,\dots,H_n)$ is equal to $2m-1$.
\end{itemize}
\end{thm*}

A direct application of the geometric vertex criteria in terms of the sets $A_i$ yields only partial results for the rose graph. A complete characterization, however, can be obtained using the \emph{collapsing} technique introduced in \cite{kharoof2023homotopical}. The key idea is that the vertex structure behaves well under collapses of simplices. 
%More precisely, such collapses are realized as pushouts along degeneracy maps. 
In particular, the dipole graph $D_{n+1}$ collapses onto the rose graph $R_n$ by collapsing one of its edges to a single vertex. This allows us to obtain a similar graph-theoretic criterion (Corollary~\ref{cor:characgraphs2}). Finally, we apply the collapsing method to analyze the various ways in which complete bipartite graphs collapse onto dipole and rose graphs, yielding lower bounds on the number of vertices (Corollary~\ref{cor:bipartiwedge} and Proposition~\ref{pro:bipartiedges}). Since any connected graph can be collapsed onto a rose graph by choosing a maximal spanning tree, our methods also provide lower bounds on the number of vertices for arbitrary graph distribution polytopes.

The structure of our paper is as follows.
Section~\ref{sec:GeoPoly} develops the polytope-theoretic framework:
we introduce the preorder, establish its basic properties, prove
the vertex characterization for limits of diagrams, and specialize to
fiber products.
Section~\ref{sec:Exsimpldist} reviews simplicial distributions and
contextuality, establishes that simplicial distribution polytopes are in
standard form, and translates the results of
Section~\ref{sec:GeoPoly} to measurement spaces that arise as colimits,
particularly those obtained by gluing along a common subspace.
Section~\ref{sec:distgraphs} applies these tools to complete bipartite
graphs, rose graphs, and dipole graphs, obtaining vertex criteria and
explicit constructions in terms of product-simplex vertices.
Section~\ref{sec:graphcharac} reformulates these criteria in
graph-theoretic and linear-algebraic terms, culminating in the acyclicity
and rank characterizations.
Section~\ref{sec:1dim} applies the collapsing method to connected
one-dimensional measurement spaces, deriving lower bounds on contextual
vertices with emphasis on complete bipartite graphs.

\paragraph{Acknowledgments.}
This work was supported by the Air Force Office of Scientific Research (AFOSR) under award number FA9550-24-1-0257.
The second author also acknowledges support from the Horizon Europe project FoQaCiA (Grant Agreement No.~101070558).
We thank Atak Talay Y\"ucel for providing computations for the one-dimensional simplicial distributions appearing in Section~\ref{sec:1dim}, and Selman Ipek for enlightening discussions.

\section{Vertices of polytopes in standard form}\label{sec:GeoPoly}
 
%{In this section, we investigate structural properties of polytopes in standard form with respect to a natural preorder relation, culminating in a characterization of the vertices of polytopes that arise as limits of diagrams of polytopes in standard form. We then apply these results to the case of fiber products. This leads to our main result of the section, which provides a sufficient geometric condition for identifying vertices of fiber product polytopes.}

In this section, we investigate polytopes in standard form. We introduce a natural preorder relation on such polytopes. This preorder leads to a
complete characterization of vertices of polytopes arising as limits of
diagrams. We then apply this to fiber products, where the main result of
this section---Theorem \ref{thm:VertofGluingmain}---provides a
sufficient geometric condition for identifying vertices.

\subsection{Polytopes in standard form}

%{In this section, we define polytopes in standard form, with polytopes of graph distributions serving as our main examples. We then introduce a preorder relation on such polytopes and establish several properties related to this preorder.}

In this subsection, we introduce the preorder relation and convex-geometric properties that will be used throughout the paper. Our main examples are graph distributions, viewed as special cases of simplicial distributions to be introduced in later sections.

\begin{defn}\label{def:polystan}
A \emph{(convex) polytope} in the Euclidean space $\RR^n$ is a bounded set of the form
\[
L = \{\, x \in \RR^n :\; A x \le b \,\},
\]
where $A$ is a real $m \times n$ matrix and $b \in \RR^m$. That is, a polytope is a bounded intersection of finitely many closed half-spaces in $\RR^n$.

A polytope $L \subset \RR^n$ is said to be in \emph{standard form} if
$$
L = \{\, x \in \RR^n :\; Ax = b \text{ and } x_i \ge 0 \text{ for all } i \,\},
$$
where $A$ is a real $m \times n$ matrix and $b \in \RR^m$, and where $L$ is bounded.
Equivalently, a polytope in standard form is a bounded intersection of an affine subspace 
with the nonnegative orthant $\RR^n_{\ge 0}$. 
\end{defn}

It is a standard fact that every polytope is affinely isomorphic to a polytope in standard form (see, e.g., \cite[Section 1.1]{bertsimas1997introduction}). Writing $a_i$ for the $i$\textsuperscript{th} row of $A$, we have that $a_i \cdot x = b_i$ for all $i$. We introduce two variables $x^+_j$ and $x^-_j$ for each variable, and a variable $s_i$ for each equation. Then we define the polytope in standard form
\[
L' = \{\, (x^+, x^-, s) \in \RR^{2n + m} :\; a_i \cdot (x^+ - x^-) + s_i = b_i \text{ for all } i, \text{ and } x^+_j, x^-_j, s_i \ge 0 \text{ for all } i, j \,\}.
\]
Under the linear map $f:\RR^{2n+m} \to \RR^{n}$ defined by $f(x^+, x^-, s) = x^+ - x^-$, the polytope $L'$ is mapped bijectively onto $L$. Throughout this work we focus on polytopes in standard form. Accordingly, although some of our results extend to arbitrary polytopes, for clarity we assume that all polytopes are in standard form unless stated otherwise.

Our goal is to study the vertex structure of polytopes in standard form.  
We begin by recalling key notions from convex geometry that will be used throughout this work.

\begin{defn}\label{def:polytope-notions}
Let $L \subseteq \RR^n$ be a polytope.
\begin{enumerate}
    \item A hyperplane $H = \{x \in \RR^n :\; c^\top x = \alpha\}$ is called \emph{supporting} for $L$ if $c^\top x \le \alpha$ for all $x \in L$ and $H \cap L \neq \varnothing$.
    
    \item A \emph{face} of $L$ is a subset of the form $F = L \cap H$ for some supporting hyperplane $H$. We allow the \emph{improper faces} $F = \varnothing$ and $F = L$. All other faces are called \emph{proper}.
    
    \item A \emph{vertex} of $L$ is a $0$-dimensional face, i.e., a point $v \in L$ that cannot be written as a nontrivial convex combination of two distinct points in $L$. Equivalently, $v$ is a vertex if whenever $v = tb + (1-t)c$ for some $b, c \in L$ and $0 < t < 1$, then necessarily $b = c = v$.
    
    \item An \emph{edge} is a $1$-dimensional face of $L$.
    
    \item A \emph{facet} is a face of codimension $1$, i.e., of dimension $\dim(L) - 1$.
\end{enumerate}
\end{defn}

Our main source of examples comes from simplicial distributions to be defined in the next section. Instead, for now we introduce the one-dimensional version which corresponds to distributions on graphs. 

\begin{defn}\label{def:multigraph}
Let $X$ be a directed (multi)graph with vertex set $X_0$ and edge set $X_1$ together with source and target maps $d_1, d_0 \colon X_1 \to X_0$. A \emph{graph distribution} {$p$} on $X$ with outcomes in $\ZZ_m{=\set{0,1,\dots,m-1}}$ consists of 
    \begin{itemize}
        \item a probability vector 
        \[
        p_x =   \begin{pmatrix}
    p^0_x \\
    \vdots\\
    p^{m-1}_x
    \end{pmatrix}, \qquad x\in X_0,
        \]
        \item a probability matrix
        \[
        p_\sigma = \begin{pmatrix}
    p_{\sigma}^{0,0} \;\dots \dots \; p_{\sigma}^{0,m-1}\\
    \vdots \;\;\;\;\vdots \;\;\;\; \vdots \\
    p_{\sigma}^{m-1,0} \; \dots \dots \; p_{\sigma}^{m-1,m-1}
    \end{pmatrix}, \qquad \sigma\in X_1.
        \]
    \end{itemize}
such that 
\[
  p_{d_1(\sigma)}=\begin{pmatrix}
    p_{\sigma}^{0,0} +\dots+ p_{\sigma}^{0,m-1}\\
    \vdots  \\
    p_{\sigma}^{m-1,0} + \dots + p_{\sigma}^{m-1,m-1}
    \end{pmatrix}, \qquad
    p_{d_0(\sigma)}= \begin{pmatrix}
    p_{\sigma}^{0,0} +\dots + p_{\sigma}^{m-1,0}\\
    \vdots \\
    p_{\sigma}^{0,m-1} + \dots + p_{\sigma}^{m-1,m-1}
    \end{pmatrix}.
\]
\end{defn}
We will write $\on{Dist}(X{,m})$ for the polytope of {graph} distribution{s} on $X$ {with outcomes in $\ZZ_m$}.

A distribution is called \emph{deterministic} if all the probability matrices are deterministic, i.e., each matrix has exactly one entry equal to $1$ and all other entries equal to $0$. Deterministic distributions are vertices of $\on{Dist}(X,m)$. 

\begin{rem}
In the literature graph distributions are also called the \emph{local marginal polytope} \cite{wainwright2008graphical}. In quantum foundations they are particular examples of \emph{non-signaling polytopes} \cite{barrett2005nonlocal}. This class of polytopes has been generalized to simplicial distributions in {an} earlier work of the authors \cite{okay2022simplicial}.    
\end{rem}

\begin{ex}\label{ex:cycle}
A famous example the $n$-cycle graph $C^{(n)}$ consisting of the edges $\sigma_1,\dots,\sigma_n$ 
satisfying
\[
d_0(\sigma_i) = d_1(\sigma_{i+1}), \quad 1 \leq i \leq n-1, \qquad d_0(\sigma_n) = d_1(\sigma_1).
\]
In \cite{kharoof2024extremal}, the authors proved that the vertices of $\on{sDist}(C^{(n)},m)$ are given by the \emph{$k$-order cycle distributions} specified by a tuples $ (a_1^{(1)},\cdots,a_n^{(1)};a_1^{(2)},\cdots,a_n^{(2)};\cdots;
a_1^{(k)},\cdots,a_n^{(k)})$ of elements in $\ZZ_m$ such that $a_i^{(j)}\neq a_i^{(s)}$ for every $1\leq i \leq n$, $j\neq s$. The corresponding distribution is defined by
$$
p_{\sigma_i}^{a,b} = \left\lbrace
\begin{array}{ll}
\frac{1}{k} & (a,b)=(a_i^{(j)},a_{i+1}^{(j)})\; \, \text{for some $1\leq j \leq k$} \\  
0 & \text{otherwise,}
\end{array}
\right.
$$
for $1\leq i \leq n-1$, and
$$
p_{\sigma_n}^{a,b} = \left\lbrace
\begin{array}{ll}
\frac{1}{k} & (a,b)=(a_n^{(j)},a_{1}^{(j+1)}) \; \, \text{for some $1\leq j \leq k-1$}\\  
\frac{1}{k} & (a,b)=(a_n^{(k)},a_{1}^{(1)})\\  
0 & \text{otherwise.}
\end{array}
\right.
$$
The $1$-order cycle distributions are precisely the deterministic distributions. The $4$-cycle case {with two outcomes} is also known as the Clauser-Horne-Shimony-Holt (CHSH) scenario in quantum foundations. {In this scenario,} the $2$-order cycle distributions are known as the \emph{Popescu-Rohrlich (PR) boxes} \cite{popescu1994quantum}. For instance, a PR box is given by  
\begin{equation}\label{eq:p +-}
p_{\sigma_1} = p_+:= \begin{pmatrix}
\frac{1}{2} & 0 \\
0 & \frac{1}{2} 
\end{pmatrix},
\quad
p_{\sigma_2} = p_{\sigma_3} = p_{\sigma_4} = p_-:=
\begin{pmatrix}
 0 & \frac{1}{2} \\
\frac{1}{2} & 0
\end{pmatrix}. 
\end{equation}    
%\coc{We switch to $i,j$ from $a,b$.}
For us the case $n=1$ will be of particular interest. In this case, we will write 
\[
\on{cyc}(m) = \{ [i_1,\dots,i_k] :\; i_1,\dots,i_k \text{ are distinct numbers in } \ZZ_m,\; 1\leq k\leq m\}.
\]
For $\mu \in \on{cyc}(m)$, the distribution $p(\mu)$ is given by
\[
p(\mu)_\sigma^{a,b} = \begin{cases}
    \frac{1}{|\mu|} & \mu(a)=b,\\
    0 & \text{otherwise.}
\end{cases}
\]
For example, with this notation we can describe some of the vertices of $\on{Dist}(C^{(1)},4)$ as follows:
\[
\begin{aligned}
[0]=\begin{pmatrix}%
1& 0 & 0 & 0 \\
0 & 0 & 0 & 0 \\
0 & 0 & 0 & 0 \\
0 & 0 & 0 & 0 
\end{pmatrix},&\qquad
[0,1]=\begin{pmatrix}%
0& \frac{1}{2} & 0 & 0 \\
\frac{1}{2} & 0 & 0 & 0 \\
0 & 0 & 0 & 0 \\
0 & 0 & 0 & 0 
\end{pmatrix},\\
[0,1,2]=\begin{pmatrix}%
0 & \frac{1}{3} & 0 & 0 \\
0 & 0 & \frac{1}{3} & 0 \\
\frac{1}{3} & 0 & 0 & 0 \\
0 & 0 & 0 & 0 %\begin{cihan}{purple}{Move to section intro. Say something like:
%\medskip

%Our results in this section generalize the polytope theoretic observations in \cite{kharoof2024extremal} to polytopes in standard form.  
%}
%For simplicial distributions the relation coincides with the one given in \cite[Definition 2.3]{kharoof2024extremal}.
%\end{cihan}

%\begin{cihan}{purple}{remove, merged}
%\begin{defn}\label{def:Vsupp}
%{\rm
%Let $L$ be a polytope and let $x \in L$. We define the \emph{vertex support of $x$}, denoted by $\Vsupp(x)$, to be the set of vertices $y \in L$ such that $y \preceq x$.
%}
%\end{defn}
% \end{cihan}

\end{pmatrix},&\qquad
[0,1,2,3]=\begin{pmatrix}%
0 & \frac{1}{4} & 0 & 0 \\
0 & 0 & \frac{1}{4} & 0 \\
0 & 0 & 0 & \frac{1}{4} \\
\frac{1}{4} & 0 & 0 & 0 
\end{pmatrix}.
\end{aligned}
\]
\end{ex}

\begin{defn}\label{def:GenRelation}
Let $L$ be a polytope. For $x \in L$, define
$$
\operatorname{supp}(x) := \{\, i :\; x_i \neq 0 \,\}.
$$
For $x, y \in L$, we write ${y} \preceq {x}$ if
$$
\operatorname{supp}({y}) \subseteq \operatorname{supp}({x}).
$$
 We define the \emph{vertex support of $x$}, denoted by $\Vsupp(x)$, to be the set of vertices $y \in L$ such that $y \preceq x$.
\end{defn}

This relation $\preceq$ is a preorder on the polytope $L$.

For a subset $X \subseteq \RR^n$, we write $\Conv(X)$ for the \emph{convex hull of $X$}, the set of all convex combinations of elements of $X$:
\[
\Conv(X) = \left\{\sum_{i=1}^k \alpha_i x_i :\; k \in \NN, x_i \in X, \alpha_i \geq 0, \sum_{i=1}^k \alpha_i = 1\right\}.
\]

\begin{pro}\label{pro:yleqxVsupp}
Given a polytope $L$ and $x \in L$, we have 
$$
\Conv(\Vsupp({x}))=\set{y\in L :\; \; {y} \preceq {x} }.
$$
\end{pro}
\begin{proof}
Suppose that ${y} \preceq {x}$. Since the relation $\preceq$ is transitive, we have 
$\Vsupp({y}) \subseteq \Vsupp({x})$. On the other hand, clearly ${y} \in \Conv(\Vsupp({y}))$. Therefore, 
${y} \in \Conv(\Vsupp({x}))$.  

Now suppose that $y \in \Conv(\Vsupp(x))$, and write 
$
\Vsupp(x)=\set{x^{(1)},\dots,x^{(k)}}
$.
Then there exist coefficients $0 \leq \alpha_1,\dots,\alpha_k \leq 1$ with $\sum_{i=1}^k \alpha_i=1$ such that 
\begin{equation}\label{eq:y=alphax}
{y}=\alpha_1 {x}^{(1)} + \dots +\alpha_k {x}^{(k)}.
\end{equation}
If $y_i \neq 0$, then by Equation~\eqref{eq:y=alphax}, there exists some $1\leq j \leq k$ 
such that ${x}^{(j)}_i \neq 0$. By the definition of
$\Vsupp(x)$, this implies that $x_i \neq 0$. Hence ${y}\preceq {x}$.
\end{proof}
%

%\begin{rem}\label{rem:carrier}

We now give another characterization of the convex hull of the vertex support. This relies on the notion of carrier face of a point in a polytope.

%\coc{earlier remark format was also good, but this version better emphasizes the connection, imo}

\begin{defn}\label{def:carrier}
For a subset $X\subseteq \RR^n$, we write $\Aff(X)$ for the \emph{affine hull} of $X$:
\[
\Aff(X) = \left\{\sum_{i=1}^k \alpha_i x_i :\; k \in \NN, x_i \in X, \alpha_i \in \RR, \sum_{i=1}^k \alpha_i = 1\right\}.
\]
The \emph{relative interior} of a polytope $L$, denoted by $\relint(L)$, consists of those points
$x \in L$ such that there exists $\varepsilon > 0$ with the property that
$y \in L$ whenever $y \in \Aff(L)$ and of distance less than $\varepsilon$ from $x$, i.e., 
$d(x,y) \le \varepsilon$. 
For $x$ in a polytope $L$, we define the \emph{carrier face} of $x$, denoted by $\carr(x)$, to be the unique face
$F$ {(see Definition \ref{def:polytope-notions})} such that $x \in \relint(F)$.
\end{defn} 
%Then, it is well-known (see for instance \cite[Proposition 5.1]{abramsky2016possibilities})) that 
%$$
%\carr(x)=\set{y\in L : \; {y} \preceq {x} }.
%$$
%Therefore, combining this with Proposition~\ref{pro:yleqxVsupp}, we obtain  
%\begin{equation}\label{eq:convVsupp=carr}
%\Conv(\Vsupp({x}))=\carr(x).
%\end{equation}

%\end{rem}

\begin{pro}\label{pro:carrier-conv-vsupp}
Let $L$ be a polytope in standard form and let $x\in L$. Then
\begin{equation}\label{eq:convVsupp=carr}
\Conv(\Vsupp({x}))=\carr(x).
\end{equation}
\end{pro}
\begin{proof}
It is well-known (see for instance \cite[Proposition 5.1]{abramsky2016possibilities}) that
\[
\carr(x)=\set{y\in L : \; y \preceq x }.
\]
Therefore, combining this with Proposition~\ref{pro:yleqxVsupp}, we obtain the result.
\end{proof}

The key consequence of this observation is the following lemma characterizing the preorder relation in terms of convex decompositions.

\begin{lem}\label{lem:xleqythen}
For $x,y \in L$, where $L$ is a polytope in standard form, the following are equivalent:
\begin{enumerate}
    \item ${y}\preceq {x}$;
    \item there exist $0 < \alpha \leq 1$ and ${z} \in L$ such that
    ${x}=\alpha {y}+ (1-\alpha) {z}$.
\end{enumerate}
\end{lem}
\begin{proof}
$(2)\Rightarrow (1)$:
If $x=\alpha y+(1-\alpha)z$ with $0<\alpha\leq 1$ and $z\in L$, then for every index $i$ with $y_i>0$ we have $x_i=\alpha y_i+(1-\alpha)z_i\geq \alpha y_i>0$. Hence $\on{supp}(y)\subseteq \on{supp}(x)$, that is, $y\preceq x$.

$(1)\Rightarrow (2)$: The case $x=y$ is immediate, so assume $x\neq y$ and $y\preceq x$. 
%Remark \ref{rem:carrier} 
Proposition \ref{pro:carrier-conv-vsupp}
implies that
$y \in \carr(x)$. In addition, we have
$x \in \relint(\carr(x))$, so by \cite[Theorem 6.4]{rockafellar2015convex} there exists $\mu>1$ such that $z:=\mu x + (1-\mu)y$ belongs to $\carr(x)$.
Setting $\alpha = 1-1/\mu$ gives the desired convex decomposition in (2) with $\alpha \in (0,1)$. A more direct argument can be given as follows.
Since $x\neq y$, we necessarily have $y\neq 0$. ($y=0$ implies $L=\{0\}$.) Then the support $\on{supp}(y)$
is nonempty. Define
\[
\alpha=\min\left\{\frac{x_i}{y_i}: i\in \on{supp}(y)\right\}.
\]
Because $y\preceq x$, each ratio $x_i/y_i$ is positive, so $\alpha>0$.
We claim that $\alpha<1$. Suppose that $\alpha\geq 1$. Then for every $i\in \on{supp}(y)$,
$x_i\geq y_i$,
and for $i\notin \on{supp}(y)$ we have $y_i=0$, so $x_i-y_i=x_i\geq 0$. Thus $x-y\geq 0$, and since
$x\neq y$, we have $x-y\neq 0$. On the other hand,
$A(x-y)=Ax-Ay=b-b=0$.
Hence $x-y$ is a nonzero vector in $\ker(A)\cap \mathbb{R}_{\geq 0}^n$, which would make
$x+t(x-y)\in L$ for all $t\geq 0$,
contradicting the boundedness of $L$. Therefore $\alpha<1$.

Now define
\[
z:=\frac{x-\alpha y}{1-\alpha}.
\]
For each $i\notin \on{supp}(y)$, we have $y_i=0$, so
$z_i=\frac{x_i}{1-\alpha}\geq 0$.
For each $i\in \on{supp}(y)$, the definition of $\alpha$ gives
$\frac{x_i}{y_i}\geq \alpha$,
hence $x_i-\alpha y_i\geq 0$, so again $z_i\geq 0$. Therefore $z\geq 0$. Also,
\[
Az=\frac{Ax-\alpha Ay}{1-\alpha}
=\frac{b-\alpha b}{1-\alpha}
=b.
\]
Thus $z\in L$, and by construction
$x=\alpha y+(1-\alpha)z$.
\end{proof}

An immediate consequence of this result is the following characterization of being a vertex in terms the preorder.

\begin{cor}\label{cor:vert=min}
A point $x\in L$ a vertex if and only if it is minimal with respect to the preorder $\preceq$.        
\end{cor}

The following consequence will be used later.
 
\begin{cor}\label{cor:fVsupp}
Any affine map $f \colon L_1 \to L_2$ preserves the preorder $\preceq$. In particular, for ${x}\in L_1$ and ${y}\in\Conv\big(\Vsupp({x})\big)$, we have
\[
f({y})\in\Conv\big(\Vsupp(f({x}))\big).
\]
\end{cor}
\begin{proof}
Let ${x},{y} \in L_1$ with ${y} \preceq {x}$. By Lemma \ref{lem:xleqythen}, there exist $0<\alpha\leq1$ and ${z} \in L_1$ such that 
${x} = \alpha {y} + (1-\alpha){z}$. Then,
%By the convexity of $f$, it follows that
\[
f({x}) = f\bigl(\alpha {y} + (1-\alpha){z}\bigr) 
= \alpha f({y}) + (1-\alpha)f({z}).
\]
Applying Lemma \ref{lem:xleqythen} again, we obtain that $f({y}) \preceq f({x})$. The second statement comes from an application of Proposition~\ref{pro:yleqxVsupp} to the pair
$f(y), f(x)$. This implies $f(y) \in \Conv\big(\Vsupp(f(x))\big)$.
\end{proof}

Another important consequence of 
%Remark \ref{rem:carrier} 
Proposition \ref{pro:carrier-conv-vsupp}
is about sets that generate faces of polytopes.

\begin{defn}\label{def:closeVert}
A set $\{x^{(1)},\dots,x^{(k)}\}$ of vertices of a polytope $L$ is
said to \emph{generate a face}
if its convex hull $\Conv(\{x^{(1)},\dots,x^{(k)}\})$ is a face of $L$. 
\end{defn}

An example of a set that generates a face is $\Vsupp(x)$. This follows from the identification of the convex hull of the vertex support with the carrier face in Equation~(\ref{eq:convVsupp=carr}).  
 
%\coc{maybe better as a lemma}
 
\begin{lem}\label{lem:genface}
Let $\{x^{(1)},\dots,x^{(k)}\}$ be a set that generates a face of $L$. 
Then for any choice of coefficients
$\alpha_1,\dots,\alpha_k \in [0,1]$ with $\sum_{i=1}^k \alpha_i = 1$, we have
$$
\Vsupp\!\left(\alpha_1 x^{(1)} + \cdots + \alpha_k x^{(k)}\right)
\subseteq \{x^{(1)},\dots,x^{(k)}\}.
$$    
\end{lem}
\begin{proof}
Let $F=\Conv(\{x^{(1)},\dots,x^{(k)}\})$ be the corresponding face of $L$. We may assume without loss of generality that $\alpha_i>0$ for all
$1\le i\le k$, since any index $i$ with $\alpha_i=0$ can be omitted from the
convex combination without changing the point
$\alpha_1 x^{(1)} + \cdots + \alpha_k x^{(k)}$. Under this assumption, we have 
$$
F=\carr(\alpha_1 x^{(1)} + \cdots + \alpha_k x^{(k)}).
$$
By Equation (\ref{eq:convVsupp=carr}) we have
$$
\Conv\left(\Vsupp(\alpha_1 x^{(1)} + \cdots + \alpha_k x^{(k)})\right)=\carr(\alpha_1 x^{(1)} + \cdots + \alpha_k x^{(k)}).
$$
Hence, if $x\in \Vsupp\!\left(\alpha_1 x^{(1)} + \cdots + \alpha_k x^{(k)}\right)$, then $x \in F$. Since the vertices of $F$ are precisely
$\{x^{(1)},\dots,x^{(k)}\}$, it follows that
$x \in \{x^{(1)},\dots,x^{(k)}\}$.
\end{proof}

\subsection{Vertices of fiber products of polytopes}
\label{subsec:fibpro}

%{In this section, we continue to work exclusively with polytopes in standard form. We adopt a categorical perspective that allows one to detect vertices of such constructions using affine and convex–geometric data associated with the supports of the components. This framework is used to analyze vertex structure via limits of polytopes, and is later specialized to the case of fiber products. This leads to a sufficient geometric condition for identifying their vertices, as well as partial converses that clarify the extent to which this condition is necessary.}

In this subsection, we study the vertex structure of inverse limits of polytopes. Using the preorder introduced above, we give a characterization of vertices. We then specialize to fiber products, where this characterization takes a more concrete form. This leads to a sufficient condition for identifying vertices of fiber products.

Convex polytopes and affine maps between them forms a category. Since our interest is polytopes in standard form, we will work with the full subcategory of $\catPoly$ whose objects are polytopes in standard form. All finite limits exist in this category as they are created in the category of sets.

\begin{defn}\label{def:limit-initial}
A \emph{finite diagram of polytopes} consists of 
a finite category $J$, and a functor 
\begin{equation}\label{eq:diagramchi}
\chi \colon J \longrightarrow \catPoly.
\end{equation} 
Let $J_0 \subseteq J$ denote the set of initial objects of $J$. The \emph{limit} of $\chi$, denoted by $\lim \chi$, is the subpolytope of the Cartesian product
$\prod_{j \in J_0} \chi(j)$
consisting of all tuples $({x}^{(j)})_{j \in J_0}$ such that for every 
$k \in J$ and every morphism $f \colon i \to k$ and $g\colon j \to k$ with $i,j \in J_0$, we have
\[
\chi(f)({x}^{(i)})\;=\; \chi(g)({x}^{(j)}).
\]
\end{defn}

%\begin{rem}
%The definition above is equivalent to the usual categorical definition of a limit, 
%since limits in $\catPoly_\ast$ are the same as in the category in convex sets, which created in $\mathbf{Set}$ (see  \cite[part (i) of Theorem 5.6.5.]{riehl2017category}). 
%In particular, specifying the coordinates of a point in $\lim \chi$ is equivalent 
%to specifying its components at the initial objects of $J$, because all other 
%components are determined uniquely by the morphisms.
%\end{rem}

Given a diagram $\chi$ in $\catPoly$ and a point $x \in \lim \chi$,
we construct a subdiagram of polytopes determined by the vertex supports.
\begin{defn}
Let \(\chi\) be a diagram of polytopes
% \(\catPoly_{\ast}\) as in (\ref{eq:diagramchi}), 
and let \({x}\) be a point in the limit of \(\chi\).
We define a new functor \(\chi_{{x}}\colon J \longrightarrow \catPoly\) 
%to be the diagram obtained from \(\chi\) by replacing each object \(\chi(j)\) in $\chi$ with
by
\[
\chi_x(j)= \Conv\bigl(\Vsupp(\pi^{\chi}_j ({x}))\bigr),
\]
where \(\pi^{\chi}_j \colon \lim \chi \to \chi(j)\) is the canonical projection.
% from the limit to the object \(\chi(j)\).
\end{defn}
%
%
%Given a diagram $\chi$ in $\catPoly_\ast$ as in (\ref{eq:diagramchi}). 
%For every initial object $j \in J_0$, let $\pi^{\chi}_j \colon \lim \chi \to \chi(j)$ be the canonical projection from the limit to the object \(\chi(j)\). Then

Corollary~\ref{cor:fVsupp} guarantees that 
%\(\chi_{{x}}\) 
this construction gives a well-defined functor.
%is well defined. 
Moreover, by Definition~\ref{def:limit-initial}, the limit of this new diagram can be written as
\begin{equation}\label{eq:limchix}
\lim \chi_{{x}} 
= \set{{y} \in \lim \chi :~ \pi^{\chi}_j({y}) \in \Conv\bigl(\Vsupp(\pi^{\chi}_j({x}))\bigr),\;\forall j\in J_0}.
\end{equation}
%
%

%\coc{better as a lemma, used in the theorem}

\begin{lem}\label{lem:lim=pre}
Let $\chi$ be a diagram in $\catPoly$ and let ${x} \in \lim \chi$. Then 
\[
\lim \chi_{{x}}=\set{{y} \in \lim \chi :\; {y} \preceq {x}}.
\]
\end{lem}

\begin{proof}
%Suppose that $\lim \chi=\prod_{j \in J_0} \chi(j)$ and ${x}=({x}^{(j)})_{j \in J_0}$ as in 
%Definition \ref{def:limit-initial}. Take 
%\[
%{y}=({y}^{(j)})_{j \in J_0} \in \lim \chi_{{x}},
%\]
%which means that ${y}^{(j)} \in \Conv(\Vsupp({x}^{(j)}))$ for every $j\in J_0$. By Lemma \ref{lem:yleqxVsupp}, this condition is equivalent to requiring that 
By Equation (\ref{eq:limchix}) and Proposition \ref{pro:yleqxVsupp}, we have  
$$
\lim \chi_{{x}} = \set{{y} \in \lim \chi :\; \pi^{\chi}_j({y}) \preceq \pi^{\chi}_j({x}), \forall j\in J_0}.
$$
Then, by the definition of $\preceq$, the condition $\pi^{\chi}_j({y}) \preceq \pi^{\chi}_j({x}), \forall j\in J_0$ 
is equivalent to ${y}\preceq {x}$.
\end{proof}
%
%
%

%\coc{better to downgrade to a proposition}

%
\begin{pro}\label{pro:Mainresultpoly}
Let $\chi$ be a diagram in $\catPoly$. An element ${x}$ is a vertex of $\lim \chi$ if and only if 
\[
\lim \chi_{{x}}=\set{{x}}.
\]
\end{pro}
\begin{proof}
This follows directly from Corollary \ref{cor:vert=min} and Lemma \ref{lem:lim=pre}.   
\end{proof}

%\subsection{Vertices of fiber product of polytopes in standard form}
%
%
%\coc{section title removed}
%
%In this section, we investigate the geometric structure of fiber products of polytopes in standard form. 
%In this special case of what done in the previous subsection, our focus is on understanding how vertices of the fiber product arise from the interaction of convex hulls and affine relations induced by the maps to the base polytope.

Next, we specialize to the case of fiber products of polytopes in standard form. In this case, the limit of the diagram $\chi_x$ can be described more concretely as a fiber product of convex hulls of vertex supports. This allows us to give a more explicit characterization of the vertices in such fiber products.

% we able to give a sutifficint condition for being a vertex in such fiber products.
%This condition relies on the geometric properties of specific sets in the target polytope defined using the vertex support. 

\begin{defn}\label{def:fiberproduct}
Let $L^{(1)},\dots,L^{(n)}$ and $M$ be polytopes, 
%in standard form, 
and let
$$
f_i \colon L^{(i)} \longrightarrow M \qquad (i=1,\dots,n)
$$
be affine maps.  
The \emph{fiber product} of $L^{(1)},\dots,L^{(n)}$ over $M$, denoted by
$
L^{(1)} \times_M \cdots \times_M L^{(n)}
$,
is the limit of the diagram
\begin{equation}\label{eq:fibeq}
\begin{tikzcd}
L^{(1)} \arrow[drr,"f_1"'] 
& \cdots & L^{(i)} \arrow[d,"f_i"] 
& \cdots 
& L^{(n)} \arrow[dll,"f_n"] \\
&& M &&
\end{tikzcd}
\end{equation}
More concretely, the fiber product polytope is defined by
$$
L^{(1)} \times_M \cdots \times_M L^{(n)}
:= \set{(x^{(1)},\dots,x^{(n)}) \in L^{(1)} \times \cdots \times L^{(n)}
: \; f_1(x^{(1)})=\cdots=f_n(x^{(n)})}.
$$
\end{defn}
%
%
%\coc{immediate consequence of the theorem, better as a corollary}

\begin{cor}\label{cor:thetheoremforfibre}
A point $x=(x^{(1)},\dots,x^{(n)})$ in the fiber product
$L^{(1)} \times_M \cdots \times_M L^{(n)}$
is a vertex if and only if the set
\begin{equation}\label{eq:limchixfiber}
\Bigl\{(y^{(1)},\dots,y^{(n)}) \in \prod_{i=1}^n \Conv\!\bigl(\Vsupp(x^{(i)})\bigr)
:\; f_1(y^{(1)})=\cdots=f_n(y^{(n)})\Bigr\}
\end{equation}
consists only of $x$.
\end{cor}

\begin{proof}
Let $\chi$ denote the diagram in~\eqref{eq:fibeq}. Then the limit of 
$\chi_x$ is equal to the set in~\eqref{eq:limchixfiber}. The result therefore
follows from Proposition~\ref{pro:Mainresultpoly}.
\end{proof}

Now, we come to our main result in this section, which provides a sufficient condition for a point in the fiber product to be a vertex. This condition relies on the geometric properties of specific sets in the polytope $M$. We begin by recalling the definition of affine independence.

\begin{defn}\label{def:affinindep}
A finite set of points $\{v_1, \dots, v_k\} \subset \mathbb{R}^n$ is said to be \emph{affinely independent} 
if it satisfies one of the following equivalent conditions:
\begin{enumerate}
     \item If for $0 \leq \alpha_1,\dots,\alpha_k,\beta_1,\dots,\beta_k \leq 1$ with $\sum_{i=1}^k \alpha_i = \sum_{i=1}^k \beta_i=1$, we have
    \[
   \sum_{i=1}^k \alpha_i v_i = \sum_{i=1}^k \beta_i v_i, 
    \]
    then $\alpha_i=\beta_i$ for every $1\leq i \leq k$.
    \item The only real coefficients $\lambda_1, \dots, \lambda_k$ satisfying
    \[
    \sum_{i=1}^k \lambda_i v_i = 0
    \quad \text{and} \quad
    \sum_{i=1}^k \lambda_i = 0
    \]
    are $\lambda_0 = \lambda_1 = \cdots = \lambda_k = 0$. 
\end{enumerate}
\end{defn}
%

%\coc{upgraded to a theorem as our main result}

\begin{thm}\label{thm:VertofGluingmain}
Given the diagram~\eqref{eq:fibeq} and a point
$x=(x^{(1)},\dots,x^{(n)}) \in L^{(1)} \times_M \cdots \times_M L^{(n)}$,
define, for each $1 \leq i \leq n$, the set
\[
A_i := \{\, f_i(y) :\; y \in \Vsupp(x^{(i)}) \,\}.
\]
If the following conditions hold:
\begin{itemize}
    \item for every $1 \leq i \leq n$, the restriction of $f_i$ to $\Vsupp(x^{(i)})$ is injective (equivalently, $|A_i|=|\Vsupp(x^{(i)})|$), and the set $A_i$ is affinely independent;
    \item the intersection $\bigcap_{i=1}^n \Conv(A_i)$ consists of a single point;
\end{itemize}
then $x$ is a vertex.
\end{thm}
\begin{proof}
Assume that $\tilde{x}=(\tilde{x}^{(1)},\dots,\tilde{x}^{(n)})$ belongs to the set in~\eqref{eq:limchixfiber}.
Fix $1 \leq j \leq n$, and suppose that
\begin{itemize}
    \item the elements of $\Vsupp(x^{(j)})$ are ${y^{(1)},\dots,y^{(k)}}$,
    \item $x^{(j)}=\alpha_1 y^{(1)}+\dots+\alpha_k y^{(k)}$,
    \item $\tilde{x}^{(j)}=\beta_1 y^{(1)}+\dots+\beta_k y^{(k)}$.
\end{itemize}

Since $\bigl|\bigcap_{i=1}^n \Conv A_i\bigr|=1$, applying $f_j$ yields
$$
\alpha_1 f_j(y^{(1)})+\dots+\alpha_k f_j(y^{(k)})
=
\beta_1 f_j(y^{(1)})+\dots+\beta_k f_j(y^{(k)}).
$$
Note that $f_j(y^{(1)}),\dots,f_j(y^{(k)})$ are distinct, since
$|A_j|=|\Vsupp(x^{(j)})|$. By the assumption that these points are affinely
independent, it follows that $\alpha_i=\beta_i$ for every $1\leq i \leq k$.
Hence $x^{(j)}=\tilde{x}^{(j)}$.
Since this holds for each $1 \leq j \leq n$, we conclude that $x=\tilde{x}$.
Therefore, by Corollary~\ref{cor:thetheoremforfibre}, the point $x$ is a vertex.
\end{proof}
%
%
%
%
%
%Now, we show what is needed to obtain a vertex 
%that satisfying the conditions of Proposition \ref{pro:VertofGluingmain}. 

Next, we show that the conditions of Theorem~\ref{thm:VertofGluingmain} can be satisfied by constructing sets of vertices {that generate faces} in the polytopes $L^{(i)}$.

%\begin{pro}\label{pro:VertexAimain}
%Given the diagram~\eqref{eq:fibeq} and sets $A_1,\dots,A_{n}$ such that:  
%
%\begin{itemize}
%    \item for every $1 \leq i \leq n$, the elements of $A_i$ are affinely independent,
%    \item the intersection $\cap_{i=1}^n \Conv A_i$ contains exactly one point,
%    \item for every $1 \leq i \leq n$, there is a generating face (Definition \ref{def:closeVert}) $\tilde{A}_i$ in $L^{(i)}$, such that 
%$f_i|_{\tilde{A}_i}$ induces a bijection from $\tilde{A_i}$ to $A_i$.
%\end{itemize}
%We define $x=(x^{(1)},\dots,x^{(n)}) \in L^{(1)} \times_M \cdots \times_M L^{(n)}$ by setting: 
%
%$$
%x^{(j)}=\sum_{y \in \tilde{A}_j }\alpha_{y} y.
%$$
%where $\alpha_{y}$ is the coefficient of $f_j(y)$ in the convex combination that represents the unique element in $\cap_{i=1}^n \Conv A_i$, expressed using the elements of $A_j$. Then $x$ is a vertex.
%\end{pro}

\begin{pro}\label{pro:VertexAimain}
Given the diagram~\eqref{eq:fibeq}, suppose we are given subsets $A_1,\dots,A_n\subseteq M $ such that:
\begin{itemize}
    \item for each $1 \le i \le n$, the set $A_i$ is affinely independent;
    \item the intersection $\bigcap_{i=1}^n \Conv(A_i)$ consists of exactly one point;
    \item for each $1 \le i \le n$, there exists a set $\tilde{A}_i$ {that generates a face} (Definition~\ref{def:closeVert}) in $L^{(i)}$ such that
    $f_i|_{\tilde{A}_i}$ induces a bijection $\tilde{A}_i \to A_i$.
\end{itemize}
Let $u$ denote the unique point in $\bigcap_{i=1}^n \Conv(A_i)$. For each $j$, express this point as a convex combination of elements of $A_j$,
\[
u=\sum_{z\in A_j}\alpha_z\, z,
\]
and define $x=(x^{(1)},\dots,x^{(n)}) \in L^{(1)} \times_M \cdots \times_M L^{(n)}$ by
\[
x^{(j)}=\sum_{y\in \tilde{A}_j}\alpha_{f_j(y)}\, y.
\]
Then $x$ is a vertex of $L^{(1)} \times_M \cdots \times_M L^{(n)}$.
\end{pro}

\begin{proof}
First, we show $x$ is a well-defined point of the fiber product. Since $f_j|_{\tilde{A}_j}$ induces a bijection from $\tilde{A_j}\to A_j$, we get
$$
f_j(x^{(j)})=f_j(\sum_{y \in \tilde{A}_j }\alpha_{y} y )=\sum_{y \in \tilde{A}_j }\alpha_{y} f_j(y)=\sum_{z \in A_j }\alpha_{z} z=u.
$$
%which equals the unique point in $\cap_{i=1}^n \Conv A_i$. 
So, $x$ is well defined. 
Since $\tilde{A_j}$ {generates a face}, using Lemma \ref{lem:genface} we obtain 
$$
\Vsupp(x^{(j)})=\Vsupp(\sum_{y \in \tilde{A}_j }\alpha_{y} y) \subseteq \tilde{A_j}.
$$
Then
$$
A'_j:=\set{f_j(y): \; y \in \Vsupp(x^{(j)})}\subseteq \set{f_j(y): \; y \in \tilde{A_j}} =A_j. 
$$
We conclude that the set $A'_j$ is affinely independent, $|A'_j|=|\Vsupp(p|_{X^{(j)}})|$, and 
\[
\bigcap_{i=1}^n \Conv A'_i \subseteq \bigcap_{i=1}^n \Conv A_i.\]
On the other hand, it is clear that the unique point $u$ of 
$\bigcap_{i=1}^n \Conv A_i$ is contained in  $\bigcap_{i=1}^n \Conv A'_i$. 
By Theorem \ref{thm:VertofGluingmain}, we obtain that $p$ is a vertex.
\end{proof}
The converse of Theorem \ref{thm:VertofGluingmain} does not hold in general; see Examples \ref{ex:biparcounterex} and \ref{ex:counterex} below. Nevertheless, we establish two partial converses.
% results.
%
%
%

\begin{pro}\label{pro:affinXimpliesaffinWmain}
With the notation of Theorem \ref{thm:VertofGluingmain}, suppose that $x$ is a vertex and that,
for some $1\le j\le n$, the set $\Vsupp(x^{(j)})$ is affinely independent.
Then $|A_j|=|\Vsupp(x^{(j)})|$, and the set $A_j$ is also affinely independent.
\end{pro}

\begin{proof}
Let $\Vsupp(x^{(j)})=\set{y^{(1)},\dots,y^{(k)}}$.  
Since this set is affinely independent, there exist unique coefficients 
$0 < \alpha_1, \dots, \alpha_k < 1$ such that 
$$
x^{(j)} = \alpha_1 y^{(1)} + \dots + \alpha_k y^{(k)}.
$$
Suppose, for a contradiction, that the points 
$f_j(y^{(1)}),\dots,f_j(y^{(k)})$ are affinely dependent (in particular, this includes the case where some of them coincides, meaning 
$|A_j|\neq|\Vsupp(x^{(j)})|$).  
Then, by part~(2) of Definition~\ref{def:affinindep}, there exist real numbers 
$\lambda_1, \dots, \lambda_k$, not all zero, such that
\[
\sum_{i=1}^k \lambda_i f_j(y^{(i)}) = {0},
\qquad 
\sum_{i=1}^k \lambda_i = 0.
\]
Choose $\epsilon > 0$ small enough so that $\alpha_i + \epsilon \lambda_i > 0$ for every 
$1 \leq i \leq k$.  
Define $\beta_i = \alpha_i + \epsilon \lambda_i$.  
Then
 \[\displaystyle \sum_{i=1}^k \beta_i 
    = \sum_{i=1}^k (\alpha_i + \epsilon \lambda_i) 
    = \sum_{i=1}^k \alpha_i + \epsilon \sum_{i=1}^k \lambda_i 
    = 1 + \epsilon \cdot 0 = 1\]
    and
   \[ f_j(\displaystyle \sum_{i=1}^k \beta_i y^{(i)}) 
    = \sum_{i=1}^k \alpha_i  f_j(y^{(i)}) + 
    \epsilon \sum_{i=1}^k \lambda_i  f_j(y^{(i)}) 
    = f_j(\sum_{i=1}^k \alpha_i  y^{(i)}) + \epsilon \cdot \vec{0} 
    = f_j(x^{(j)}).\]
Thus, the tuple
\[
(x^{(1)}, \dots, x^{(j-1)}, 
\textstyle\sum_{i=1}^k \beta_i y^{(i)}, 
x^{(j+1)}, \dots, x^{(n)})
\]
lies in the set defined in~\eqref{eq:limchixfiber}.  
Since the set $\{y^{(1)}, \dots, y^{(k)}\}$ is affinely independent, we have 
$\sum_{i=1}^k \beta_i y^{(i)} \neq x^{(j)}$,  
which contradicts Corollary~\ref{cor:thetheoremforfibre}.  
Hence, $|A_j|=|\Vsupp(x^{(j)})|$ and $A_j$ must be affinely independent.
\end{proof}

%\begin{pro}\label{pro:VertofGluihalfsecondirmain}
%With the same notation as in Proposition~\ref{pro:VertofGluingmain}, if $x$ is a vertex, then 
%the intersection $\bigcap_{i=1}^n \Conv(A_i)$ contains exactly one point. {In addition, for every $1 \leq j \leq n$, we can write an 
%affine combination for the common point in $\bigcap_{i=1}^n \Conv(A_i)$ by the elements of $A_j$ such that 
%the coefficient of every $x \in A_j$ is nonzero.}
%\end{pro}

\begin{pro}\label{pro:VertofGluihalfsecondirmain}
With the notation of Theorem \ref{thm:VertofGluingmain}, if $x$ is a vertex, then
the intersection $\bigcap_{i=1}^n \Conv(A_i)$ consists of exactly one point $u$.
{Moreover, for every $1 \le j \le n$, the unique point $u$ admits an affine
decomposition in terms of the elements of $A_j$ in which all coefficients are nonzero.}
\end{pro}

\begin{proof}
Given $1 \leq j \leq n$ and $y \in \Vsupp(x^{(j)})$, by Lemma~\ref{lem:xleqythen} there exist
$0<\alpha\leq 1$ and $ y' \in L^{(j)}$ such that
\begin{equation}\label{eq:p|Xqands}
x^{(j)}=\alpha y + (1-\alpha)y' .
\end{equation}
If $\alpha<1$, then 
by 
%the converse direction of 
Lemma~\ref{lem:xleqythen}, we have $y' \preceq x^{(j)}$.
Hence, by Proposition~\ref{pro:yleqxVsupp}, $y'$ can be written as a convex combination of vertices in
$\Vsupp(x^{(j)})$.
Combining this with \eqref{eq:p|Xqands}, we can express $x^{(j)}$ as a convex combination of vertices in
$\Vsupp(x^{(j)})$ in which the coefficient of $y$ is nonzero. This also the case if $\alpha=1$ ($y=x^{(j)}$). Denote such a representation by $T_y$.
Averaging over all $y \in \Vsupp(x^{(j)})$, we obtain
$$
x^{(j)}=\sum_{y \in \Vsupp(x^{(j)})} \frac{1}{|\Vsupp(x^{(j)})|} T_y.
$$
Therefore, there exist coefficients $0<\alpha_y \leq 1$ for
$y \in \Vsupp(x^{(j)})$ such that
$$
x^{(j)}=\sum_{y \in \Vsupp(x^{(j)})} \alpha_y y \; \;\text{and} \;\; \sum_{y \in \Vsupp(x^{(j)})} \alpha_y =1 .
$$
As a result, we have the following 
%common 
point in $\bigcap_{i=1}^n \Conv(A_i)$:
$$
\tilde{z}=f_j(x^{(j)})= \sum_{y \in \Vsupp(x^{(j)})} \alpha_y f_j(y).
$$  
If there were another element $z \in \bigcap_{i=1}^n \Conv(A_i)$, then for each $1\leq j \leq n$ and $y \in \Vsupp(x^{(j)})$, 
there exist coefficients $\alpha_y^{(j)}$ such that  
$$
z = \sum_{y\in \Vsupp(x^{(j)})} \alpha_y^{(j)} \, f_j(y).
$$
Hence
$$
\biggl(\, \sum_{y\in \Vsupp(x^{(1)})} \alpha_y^{(1)} \,y, \dots,  \sum_{y\in \Vsupp(x^{(n)})} \alpha_y^{(n)} \, y \biggr) 
\in \prod_{i=1}^n \Conv\bigl(\Vsupp(x^{(i)})\bigr).
$$
In addition, for each $1\leq j \leq n$, we have
$$
f_j(\sum_{y\in \Vsupp(x^{(j)})} \alpha_y^{(j)} \, y ) =\sum_{y\in \Vsupp(x^{(j)})} \alpha_y^{(j)} \, f_j(y) =z, 
$$
and since $ z \neq  \tilde{z} = f_j(x^{(j)})$ we obtain that  $\sum_{y\in \Vsupp(x^{(j)})} \alpha_y^{(j)} \, y  \neq x^{(j)}$.
This contradicts Corollary~\ref{cor:thetheoremforfibre}.  
Therefore, the intersection must contain exactly one point.
\end{proof}

\section{Extremal simplicial distributions}\label{sec:Exsimpldist}

Simplicial distributions are central to the study of contextuality, nonlocality, and related phenomena in quantum foundations. In this section, we briefly review simplicial distributions and show that they form a polytope in standard form. This allows us to apply the results of Section~\ref{subsec:fibpro} to characterize extremal simplicial distributions on measurement spaces arising as colimits of simplicial sets, in particular those obtained by gluing spaces along a common subspace.

\subsection{Simplicial distributions}\label{sec:Simdist}

%In this section, we introduce the notion of simplicial distributions. To do so, we first recall the concept of a simplicial set, which provides the underlying combinatorial structure on which these distributions are defined.

%Our main object of study in this paper is the polytopes that arise from simplicial distributions. 
We begin by introducing the preliminaries from the theory of simplicial sets. 

\begin{defn}
A \emph{simplicial set} consists of a sequence of sets
\[
X_0, X_1, X_2, \dots
\]
together with maps called \emph{face maps}
\[
d_i=d^X_i \colon X_n \to X_{n-1} \qquad \text{for } 0 \leq i \leq n,
\]
and \emph{degeneracy maps}
\[
s_j =s^X_j \colon X_n \to X_{n+1} \qquad \text{for } 0 \leq j \leq n,
\]
for each $n \geq 0$, satisfying the \emph{simplicial identities}.
%following simplicial identities for all indices where the expressions make sense:
%\begin{align*}
%d_i d_j &= d_{j-1} d_i \quad \text{if } i < j, \\
%s_i s_j &= s_{j+1} s_i \quad \text{if } i \leq j, \\
%d_i s_j &=
%\begin{cases}
%s_{j-1} d_i & \text{if } i < j, \\
%\text{id} & \text{if } i = j \text{ or } i = j+1, \\
%s_j d_{i-1} & \text{if } i > j+1.
%\end{cases}
%\end{align*}
The elements of $X_n$ are called the \emph{$n$-simplices}. A simplex is called \emph{degenerate} if it lies in the image of a degeneracy map. Otherwise, it is called \emph{non-degenerate}.
A non-degenerate simplex is called a \emph{generating simplex} if it is not in the image of a face map.
%a face of another simplex.

Given simplicial sets $X,Y$, \emph{a simplicial set map} $f\colon X \to Y$ is a collection of set maps $\{f_n\colon X_n \to Y_n\}_{n\geq 0}$ such that 
$$
f_{n-1} \circ d_i =d_i \circ f_n, \quad f_{n+1} \circ s_i =s_i \circ f_n.
$$
For an $n$-simplex $x$, we will write $f_x$ instead of $f_n(x)$. The \emph{category of simplicial sets} will be denoted by $\catsSet$, and so the set of simplicial maps from $X$ to $Y$ will be denoted by $\catsSet(X,Y)$.
\end{defn}

The simplicial identities omitted from the definition above can be found in \cite{friedman2008elementary}.
We will not need their explicit form here. Intuitively, the face maps encode how simplices
in each dimension are glued together, while the degeneracy maps indicate which simplices
are regarded as degenerate (and hence collapsed).

%The face maps describe how to extract the $(n-1)$-dimensional "faces" of an $n$-simplex, while the degeneracy maps represent simplices that arise from inserting repetitions (i.e., degenerate simplices).

%Intuitively, a simplicial set encodes a combinatorial structure made of vertices (0-simplices), edges (1-simplices), triangles (2-simplices), and higher-dimensional analogues, together with consistent ways to "glue" them along faces.

%Important examples of simplicial sets that will be used in this paper: 
%\begin{itemize}
%    \item The standard simplex $\Delta^n$, which is the simplicial set generated by one nondegenerate $n$-simplex.  
%    Concretely, it has $n+1$ vertices, and every face of the $n$-simplex gives a nondegenerate simplex in lower dimension.
%    \item A simplicial set consisting of a finite set of nodes is one in which, for every $n \geq 1$, all $n$-simplices are degenerate.  
%In this case, we will simply identify the simplicial set with its set of nodes.
%    \item A connected directed graph, which can be represented as a connected $1$-dimensional simplicial set. It is the case when for every $n \geq 2$, all the $n$-simplices are degenerate simplices. 
%    \item For $m \geq 2$, the simplicial set $\Delta_{\ZZ_m}$, which defined by setting $(\Delta_{\ZZ_m})_n=\ZZ_m^{n+1}$, where the face and degeneracy maps are given by deletion and repetition, respectively.
%\end{itemize}

\begin{ex}\label{ex:simplsets}
We list the main examples of simplicial sets that will appear in this paper.
\begin{enumerate}
\item
Let $[n]$ denote the ordered set $\{0,1,\dots,n\}$. The \emph{$k$-simplex} is the simplicial set,
denoted by $\Delta^k$, whose $n$-simplices are order-preserving functions $\sigma \colon [n]\to [k]$.
For $0\le i\le n$, let
\[
d^i\colon[n-1]\to [n]
\]
be the injective order-preserving map whose image is $[n]\setminus\{i\}$ (equivalently, it ``skips'' $i$).
Then the $i$th face map is given by precomposition,
\[
d_i(\sigma)=\sigma\circ d^i.
\]
For $0\le j\le n$, let
\[
s^j\colon [n+1]\to [n]
\]
be the surjective order-preserving map that identifies $j$ and $j+1$ (equivalently, $s^j(j)=s^j(j+1)=j$
and $s^j$ is strictly increasing elsewhere). Then the $j$th degeneracy map is
\[
s_j(\sigma)=\sigma\circ s^j.
\]
This simplicial set models the topological $k$-simplex with vertices $0,1,\dots,k$.

\item
A \emph{one-dimensional simplicial set} amounts to a {directed} (multi)graph $X=(X_0,X_1,d_0,d_1)$ together with
a specified loop $s_0\colon X_0\to X_1$ at each vertex, satisfying $d_i\circ s_0=\on{id}_{X_0}$ for $i=0,1$.
Note that such a simplicial set can be constructed by gluing a collection of $1$-simplices $\Delta^1$
in the form of a colimit. In fact, this observation generalizes to arbitrary simplicial sets:
any simplicial set can be expressed as a colimit of its simplices. Later in this paper we will consider
special types of colimits obtained by gluing a collection of simplicial sets along a common simplicial subset.

%\coc{in the item below, saying that this construction is related to the standard simplex makes it more acceptable, so it is not coming from nowhere}

\item 
For a set $U$, let $\Delta_U$ denote the simplicial set defined by
$
(\Delta_U)_n = U^{n+1}$,
with face and degeneracy maps given by deletion and repetition:
\[
d_i(u_0,\dots,u_n)=(u_0,\dots,\widehat{u_i},\dots,u_n), \qquad 0\le i\le n,
\]
and
\[
s_j(u_0,\dots,u_n)=(u_0,\dots,u_j,u_j,\dots,u_n), \qquad 0\le j\le n.
\]
When $U=[m-1]=\{0,\dots,m-1\}$, the simplicial set $\Delta_{[m-1]}$ may be viewed as an unoriented version of the standard simplex $\Delta^{m-1}$, since its simplices are all functions $[n]\to [m-1]$, without the order-preserving requirement. It is convenient to identify $[m-1]$ with $\ZZ_m$. In this paper, we will specialize to the outcome space
$\Delta_{\ZZ_m}$,
where $m\ge 2$
%. This simplicial set plays the role of the outcome space in the theory of simplicial distributions 
(see Remark~\ref{rem:simdistmeaning}).

%{For $m \geq 2$, the simplicial set $\Delta_{\ZZ_m}$, which defined by setting $(\Delta_{\ZZ_m})_n=\ZZ_m^{n+1}$, where the face and degeneracy maps are given by deletion and repetition, respectively. This simplicial set typically plays the role of the outcome space in the theory of simplicial distributions 
%(see Remark \ref{rem:simdistmeaning}).
%}
%There is an unordered version of the standard simplex. Let $\Upsilon^m$ denote the simplicial set,
%called the \emph{unoriented $m$-simplex}, whose $n$-simplices are \emph{all} functions $\sigma\colon [n]\to [m]$
%(with no order-preserving requirement). The face and degeneracy maps are still given by the same
%precomposition formulas
%\[
%d_i(\sigma)=\sigma\circ d^i, \qquad s_j(\sigma)=\sigma\circ s^j,
%\]
%using the maps $d^i$ and $s^j$ defined above. The standard simplex $\Delta^m$ is a simplicial subset of
%$\Upsilon^m$.
\end{enumerate}
\end{ex}

\begin{defn}\label{ref:distribution monad} 
For a set $X$, the set $D(X)$ of probability distributions over $X$ is defined by
\[
D(X) = \left\{p\colon X \to \RR_{\geq0} :\; \sum_{x\in X} p(x) = 1 \right\}.
\]
For a map $f\colon X\to Y$, the induced map $D(f) \colon D(X) \to D(Y)$ is defined by:
    \[
    D(f)(p)(y) = \sum_{x:\; f(x)=y}  p(x).
    \]
This assignment can be extended to a functor $D\colon \catSet \to \catSet$, called the \emph{distribution functor}.
\end{defn}

The distribution functor can be employed with the structure of a monad. Within this framework, algebras in the category of sets are precisely the convex sets. This structure can be carried over to simplicial sets. Let $X$ be a simplicial set. The distribution monad can be applied levelwise to define a new simplicial set $D(X)$ by:
    \[
    D(X)_n := D(X_n),
    \]
    where the face and degeneracy maps of $D(X)$ are defined by applying $D$ to the corresponding structure maps of $X$:
    \[
    d_i^{D(X)} := D(d_i^{X}) \quad \text{and} \quad s_j^{D(X)} := D(s_j^{X}).
    \]
The action on simplicial set maps is also defined levelwise. In this way, $D$ defines a functor on the category of simplicial sets, $D\colon \catsSet \to \catsSet$, {see} \cite[{Proposition 2.6}]{kharoof2022simplicial}.

We are now ready to introduce simplicial distributions \cite{okay2022simplicial}.

\begin{defn}\label{def:simpDistDeter} 
Let $X$ and $Y$ be simplicial sets. A \emph{simplicial distribution} on $(X,Y)$ is a simplicial set map
\[
p \colon X \to D(Y).
\]
We will write $\sDist(X,Y)$ for the set of simplicial distributions on $(X,Y)$. A simplicial distribution is called \emph{deterministic} if, for every $n$-simplex $x \in X_n$, 
the distribution $p_x\in D(Y_n)$ satisfies $p_x^{y} = 1$ for some $y \in Y_n$ and is zero for all others.
\end{defn}

Explicitly, a simplicial distribution assigns to each simplex $x \in X_n$ a probability distribution $p_x$ over the $n$-simplices of $Y$, 
in a way that is compatible with restriction along the face maps.  When $X$ is a finitely generated simplicial set (i.e., with finitely many generating simplices) and $Y$ has finitely many simplices in each dimension, then $\sDist(X,Y)$ is a polytope in standard form.

\begin{rem}\label{rem:simdistmeaning}
Simplicial distributions provide a natural framework for modeling measurement statistics in quantum mechanics. In this setting, $X$ represents a measurement {space} (the space of measurements), while $Y$ represents the outcome space. Both are equipped with the structure of simplicial sets, where each measurement and outcome corresponds to a simplex of appropriate dimension. Specifically, for a measurement $x \in X_n$, the probability $p_x^y \in [0,1]$ represents the probability of observing outcome $y \in Y_n$. The simplicial structure of $X$ encodes the compatibility relations between measurements—a fundamental aspect of quantum mechanics arising from the non-commutativity of observables. In quantum experiments, we only have access to joint measurement statistics for compatible measurements. These compatible sets are precisely the simplices of $X$, and the corresponding probability distributions form the simplicial distribution $p$.
\end{rem}

\begin{pro}\label{pro:simdistspecialpoly}
Given a pair $(X,Y)$ of simplicial sets where $X$ is finitely generated and $Y$ has finitely many simplices in each dimension, the set $\sDist(X, Y)$ is a polytope in standard form.
\end{pro}

\begin{proof}
Suppose that $X$ is generated by finitely many simplices
$\sigma_1 \in X_{n_1}, \dots, \sigma_k \in X_{n_k}$.
A simplicial distribution $p\colon X\to D(Y)$ is uniquely determined by the tuple
$(p_{\sigma_1},\dots,p_{\sigma_k})$, where each $p_{\sigma_j}$ is a probability
distribution on the finite set $Y_{n_j}$.
Hence $\sDist(X,Y)$ can be identified with a subset of $\RR^{N}$, where
\[
N=\sum_{j=1}^k |Y_{n_j}|,
\]
consisting of vectors $p=(p_i)$ satisfying:
\begin{enumerate}
\item \emph{Nonnegativity:} $p_i\ge 0$ for all $1\le i\le N$;
\item \emph{Normalization:} for each generator $\sigma_j$,
\[
\sum_{y\in Y_{n_j}} p_{\sigma_j}(y)=1;
\]
\item \emph{Compatibility (non-signaling):} the simplicial identities for the face
maps impose linear equalities requiring that the restrictions of
$p_{\sigma_j}$ to common faces agree.
\end{enumerate}
Conditions (2) and (3) are linear equalities, so they can be written as
$Ap=b$ for a suitable matrix $A$ and vector $b$. Condition (1) is exactly
$p\in \RR^{N}_{\ge 0}$.

Finally, $\sDist(X,Y)$ is bounded: each block $p_{\sigma_j}$ lies in a standard
simplex (nonnegative entries summing to $1$), so $\sDist(X,Y)\subseteq
\Delta^{|Y_{n_1}|-1}\times\cdots\times \Delta^{|Y_{n_k}|-1}$.
Therefore $\sDist(X,Y)=\{p\in \RR^{N}:\; Ap=b,\; p\ge 0\}$ is a polytope in
standard form.
\end{proof}

\begin{ex}\label{ex:one dim sdist}
A particularly interesting case arises when we restrict $X$ to be a one-dimensional simplicial set. As seen in earlier examples, the data of a connected one-dimensional simplicial set corresponds to a directed (multi)graph $X=(X_0,X_1,d_0,d_1,s_0)$, where each vertex may have self-loops. A simplicial distribution $p\colon X\to D({\Delta_{\ZZ_m}})$ consists of probability distributions $p_x\in D(\ZZ_m^2)$ for each generating simplex $x$, subject to compatibility conditions imposed by the face maps. These compatibility conditions correspond precisely to the row and column sum conditions on the probability matrices: for each edge $x$, we require that the row sums match the probabilities at $d_1(x)$ and the column sums match those at $d_0(x)$. 
For degenerate self-loops at a vertex $v\in X_0$, namely $s_0(v)$, the distribution is determined by the degeneracy map and is given by a diagonal matrix whose diagonal entries are $p_v^{a,a}$ for $a\in \ZZ_m$, while all off-diagonal entries are zero. %\ak{Here, "self-loops" could be misread as non-degenerate loops, which would not be forced to be diagonal.}. 
This structure yields a canonical bijection
\[
\on{Dist}(X,m) \cong \sDist(X,{\Delta_{\ZZ_m}}).
\]
\end{ex}

In the framework of simplicial distributions, contextuality is characterized by the inability to express a distribution as a convex combination of deterministic distributions.

\begin{defn}\label{def:contextualll}
A simplicial distribution $p \colon X \to D(Y)$ is called \emph{non-contextual} if it can be written as a convex combination of deterministic distributions.  
If no such representation exists, we say that $p$ is \emph{contextual}.
\end{defn}

\begin{rem}\label{rem:Bellineq}
The polytope $\sDist(X,Y)$ contains two types of vertices: deterministic distributions (see \cite[Proposition 4.16]{kharoof2022simplicial}) and contextual vertices. The convex hull of all deterministic simplicial distributions forms a subpolytope known as the \emph{Bell polytope}. Contextual simplicial distributions lie outside this polytope and can be detected via Bell inequalities—linear inequalities satisfied by all points in the Bell polytope. The contextual vertices are precisely those simplicial distributions that violate at least one Bell inequality.
\end{rem}

%

%\begin{ex}\label{ex:treenoncont}
%Any tree can be regarded as a simplicial set $X$, obtained by successively gluing edges along common nodes:  
%first gluing two edges at a node, then gluing the resulting space with another edge at some node, and so on.  
%Since every simplicial distribution on $\Delta^1$ is noncontextual (\cite[Example~3.11]{okay2022simplicial}), and by \cite[Corollary~4.6]{okay2022simplicial}, it follows that every simplicial distribution on $(X,Y)$ is noncontextual.
%\end{ex}
%
%
%

%We now give an example of scenarios that admit only noncontextual simplicial distributions; consequently, all vertices are deterministic distributions.

%\begin{ex}\label{ex:treevert}
%If $X$ is a tree, then it can be constructed inductively by starting from 
%$\Delta^1$ and repeatedly gluing copies of 
%$\Delta^1$ along $\Delta^0$. By \cite[Example~3.11 and Corollary~4.6]{okay2022simplicial}, every simplicial distribution in 
%$\sDist(X,Y)$ is noncontextual. Therefore, every vertex of
%$\sDist(X,Y)$ is a deterministic distribution.
%\end{ex}

%\begin{cihan}{orange}{To be moved to a later section just before collapsing is used for the first time.}
%\subsection{The collapsing method}\label{subsec:collapsing}
%

%\coc{this section can appear later after main results when we use it}
%\ak{I prefer to keep it here, since it is used in several places later on (Subsection~\ref{sec:gluingthrtwoedges} and Section~\ref{sec:1dim}). Moreover, this has become a classical technique for extermal simplicial distributions, so it naturally fits in this section.}

%\end{cihan}

\subsection{Gluing measurement spaces}
\label{sec:Glumeas}

In this subsection, we prove our main result on the characterization of vertices of simplicial distributions on measurement spaces obtained by gluing spaces along a common subspace. The argument relies on the corresponding fiber product result, Theorem \ref{thm:VertofGluingmain}, established in the previous section.

% Proposition \ref{pro:VertofGluingmain} on the characterization of vertices of fiber products of polytopes. Our result relies on an analysis of vertex structure of arbitrary inverse limits (Theorem~\ref{thm:Mainresultpoly}).  

%In this section, we generalize \cite[Theorem 5.3]{kharoof2024extremal} by giving a general characterization of vertices for simplicial distributions defined on colimits of finite diagrams of simplicial sets. The argument relies on Theorem~\ref{thm:Mainresultpoly} together with the fact that the functor $\sDist(-,Y)$ sends colimits of simplicial sets to limits of polytopes. 

%We then specialize this characterization to the case in which the measurement space is obtained by gluing several simplicial sets along a common subspace. In this setting, we derive a concrete geometric sufficient condition for a simplicial distribution to be a vertex and establish partial converse under additional assumptions. 

We begin by recalling the notion of colimits in the category of simplicial sets.

%\subsection{The general case}
%

%

\begin{defn}\label{def:colim}
Let $I$ be a finite category, and let 
\begin{equation}\label{eq:diagrammathcalF}
\mathcal{F} \colon I \longrightarrow \catsSet
\end{equation}
be a diagram of simplicial sets. 
The \emph{colimit of $\mathcal{F}$}, denoted by $\colim \mathcal{F}$, is the simplicial set whose $n$-simplices are given by
\[
(\colim \mathcal{F})_n
:=
\left(\bigsqcup_{i\in \mathrm{Ob}(I)} \mathcal{F}(i)_n\right)\Big/\sim,
\]
for every $n\geq 0$, where $\sim$ is the equivalence relation generated by
\[
x \sim \mathcal{F}(f)_n(x), \quad \text{for all morphisms } f \colon i \to j \text{ in } I \text{ and } x \in \mathcal{F}(i)_n.
\]
Explicitly, the face and degeneracy maps on $\colim \mathcal{F}$ are induced from those of the $\mathcal{F}(i)$ and are well-defined because each $\mathcal{F}(f)$ is a simplicial map.
\end{defn}

%\begin{cihan}{purple}{add this phrase to the intro of the section as motivation}
%The following Theorem generalizes \cite[Theorem 5.3]{kharoof2024extremal}.

%\end{cihan}

Simplicial distributions can be pushed forward along simplicial maps, as follows.

\begin{defn}\label{def:pushforward} 
Given a simplicial set map $f\colon X \to X'$, we define the map 
    $$
    f^\ast \colon \sDist(X',Y) \to \sDist(X,Y),
    $$
by sending $p\colon X' \to D(Y)$ to the composite  $ p \circ f\colon X \to D(Y)$.   
\end{defn}
{The map $f^{\ast}$ preserves convex combinations.}

For limits that come from colimits of simplicial sets, we have the following vertex characterization result, which is a direct consequence of Proposition~\ref{pro:Mainresultpoly}.

\begin{thm}\label{thm:GenGluing}
Let $I$ be a finite category, and let 
\begin{equation}\label{eq:diagramchi}
\mathcal{F} \colon I \longrightarrow \catsSet
\end{equation}
be a diagram of simplicial sets.  
Let $I_0 \subseteq I$ denote the set of terminal objects of $I$, and for each $i \in I_0$, let 
$s_i\colon \mathcal{F}(i) \to \colim \mathcal{F}$ denote the structure map.  
Given a simplicial set $Y$, a simplicial distribution 
$p \colon \colim \mathcal{F} \to D(Y)$ is a vertex if and only if it is the unique simplicial distribution 
$q\colon \colim \mathcal{F} \to D(Y)$ such that 
\[
(s_i^{\ast}(q))_{i\in I_0} \in \prod_{i \in I_0} \Conv\!\bigl(\Vsupp(s_i^{\ast}(p))\bigr).
\]
\end{thm}
\begin{proof}
By Proposition~\ref{pro:simdistspecialpoly}, applying the functor 
$\sDist(-,Y)=\catsSet(-,D(Y))$ to the diagram $\mathcal{F}$ yields a diagram in $\catPoly$:
\[
\chi \colon I^{\op} \to \catPoly,
\]
where $\{\sDist(\mathcal{F}(i),Y)\}_{i\in I_0}$ are the initial objects in $\chi$.  
In addition, by~\cite[Theorem 3.4.7]{riehl2017category}, we have
\[
\lim \chi \;=\; \catsSet(\colim \mathcal{F},D(Y)) \;=\; \sDist(\colim \mathcal{F},Y).
\]
Therefore, by Proposition~\ref{pro:Mainresultpoly}, a simplicial distribution 
$p \in \sDist(\colim \mathcal{F},Y)$ is a vertex if and only if it is the unique simplicial distribution 
$q\colon \colim \mathcal{F} \to D(Y)$ contained in $\lim \chi_{p}$.  
By Equation~(\ref{eq:limchix}), the condition $q \in \lim \chi_{p}$ is equivalent to requiring that 
\[
s_i^{\ast}(q) \in \Conv\!\bigl(\Vsupp(s_i^{\ast}(p))\bigr), \qquad \forall i \in I_0.
\]
\end{proof}

This result generalizes \cite[Theorem 5.3]{kharoof2024extremal}.

Our main interest is the following special case.
Let $X^{(1)},\dots,X^{(n)}$ be simplicial sets, and let $W$ be a common simplicial
subset of each $X^{(i)}$. Consider the diagram $\mathcal{F}$:
\begin{equation}\label{eq:colimit}
\begin{tikzcd}[column sep=small]
&& W \arrow[dll,"s_1"'] \arrow[d,"s_i"]  \arrow[drr,"s_n"] & \\
X^{(1)} & \cdots& X^{(i)} &\cdots & X^{(n)}
\end{tikzcd}
\end{equation}
where $s_i \colon W \to X^{(i)}$ are the inclusion maps.
Then the colimit of $\mathcal{F}$ is the simplicial set obtained by gluing all
$X^{(i)}$ along $W$:
$$
\colim \mathcal{F} \cong X^{(1)} \cup_W \cdots \cup_W X^{(n)}.
$$
In each dimension $k\ge 0$, the set of $k$-simplices of $\colim \mathcal{F}$ is the disjoint union of the $k$-simplices of $X^{(1)},\dots,X^{(n)}$, modulo the identifications induced by $W$.
Applying the functor $\sDist(-,Y)$ {on Diagram (\ref{eq:colimit})} yields the following diagram:
\begin{equation}\label{eq:fibfromcolim}
\begin{tikzcd}
\sDist(X^{(1)},Y) \arrow[drr,"s^{\ast}_1"'] 
& \cdots & \sDist(X^{(i)},Y) \arrow[d,"s^{\ast}_i"] 
& \cdots 
& \sDist(X^{(n)},Y) \arrow[dll,"s^{\ast}_n"] \\
&& \sDist(W,Y) &&
\end{tikzcd}
\end{equation}
In this special case, $\sDist(X,Y)$ is the limit of Diagram~\eqref{eq:fibfromcolim}, i.e., the corresponding fiber product:
\[
\sDist(X,Y) \cong \sDist(X^{(1)},Y) \times_{\sDist(W,Y)} \cdots \times_{\sDist(W,Y)} \sDist(X^{(n)},Y).
\]

Before we state our vertex characterization results, we will recall an important result about non-contextual distributions on glued measurement spaces. This result appears in \cite{okay2022simplicial} as the \emph{gluing lemma}.

\begin{lem}
    \label{lem:noncontextglue}
Let $X$ be the simplicial set obtained by gluing $X^{(1)}$ and $X^{(2)}$ along a common simplicial subset $W=\Delta^n$. Then a simplicial distribution $p \in \sDist(X,Y)$ is non-contextual if and only if its restriction to each $X^{(i)}$ is non-contextual.
\end{lem}

\begin{ex}\label{ex:tree}
    Let $X$ be a one-dimensional simplicial set whose underlying graph is a tree. By induction we can show that every simplicial distribution on $X$ is non-contextual. For that we think of $X$ as obtained by a sequence of gluings where each time a $\Delta^1$ is glued along one of its node to the previous stage, which is also a tree. Then we can apply Lemma \ref{lem:noncontextglue} at each stage of the gluing to conclude that every simplicial distribution on $X$ is non-contextual.
\end{ex}

Now, we turn to the problem of characterizing vertices of $\sDist(X,Y)$ when $X$ is obtained by gluing several spaces along a common subspace.
    Under the identification $L^{(i)}=\sDist(X^{(i)},Y)$, we can apply the results of Section~\ref{subsec:fibpro} to this setting. We begin with applying Corollary \ref{cor:thetheoremforfibre}.
%

%\coc{better as a corollary}

\begin{cor}\label{cor:insteadoflimit}
For $X=X^{(1)} \cup_{W} \cdots \cup_{W} X^{(n)}$, a simplicial distribution
$p \colon X \to D(Y)$ is a vertex in $\sDist(X,Y)$ 
if and only if the set
\begin{equation}\label{eq:q1qn}
\set{(q^{(1)},\dots,q^{(n)}) \in  \prod_{i=1}^n \Conv\!\left(\Vsupp(p|_{X^{(i)}})\right) :\; q^{(1)}|_{W}=\cdots=q^{(n)}|_{W}}
\end{equation}
consists only of $(p|_{X^{(1)}},\dots,p|_{X^{(n)}})$.
\end{cor}

Next, we apply Theorem \ref{thm:VertofGluingmain} to obtain the following result.

%\coc{changed to a corollary}

\begin{cor}\label{cor:VertofGluing}
Let $p\in \sDist(X,Y)$, where $X=X^{(1)} \cup_{W} \cdots \cup_{W} X^{(n)}$. For each $1 \leq i \leq n$, define 
\[
A_i=\set{q|_{W}: q \in \Vsupp(p|_{X^{(i)}})}.\] 
If the following conditions hold:
\begin{itemize}
    \item for every $1 \leq i \leq n$, $|A_i|=|\Vsupp(p|_{X^{(i)}})|$, and the set $A_i$ is affinely independent;
    \item the intersection $\cap_{i=1}^n \Conv A_i$ consists of a single point;
\end{itemize}
then $p$ is a vertex.
\end{cor}

Proposition~\ref{pro:VertexAimain} yields the following constructive converse.

%\coc{better as a corollary}

\begin{cor}\label{cor:VertexAi}
Let $X=X^{(1)} \cup_{W} \cdots \cup_{W} X^{(n)}$. Given sets $A_1,\dots,A_{n} \subseteq \sDist(W,Y)$ such that:  
\begin{itemize}
    \item for every $1 \leq i \leq n$, the set $A_i$ is affinely independent;
    \item the intersection $\bigcap_{i=1}^n \Conv (A_i)$ consists of exactly one point $u$;
    \item for every $1 \leq i \leq n$, there exists a set $\tilde{A}_i$ {that generates a face} in $\sDist(X^{(i)},Y)$ such that $s_i^*|_{\tilde{A}_i}$ {is} a bijection from $\tilde{A_i}\to A_i$.
\end{itemize}
We define the simplicial distribution $p\colon X \to D(Y)$ by setting: 
$$
p|_{X^{(j)}}=\sum_{q \in \tilde{A}_j }\alpha_{q|_W} q
$$
where $\alpha_{q|_W}$ is the coefficient in $u=\sum_{q\in \tilde A_j} \alpha_{q|_W} q|_W$. Then $p$ is a vertex.
\end{cor}

Finally, Propositions~\ref{pro:affinXimpliesaffinWmain} and~\ref{pro:VertofGluihalfsecondirmain} imply the following partial converse to Corollary~\ref{cor:VertofGluing}.
%

%\coc{better as a corollary}

\begin{cor}\label{cor:VertofGluihalfsecondir} 
With the same notation as in Corollary~\ref{cor:VertofGluing}, if $p \colon X \to D(Y)$ is a vertex, then the following conditions hold:
\begin{enumerate}
    \item If for some $1 \leq j \leq n$, the set $\Vsupp(p|_{X^{(j)}})$ is affinely independent, then $|A_j|=|\Vsupp(p|_{X^{(j)}})|$ and the set $A_j$ is also affinely independent.
    \item  The intersection $\bigcap_{i=1}^n \Conv(A_i)$ contains exactly one point $u$. In addition, for every $1 \leq j \leq n$, we can write an affine combination for the common point $u$ by the elements of $A_j$ such that the coefficient of every $q \in A_j$ is nonzero.
\end{enumerate}
\end{cor}

\section{Distributions on graphs}\label{sec:distgraphs}
%\coc{In this section we should explicitly say that 
%\[
%\sDist(X,\Delta_{\ZZ_m}) \cong \Dist(X,m)
%\]
%when $X$ is one-dimensional and then in the rest of the sections, since all uses graphs as measurement spaces, we should use $\Dist(X,m)$ consistently.
%}

%In this section, we study extremal simplicial distributions on three types of one-dimensional measurement spaces 
%(directed (multi)graphs):
%\begin{itemize}
%\item \co{complete} bipartite graphs,
%\item rose graphs,
%\item dipole graphs.
%\end{itemize}
%We view each of them as a gluing of several spaces along a common subspace, which allows us to apply the results from Section~\ref{sec:Glumeas}.

In this section, we study extremal simplicial distributions on three types of
one-dimensional measurement spaces: complete bipartite graphs, rose graphs,
and dipole graphs. In each case, the measurement space is naturally viewed
as a gluing of simpler spaces along a common subspace, allowing us to apply
the results of Section~\ref{sec:Glumeas}.

As mentioned in Example~\ref{ex:one dim sdist}, a one-dimensional simplicial set $X$ may be viewed as a directed (multi)graph, and there is a canonical bijection
$$
\Dist(X, m) \cong \sDist(X, \Delta_{\mathbb{Z}_m}).
$$
Since from now on we will work with graphs as measurement spaces, we will use the notation $\Dist(X, m)$ for these simplicial distributions.

\subsection{Complete bipartite graph}
\label{subsec:Gluengtrees}

%\co{Complete} bipartite graphs play a distinguished role in the study of contextuality and nonlocality. The vertices of the graph partitioned into two sets of nodes, which we call $A$ and $B$, represent the measurement settings of two parties, and the edges represent the joint measurements \cite{brunner2014bell}. Using our methods from Section \ref{sec:Glumeas}, we can detect vertices in the corresponding polytopes of simplicial distributions, extending the known classes of vertices in the literature. We treat a \co{complete} bipartite graph as obtained by gluing a family of star graphs, indexed by the nodes in $B$, along their external nodes, which are the nodes in $A$. We then apply the results of Section \ref{sec:Glumeas} to detect vertices in the corresponding fiber product of simplicial distributions.

Complete bipartite graphs arise naturally in the study of contextuality and
nonlocality, where the two parts of the graph represent the measurement
settings of two parties and the edges their joint measurements
\cite{brunner2014bell}.

\begin{defn}\label{def:Kn1n2}
For two natural numbers $n_1$ and $n_2$, we define the simplicial set $K_{n_1,n_2}$ as follows:
\begin{itemize}
    \item $(K_{n_1,n_2})_0=A \sqcup B$, where $|A|=n_1$ and $|B|=n_2$.
    \item The generating simplices consist of the edges $\sigma_{x,y}$ for every $x \in A$ and $y \in B$, with $d_0(\sigma_{x,y})=y$ and $d_1(\sigma_{x,y})=x$.  
\end{itemize}
We sometimes write $K_{A,B}$ to emphasis the sets of nodes. Note that $K_{n_1,n_2}$ is the \emph{complete bipartite graph} with $n_1$ and $n_2$ nodes in the two parts, viewed as a simplicial set; see Figure \ref{fig:bipartite-k33}.

\end{defn}

We view $K_{n_1,n_2}$ as a gluing of star graphs
indexed by the nodes in $B$, along their common external nodes $A$, and
apply the framework of Section~\ref{sec:Glumeas} to detect vertices of the
corresponding fiber product of simplicial distribution polytopes. 
\begin{figure}[ht]
\centering
\begin{tikzpicture}[x=1cm,y=1cm]
  \coordinate (L1) at (0,  1.4);
  \coordinate (L2) at (0,  0);
  \coordinate (L3) at (0, -1.4);
  \coordinate (R1) at (3.2,  1.4);
  \coordinate (R2) at (3.2,  0);
  \coordinate (R3) at (3.2, -1.4);

  \foreach \i in {1,2,3}{
    \foreach \j in {1,2,3}{
      \draw[line width=0.9pt] (L\i)--(R\j);
    }
  }

  \foreach \v in {L1,L2,L3,R1,R2,R3}{
    \node[circle,fill=black,inner sep=1.2pt] at (\v) {};
  }
\end{tikzpicture}
\caption{The complete bipartite graph \(K_{3,3}\).}
\label{fig:bipartite-k33}
\end{figure}
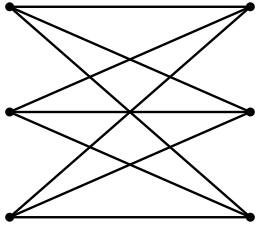

\begin{defn}\label{def:CordG}
A \emph{product-simplex vertex in $\RR^{km}$} is a vector of the form 
\begin{equation}\label{eq:coorvec}
(e_{i_1},
e_{i_2}, \dots,
e_{i_k})^T
=
\begin{pmatrix}
e_{i_1} \\
\rule{0.6em}{0.4pt} \\
e_{i_2} \\
\rule{0.6em}{0.4pt} \\
\vdots \\
\rule{0.6em}{0.4pt} \\
e_{i_k}
\end{pmatrix},
\end{equation}
where $e_{i_j} \in \RR^m$ is the $i_j$\textsuperscript{th} standard basis vector, that is, the vector with a $1$ in the $i_j$-th coordinate and zeros elsewhere. Equivalently, the vector in \eqref{eq:coorvec} is the vector in $\RR^{km}$ whose entries are equal to $1$ in positions
$
i_1,\; m+i_2,\; \dots,\; (k-1)m+i_k,
$
and equal to $0$ elsewhere. These are precisely the vertices of the product of simplices
\[
\underbrace{\Delta^{m-1}\times \cdots \times \Delta^{m-1}}_{k \text{ factors}},
\]
hence the terminology.
\end{defn}

\begin{rem}\label{rem:gluingstars}
The simplicial set
$K_{\{x_1,\dots,x_{n_1}\},\{y_1,\dots,y_{n_2}\}}$
can be viewed as the gluing of the stars  
$$
K_{\{x_1,\dots,x_{n_1}\},\{y_1\}},\dots,
K_{\{x_1,\dots,x_{n_1}\},\{y_{n_2}\}}
$$
along the set of nodes $\{x_1,\dots,x_{n_1}\}$. Since each star is a tree, by Example~\ref{ex:tree} every simplicial distribution on each star is non-contextual and hence every vertex is deterministic.
%Note that all vertices on a star are deterministic distributions
%(see \co{Example~\ref{ex:tree}}). 
More precisely, a vertex in
$\Dist\!\bigl(K_{\{x_1,\dots,x_k\},\{y\}},{m}\bigr)$
is given by the following tuple
\begin{equation}\label{eq:Eikj}
\langle E_{i_1,j},\dots,E_{i_k,j} \rangle,
\end{equation}
where $E_{i,j}$ denotes the $m \times m$ matrix with a $1$ in position $(i,j)$ and zeros elsewhere. This means that the restriction to the edge $\sigma_{x_i,y}$ equals $E_{i,j}$. So the restriction to the set of nodes $\set{x_1,\dots,x_k}$ equals the product-simplex vertex in $\RR^{km}$ of Equation (\ref{eq:coorvec}).
\end{rem}

\begin{ex}\label{ex:A1A2AA4}
Each of the following sets is an affinely independent set of product-simplex vertices in $\RR^{2\cdot 3}$
$$
A_1=\set{\begin{pmatrix}
1 \\
0 \\
0 \\
1 \\
0 \\
0
\end{pmatrix},
\begin{pmatrix}
1 \\
0 \\
0 \\
0 \\
0 \\
1
\end{pmatrix},
\begin{pmatrix}
0 \\
1 \\
0 \\
0 \\
1 \\
0
\end{pmatrix}
\begin{pmatrix}
0 \\
0 \\
1 \\
1 \\
0 \\
0
\end{pmatrix}
}, \;
 A_2=\set{\begin{pmatrix}
1 \\
0 \\
0 \\
1 \\
0 \\
0
\end{pmatrix},
\begin{pmatrix}
1 \\
0 \\
0 \\
0 \\
1 \\
0
\end{pmatrix},
\begin{pmatrix}
0 \\
1 \\
0 \\
1 \\
0 \\
0
\end{pmatrix}
\begin{pmatrix}
0 \\
0 \\
1 \\
0 \\
0 \\
1
\end{pmatrix}
}
$$
$$
A_3=\set{\begin{pmatrix}
1 \\
0 \\
0 \\
1 \\
0 \\
0
\end{pmatrix},
\begin{pmatrix}
1 \\
0 \\
0 \\
0 \\
0 \\
1
\end{pmatrix},
\begin{pmatrix}
0 \\
1 \\
0 \\
1 \\
0 \\
0
\end{pmatrix}
\begin{pmatrix}
0 \\
0 \\
1 \\
0 \\
1 \\
0
\end{pmatrix}
}, \;
 A_4=\set{\begin{pmatrix}
1 \\
0 \\
0 \\
0 \\
1 \\
0
\end{pmatrix},
\begin{pmatrix}
1 \\
0 \\
0 \\
0 \\
0 \\
1
\end{pmatrix},
\begin{pmatrix}
0 \\
1 \\
0 \\
1 \\
0 \\
0
\end{pmatrix}
\begin{pmatrix}
0 \\
0 \\
1 \\
1 \\
0 \\
0
\end{pmatrix}
 }.
$$
Moreover, the intersection of their convex hulls has exactly one point given by the average of the four vertices, i.e., { with all coefficients} equal to $\frac{1}{4}$.

Now, for every $1\leq i \leq 4$, we choose a simplicial distribution $p^{(i)}\colon K_{\set{x_1,x_{2}},\set{y_i}} \to D(\Delta_{\ZZ_3})$ such that 
$\set{q|_{\set{x_1,x_2}}: \; q\in \Vsupp(p^{(i)})}=A_i$, where we mean by $\set{x_1,x_2}$ the simplicial set generated by the nodes $x_1,x_2$. For instance, we define 
$$
p^{(1)}=\langle \begin{pmatrix}%
 \frac{1}{2} & 0 & 0   \\
0 & 0 & \frac{1}{4}   \\
0& \frac{1}{4} & 0   
\end{pmatrix},%
\begin{pmatrix}%
\frac{1}{4} & \frac{1}{4} & 0 \\
0 & 0 & \frac{1}{4}  \\
\frac{1}{4} & 0 & 0    
\end{pmatrix} \rangle
$$
$$
p^{(2)}=\langle \begin{pmatrix}%
 \frac{1}{2} & 0 & 0   \\
0 & 0 & \frac{1}{4}   \\
0& \frac{1}{4} & 0   
\end{pmatrix},%
\begin{pmatrix}%
\frac{1}{4} & 0 & \frac{1}{4}  \\
\frac{1}{4} & 0 & 0  \\
0 & \frac{1}{4} & 0    
\end{pmatrix} \rangle
$$
$$
p^{(3)}=\langle \begin{pmatrix}%
 \frac{1}{2} & 0 & 0   \\
0 & 0 & \frac{1}{4}   \\
0& \frac{1}{4} & 0   
\end{pmatrix},%
\begin{pmatrix}%
\frac{1}{4} & 0 & \frac{1}{4}  \\
0 & \frac{1}{4} & 0  \\
\frac{1}{4} & 0 & 0    
\end{pmatrix} \rangle
$$
$$
p^{(4)}=\langle \begin{pmatrix}%
 \frac{1}{2} & 0 & 0   \\
0 & 0 & \frac{1}{4}   \\
0& \frac{1}{4} & 0   
\end{pmatrix},%
\begin{pmatrix}%
0 & \frac{1}{4} & \frac{1}{4} \\
\frac{1}{4} & 0 & 0  \\
\frac{1}{4} & 0 & 0    
\end{pmatrix} \rangle.
$$
Consider the measurement space $K_{2,4}=K_{\set{x_1,x_2},\set{y_1,y_2,y_3,y_4}}$.
%For every $1\leq i \leq 4$, we define $X^{(i)}$ to be the subsimplicial set of $K_{2,4}$ 
%that generated by the $1$-simplices $\sigma_{x_1,y_i}, \sigma_{x_2,y_i}$, and define $W$ to be the two points $x_1,x_2$. 
We define $p\colon K_{2,4} \to D(\Delta_{\ZZ_3})$ 
by setting $p|_{K_{\set{x_1,x_2},\set{y_i}}}=p^{(i)}$. This is well defined since for all these distributions the restriction on 
$\set{x_1,x_2}$ is 
$$
\begin{pmatrix}%
 1/2 \\
1/4\\
1/4 \\  1/2 \\
1/4\\
1/4
\end{pmatrix}.
$$
By Corollary~\ref{cor:VertofGluing} we obtain that $p$ is a vertex.
\end{ex}

\begin{defn}
Given a product-simplex vertex in $\RR^{nm}$  
$$
v=(e_{i_1},
e_{i_2}, \dots  
e_{i_k})^T,
$$ 
and given an integer $j$ such that $0\leq j \leq m-1$, we define the simplicial distribution 
$$
p(v,j)\colon K_{\set{x_1,\dots,x_n},\set{y}} \to D(\Delta_{\ZZ_m})
$$ 
to be the deterministic distribution in (\ref{eq:Eikj}).
%with $p(v,j)_{\sigma_{x_i,y}}=E^{i_1,j}$
%$$
%p(v,j)_{\sigma_{x_i,y}}^{ab}=1 \;\; \text{if and only if} \; \; v((i-1)m+a)=1 \; \text{and} \; b=j
%$$
%
\end{defn}
Note that the restriction of $p(v,j)$ to $\set{x_1,\dots,x_{n}}$ is equal to the vector $v$.
\begin{ex}
For $$
v=(e_0,e_1,e_0,e_2)^T,
%\begin{pmatrix}%
%1 \\
%0 \\
%0 \\
%0  \\
%1  \\
%0  \\
%1 \\
%0  \\
%0  \\
%0  \\
%0 \\
%1
%\end{pmatrix},
$$
a product-simplex vertex in $\RR^{4 \cdot 3}$, and $j=2$. The simplicial distribution 
$$
p =p(v,j) \colon K_{\set{x_1,x_2,x_3,x_4},\set{y}} \to D(\Delta_{\ZZ_3})
$$ 
is defined by
$$
p_{\sigma_{x_1,y}}=\begin{pmatrix}
0 & 0 & 1\\
0 & 0 & 0\\
0 & 0 & 0
\end{pmatrix}, \;\; p_{\sigma_{x_2,y}}=\begin{pmatrix}
0 & 0 & 0\\
0 & 0 & 1\\
0 & 0 & 0
\end{pmatrix}, \;\;
p_{\sigma_{x_1,y}}=\begin{pmatrix}
0 & 0 & 1\\
0 & 0 & 0\\
0 & 0 & 0
\end{pmatrix}, \;\;
p_{\sigma_{x_1,y}}=\begin{pmatrix}
0 & 0 & 0\\
0 & 0 & 0\\
0 & 0 & 1
\end{pmatrix}.
$$
\end{ex}
%
%
%
%

%\coc{better as a corollary}

\begin{cor}\label{cor:VertexKn1n2}
Let $A_1,\dots,A_{n_2}$ be sets of product-simplex vertices in $\RR^{n_1 \cdot m}$, such that: 
\begin{itemize}
\item $|A_i|\leq m$ for every $1\leq i \leq n_2$,
    \item for every $1 \leq i \leq n_2$, the set $A_i$ is affinely independent,
    \item The intersection $\cap_{i=1}^n \Conv A_i$ contains exactly one point.
\end{itemize}
Consider the measurement space $K_{n_1,n_2}=K_{\set{x_1,\dots,x_{n_1}},\set{y_1,\dots,y_{n_2}}}$. {If we can choose,} for each $1 \leq i \leq n_2$, an injective map 
$f_i \colon A_i \to \ZZ_m$, then we define 
the simplicial distribution $p\colon K_{n_1,n_2} \to D(\Delta_{\ZZ_m})$ by setting
\begin{equation}\label{eq:p|Xi}
p|_{K_{\set{x_1,\dots,x_{n_1}},\set{y_j}}}=\sum_{v\in A_j }\alpha_{v} p(v,f_j(v)),
\end{equation}
where $\alpha_{v}$ is the coefficient of $v$ in the convex combination $u=\sum_{v\in A_j} \alpha_v v$ that represents the unique element $u\in \cap_{i=1}^n \Conv (A_i)$. Then $p$ is a vertex.
\end{cor}
\begin{proof}
For every $1\leq i \leq n_2$, let $X^{(i)}:=K_{\set{x_1,\dots,x_{n_1}},\set{y_i}}$, and let $W$ be the simplicial set of the nodes $x_1,\dots,x_{n_1}$. 
Then, $K_{n_1,n_2}$ is obtained by gluing $X^{(1)},\dots,X^{(n_2)}$ along $W$. Define
$$
\tilde{A}_i=\set{ p(v,f_i(v)): ~ v \in A_i}.
$$
Since each $f_i$ is injective and all the vertices of the polytope of simplicial distributions on a star graph are deterministic distributions (Example~\ref{ex:tree}), it follows that each
$\tilde{A}_i$ 
%is a closed set of vertices in 
generates a face in
$\Dist(X^{(i)},{m})$. Moreover, the restriction of $s_i^*$ (see Diagram (\ref{eq:fibfromcolim})) to $\tilde{A_i}$
induces a bijection 
%from 
$\tilde{A}_i\to A_i$. Consequently, we may rewrite Equation~\eqref{eq:p|Xi} as
$$
p|_{X^{(j)}}=\sum_{q\in \tilde{A}_j }\alpha_{q|_W} q.
$$
By Corollary \ref{cor:VertexAi}, we conclude that $p$ is a vertex.
\end{proof}

\begin{ex}
We apply Corollary~\ref{cor:VertexKn1n2} to detect a vertex in the scenario
$(K_{3,3},\Delta_{\ZZ_3})$.
Each of the following sets is an affinely independent set of product-simplex vertices in $\RR^{3\cdot 3}$:
$$
A_1=\set{\begin{pmatrix}
1 \\
0 \\
0 \\
1 \\
0 \\
0\\
1 \\
0 \\
0
\end{pmatrix},
\begin{pmatrix}
0 \\
0 \\
1 \\
0 \\
0 \\
1 \\
0 \\
0 \\
1
\end{pmatrix},
\begin{pmatrix}
0 \\
1 \\
0 \\
0 \\
1 \\
0 \\
0 \\
1 \\
0
\end{pmatrix}
}, \;
 A_2=\set{\begin{pmatrix}
0 \\
1 \\
0 \\
0 \\
0 \\
1 \\
1 \\
0 \\
0
\end{pmatrix},
\begin{pmatrix}
1 \\
0 \\
0 \\
0 \\
1 \\
0 \\
0 \\
0 \\
1
\end{pmatrix},
\begin{pmatrix}
0 \\
0 \\
1 \\
1 \\
0 \\
0 \\
0 \\
1 \\
0
\end{pmatrix}
}, \;\;
A_3=\set{\begin{pmatrix}
0 \\
1 \\
0 \\
1 \\
0 \\
0 \\ 
0 \\
0 \\
1
\end{pmatrix},
\begin{pmatrix}
1 \\
0 \\
0 \\
0 \\
0 \\
1 \\
0 \\
1 \\
0
\end{pmatrix},
\begin{pmatrix}
0 \\
0 \\
1 \\
0 \\
1 \\
0 \\
1 \\
0 \\
0
\end{pmatrix}
}. 
$$
The intersection $\Conv(A_1) \cap \Conv(A_2) \cap \Conv(A_3)$ consists of exactly one point, with all convex combination coefficients equal
to $\frac{1}{3}$. 

For each $1 \leq i \leq 3$, define a map
$f_i\colon A_i \to \ZZ_3$ sending the first vertex of $A_i$ to $0$, the second
to $1$, and the third to $2$ (according to the order listed above).
Using Equation~\eqref{eq:p|Xi}, we define the vertex
$$
p\colon K_{\set{x_1,x_2,x_3},\set{y_1,y_2,y_3}} \to D(\Delta_{\ZZ_3}),
$$
as follows:
$$
p_{\sigma_{x_1,y_1}}=p_{\sigma_{x_2,y_1}}=p_{\sigma_{x_3,y_1}}=p_{\sigma_{x_2,y_3}}=p_{\sigma_{x_3,y_2}}=\begin{pmatrix}
\frac{1}{3} & 0 & 0  \\
0 & 0 & \frac{1}{3} \\
0 & \frac{1}{3} & 0 \\
\end{pmatrix},
$$
$$
p_{\sigma_{x_1,y_2}}=p_{\sigma_{x_1,y_3}}=\begin{pmatrix}
0 & \frac{1}{3} & 0  \\
\frac{1}{3} & 0 & 0 \\
0 & 0 & \frac{1}{3} \\
\end{pmatrix}, \;\;
p_{\sigma_{x_2,y_2}}=p_{\sigma_{x_3,y_3}}=\begin{pmatrix}
0 & 0 & \frac{1}{3}  \\
0 & \frac{1}{3} & 0 \\
\frac{1}{3} & 0 & 0 \\
\end{pmatrix}.
$$
\end{ex}

Note that Example \ref{ex:A1A2AA4} does not satisfy the conditions of Corollary \ref{cor:VertexKn1n2}, since $|A_i|=4 >|\ZZ_3|=3$. Hence it falls outside the scope of the general construction given in this proposition.

We conclude this section by giving an example showing that, in a bipartite scenario, using a decomposition as described in Remark \ref{rem:gluingstars}, {the first condition of Corollary~\ref{cor:VertofGluing} is not necessary for a point to be a vertex.}
\begin{ex}\label{ex:biparcounterex}
Consider the graph $K_{2,5}=K_{\set{x_1,x_2},\set{y_1,y_2,y_3,y_4,y_5}}$. It is obtained by gluing the following stars:
$$
X^{(i)}=K_{\set{x_1,x_2},\set{y_i}}, \;\; 1\leq i \leq 5.
$$
We define the simplicial distribution $p\colon K_{2,5} \to D(\Delta_{\ZZ_4})$ as follows: 
$$
p|_{X^{(1)}}=\langle \begin{pmatrix}%
\frac{1}{8} & 0 & 0 & 0   \\
\frac{1}{8} & 0 & 0 & 0   \\
0 & \frac{1}{8} & 0 & \frac{1}{8} \\
0 & 0 & \frac{1}{2} & 0
\end{pmatrix},%
\begin{pmatrix}%
\frac{1}{4} & 0 & 0 & 0 \\
0 & \frac{1}{8} & 0 & 0  \\
0 & 0 & \frac{1}{2} & 0  \\
0 & 0 & 0  & \frac{1}{8}
\end{pmatrix}\rangle
$$
$$
p|_{X^{(2)}}=\langle \begin{pmatrix}%
0 & 0 & \frac{1}{8} & 0   \\
0 & 0 & \frac{1}{8} & 0   \\
0 & 0 & \frac{1}{4} & 0   \\
\frac{1}{4} & \frac{1}{8} & 0 & \frac{1}{8} 
\end{pmatrix},%
\begin{pmatrix}%
\frac{1}{4} & 0 & 0 & 0 \\
0 & \frac{1}{8} & 0 & 0  \\
0 & 0 & \frac{1}{2} & 0  \\
0 & 0 & 0  & \frac{1}{8}
\end{pmatrix} \rangle
$$
$$
p|_{X^{(3)}}=\langle \begin{pmatrix}%
0 & \frac{1}{8} & 0 & 0 \\
0 & 0 & 0 & \frac{1}{8} \\
0 & 0 & \frac{1}{4} & 0   \\
\frac{1}{4} & 0 & \frac{1}{4} & 0 
\end{pmatrix},%
\begin{pmatrix}%
\frac{1}{4} & 0 & 0 & 0 \\
0 & \frac{1}{8} & 0 & 0  \\
0 & 0 & \frac{1}{2} & 0  \\
0 & 0 & 0  & \frac{1}{8}
\end{pmatrix} \rangle
$$
$$
p|_{X^{(4)}}= \langle \begin{pmatrix}%
0 & 0 & \frac{1}{8} & 0  \\
\frac{1}{8} & 0 & 0 & 0\\
0 & 0 & \frac{1}{4} & 0   \\
0 & \frac{1}{4} & 0 & \frac{1}{4} 
\end{pmatrix},%
\begin{pmatrix}%
\frac{1}{8} & \frac{1}{8} & 0 & 0 \\
0 & 0 & \frac{1}{8} & 0  \\
0 & 0 & \frac{1}{4} & \frac{1}{4}  \\
0 & \frac{1}{8} & 0 & 0
\end{pmatrix} \rangle
$$
$$
p|_{X^{(5)}}=\langle \begin{pmatrix}%
0 & 0 & 0 & \frac{1}{8}  \\
0 & \frac{1}{8} & 0 & 0 \\
\frac{1}{4} & 0 & 0 & 0   \\
0 & 0 & \frac{1}{2} & 0 
\end{pmatrix},%
\begin{pmatrix}%
\frac{1}{4} & 0 & 0 & 0 \\
0 & \frac{1}{8} & 0 & 0  \\
0 & 0 & \frac{1}{2} & 0  \\
0 & 0 & 0  & \frac{1}{8}
\end{pmatrix} \rangle
$$
As before, we denote by $E_{i,j}$ the $m \times m$ matrix with entry $1$ in position $(i,j)$ and zeros elsewhere. {Then}, the vertex supports are:
\[
\Vsupp\!\left(p\big|_{X^{(1)}}\right)
=
\Big\{
\langle E_{0,0}, E_{0,0} \rangle,
\langle E_{1,0}, E_{0,0} \rangle,
\langle E_{2,1}, E_{1,1} \rangle,
\langle E_{3,2}, E_{2,2} \rangle,
\langle E_{2,3}, E_{3,3} \rangle
\Big\}
\]
\[
\Vsupp\!\left(p\big|_{X^{(2)}}\right)
=
\Big\{
\langle E_{3,0}, E_{0,0} \rangle,
\langle E_{3,1}, E_{1,1} \rangle,
\langle E_{0,2}, E_{2,2} \rangle,
\langle E_{1,2}, E_{2,2} \rangle,
\langle E_{2,2}, E_{2,2} \rangle,
\langle E_{3,3}, E_{3,3} \rangle
\Big\}
\]
\[
\Vsupp\!\left(p\big|_{X^{(3)}}\right)
=
\Big\{
\langle E_{3,0}, E_{0,0} \rangle,
\langle E_{0,1}, E_{1,1} \rangle,
\langle E_{2,2}, E_{2,2} \rangle,
\langle E_{3,2}, E_{2,2} \rangle,
\langle E_{1,3}, E_{3,3} \rangle
\Big\}
\]
\[
\Vsupp\!\left(p\big|_{X^{(4)}}\right)
=
\Big\{
\begin{aligned}
&\langle E_{1,0}, E_{0,0} \rangle,
 \langle E_{3,1}, E_{0,1} \rangle,
 \langle E_{3,1}, E_{3,1} \rangle,\\
&\langle E_{0,2}, E_{1,2} \rangle,
 \langle E_{0,2}, E_{2,2} \rangle,
 \langle E_{2,2}, E_{1,2} \rangle,\\
&\langle E_{2,2}, E_{2,2} \rangle,
 \langle E_{3,3}, E_{2,3} \rangle
\end{aligned}
\Big\}
\]
\[
\Vsupp\!\left(p\big|_{X^{(5)}}\right)
=
\Big\{
\langle E_{2,0}, E_{0,0} \rangle,
\langle E_{1,1}, E_{1,1} \rangle,
\langle E_{3,2}, E_{2,2} \rangle,
\langle E_{0,3}, E_{3,3} \rangle
\Big\}.
\]
Given 
$$
q^{(1)}=\langle \begin{pmatrix}%
\alpha_1 & 0 & 0 & 0  \\
\beta_1 & 0 & 0 & 0 \\
0 & \gamma_1 & 0 & \epsilon_1   \\
0 & 0 & \delta_1 & 0 
\end{pmatrix}, \begin{pmatrix}%
\alpha_1+\beta_1 & 0 & 0 & 0  \\
0 & \gamma_1 & 0 & 0 \\
0 & 0 & \delta_1 & 0   \\
0 & 0 & 0 & \epsilon_1 
\end{pmatrix} \rangle  \in \Conv(\Vsupp\!\left(p\big|_{X^{(1)}}\right)),
$$
$$
q^{(2)}=\langle \begin{pmatrix}%
0 & 0 & \gamma_2 & 0  \\
0 & 0 & \delta_2 & 0 \\
0 & 0 & \epsilon_2  & 0 \\
\alpha_2 & \beta_2 & 0 & \lambda_2  
\end{pmatrix}, \begin{pmatrix}%
\alpha_2 & 0 & 0 & 0  \\
0 & \beta_2 & 0 & 0 \\
0 & 0 & \gamma_2 +\delta_2+ \epsilon_2 & 0   \\
0 & 0 & 0 & \lambda_2 
\end{pmatrix} \rangle  \in \Conv(\Vsupp\!\left(p\big|_{X^{(2)}}\right)),
$$
$$
q^{(3)}=\langle \begin{pmatrix}%
0 & \beta_3 & 0 & 0  \\
0 & 0 & 0 & \epsilon_3 \\
0 & 0 & \gamma_3  & 0 \\
\alpha_3 & 0 & \delta_3 & 0 
\end{pmatrix}, \begin{pmatrix}%
\alpha_3 & 0 & 0 & 0  \\
0 & \beta_3 & 0 & 0 \\
0 & 0 & \gamma_3 +\delta_3 & 0   \\
0 & 0 & 0 & \epsilon_3 
\end{pmatrix} \rangle  \in \Conv(\Vsupp\!\left(p\big|_{X^{(3)}}\right)),
$$
$$
q^{(4)}=\langle \begin{pmatrix}%
0 &  0 & \delta_4+\epsilon_4 & 0  \\
\alpha_4 & 0 & 0 & 0 \\
0 & 0 & \lambda_4+\mu_4  & 0  \\
0 & \beta_4+\gamma_4 & 0 & \kappa_4  
\end{pmatrix}, \begin{pmatrix}%
\alpha_4 & \beta_4 & 0 & 0  \\
0 & 0 & \delta_4 +\lambda_4 & 0 \\
0 & 0 & \epsilon_4 + \mu_4 & \kappa_4    \\
0 & \gamma_4  & 0 & 0
\end{pmatrix} \rangle  \in \Conv(\Vsupp\!\left(p\big|_{X^{(4)}}\right)),
$$
$$
q^{(5)}=\langle \begin{pmatrix}%
0 & 0 & 0 & \delta_5  \\
0 & \beta_5 & 0 & 0 \\
\alpha_5 & 0 & 0  & 0 \\
0 & 0 & \gamma_5 & 0 
\end{pmatrix}, \begin{pmatrix}%
\alpha_5 & 0 & 0 & 0  \\
0 & \beta_5 & 0 & 0 \\
0 & 0 & \gamma_5 & 0   \\
0 & 0 & 0 & \delta_5 
\end{pmatrix} \rangle  \in \Conv(\Vsupp\!\left(p\big|_{X^{(5)}}\right)),
$$
such that the restrictions on the nodes $x_1$ and $x_2$ coincide; that is, {for every $0\leq k \leq 3$ and $1 \leq i,j \leq 5$, 
the sum the $k$ row of the first matrix in $q^{(i)}$ is equal to the sum of the $k$ row of the first matrix in $q^{(j)}$. Same thing happens for the second matrices.} 
%\ak{or we can write the old sentence without adding unnecessary indices:  "the sums of the corresponding
%rows of the matrices on the left are equal, and the sums of the corresponding rows of the matrices
%on the right are also equal."}
%the sums of the corresponding rows of the matrices on the left are equal, and the sums of the corresponding rows of the matrices on the right are also equal. 
Hence, we obtain the following equations:
$$
\begin{aligned}
\alpha_1
&= \gamma_2
= \beta_3
= \delta_4 + \epsilon_4
= \delta_5, \\[4pt]
\beta_1
&= \delta_2
= \epsilon_3
= \alpha_4
= \beta_5, \\[4pt]
\gamma_1 + \epsilon_1
&= \epsilon_2
= \gamma_3
= \lambda_4 + \mu_4
= \alpha_5, \\[4pt]
\delta_1
&= \alpha_2 + \beta_2 + \lambda_2
= \alpha_3 + \delta_3
= \beta_4 + \gamma_4 + \kappa_4
= \gamma_5,
\end{aligned}
$$
and 
$$
\begin{aligned}
\alpha_1+\beta_1
&= \alpha_2
= \alpha_3
= \alpha_4 + \beta_4
= \alpha_5, \\[4pt]
\gamma_1
&= \beta_2
= \beta_3
= \delta_4+\lambda_4
= \beta_5, \\[4pt]
\delta_1
&= \gamma_2 +\delta_2 +\epsilon_2
=\gamma_3 +\delta_3 
= \epsilon_4 + \mu_4+\kappa_4
= \gamma_5, \\[4pt]
\epsilon_1
&=\lambda_2
= \epsilon_3
= \gamma_4 
=\delta_5,
\end{aligned}
$$
where $\alpha_1+\beta_1+\gamma_1+\delta_1+\epsilon_1=1$. This leads to 
%the following:
%
$$
\begin{aligned}
\alpha_1=&\beta_1=\gamma_1=\epsilon_1=\beta_2=\gamma_2=\delta_2=\lambda_2=\beta_3=\epsilon_3=\alpha_4=\beta_4=\gamma_4=\beta_5=\delta_5=\frac{1}{8},  \\
& \alpha_2=\epsilon_2=\alpha_3=\gamma_3=\delta_3=\kappa_4=\alpha_5=\frac{1}{4}, \;\; \delta_1=\gamma_5=\frac{1}{2}, \\
& \delta_4 + \epsilon_4=\frac{1}{8}, \;\; \lambda_4 + \mu_4 =\frac{1}{4}, \;\; \delta_4+\lambda_4=\frac{1}{8}, 
\;\; \epsilon_4+ \mu_4=\frac{1}{4}. \\
\end{aligned}
$$
We obtain that $(q^{(1)},q^{(2)},q^{(3)},q^{(4)},q^{(5)})=p$. Hence, by Corollary~\ref{cor:insteadoflimit}, $p$ is a vertex. 
%Anyway, 
Note that
$$
A_4
=
\set{\, q|_{\set{x_1,x_2}} :\; q \in 
\Vsupp\!\left(p\big|_{X^{(4)}}\right) \,}
=
\left\lbrace
\begin{aligned}
&(e_1,e_0)^T,\ (e_3,e_0)^T,\ (e_3,e_3)^T,\ (e_0,e_1)^T,\ (e_0,e_2)^T,\\
&(e_2,e_1)^T,\ (e_2,e_2)^T,\ (e_3,e_2)^T,\ (e_2,e_3)^T
\end{aligned}
\right\rbrace.
$$
We then have the following linear dependence among the elements of $A_4$:
$$
\begin{aligned}
&0\cdot (e_1,e_0)^T
+ 0\cdot (e_3,e_0)^T
+ 0\cdot (e_3,e_3)^T
+ 1\cdot (e_0,e_1)^T \\
&\quad - 1\cdot (e_0,e_2)^T
- 1\cdot (e_2,e_1)^T
+ 1\cdot (e_2,e_2)^T
+ 0\cdot (e_3,e_2)^T
+ 0\cdot (e_2,e_3)^T
= 0
\end{aligned}
$$
\end{ex}
%
%
%\coc{Better to avoid $\vec{0}$ notation. I prefer $0$.}

%

%\section{Vertices on gluing of cycle scenarios}\label{sec:GluCycles}
%

%\begin{cihan}{purple}{to be revised}
%In this section, we study vertices on two types of measurement scenarios:  
%1)  the wedge of $1$-circles, and  
%2)  the gluing of edges through two points.  
%In both cases, we characterize vertices in terms of sets of vectors in the cube satisfying specific geometric properties.
%We then reformulate these geometric properties in algebraic and graph-theoretic terms.

%\end{cihan}

\subsection{Rose graph} 
\label{subsec:wedgeofqcycles}

%{The rose graph provides a fundamental example of a one-dimensional measurement space obtained by gluing several cycles along a single common node. In this section, we formulate the corresponding special cases of results from Section~\ref{sec:Glumeas}. In particular, using the description of the vertices on a cycle (Example~\ref{ex:cycle}), we derive a sufficient condition for the extremality of simplicial distributions on the rose graph in terms of geometric properties of certain associated vectors, called average points. Finally, we present an example showing that this condition is sufficient but not necessary.}

The rose graph is a one-dimensional measurement space obtained by gluing circles along a single common node. Using the known classification of vertices on a single circle, we derive a sufficient condition for extremality of distributions on this graph. We also show by example that this condition is not necessary. In the next subsection, we sharpen these results to obtain a complete characterization of vertices.

\begin{defn}\label{def:rosegra}
For $n\geq 1$, we define the simplicial set $R_n$ as the wedge product $\vee_{i=1}^n C^{(1)}$ of $n$ copies of the $1$-circle consisting of the generating $1$-simplex $\sigma_i$ for each $1\leq i \leq n$ and the unique node $\ast$. This underlying graph of this simplicial set will be referred to as the \emph{rose graph}, and will be denoted by $R_n$ as well. See Figure \ref{fig:rose-r4}.
\end{defn}

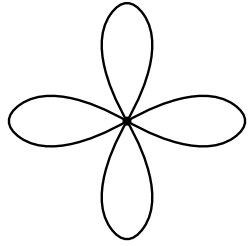
\begin{figure}[ht]
\centering
\begin{tikzpicture}[x=1cm,y=1cm]
  \coordinate (O) at (0,0);

  % top petal
  \draw[line width=0.9pt]
    (O) .. controls (-0.08,0.14) and (-0.55,0.95) .. (-0.22,1.42)
        .. controls (-0.08,1.62) and ( 0.08,1.62) .. ( 0.22,1.42)
        .. controls ( 0.55,0.95) and ( 0.08,0.14) .. (O);

  % right petal
  \draw[line width=0.9pt]
    (O) .. controls (0.14, 0.08) and (0.95, 0.55) .. (1.42, 0.22)
        .. controls (1.62, 0.08) and (1.62,-0.08) .. (1.42,-0.22)
        .. controls (0.95,-0.55) and (0.14,-0.08) .. (O);

  % bottom petal
  \draw[line width=0.9pt]
    (O) .. controls ( 0.08,-0.14) and ( 0.55,-0.95) .. ( 0.22,-1.42)
        .. controls ( 0.08,-1.62) and (-0.08,-1.62) .. (-0.22,-1.42)
        .. controls (-0.55,-0.95) and (-0.08,-0.14) .. (O);

  % left petal
  \draw[line width=0.9pt]
    (O) .. controls (-0.14,-0.08) and (-0.95,-0.55) .. (-1.42,-0.22)
        .. controls (-1.62,-0.08) and (-1.62, 0.08) .. (-1.42, 0.22)
        .. controls (-0.95, 0.55) and (-0.14, 0.08) .. (O);

  \node[circle,fill=black,inner sep=1.2pt] at (O) {};
\end{tikzpicture}
\caption{The rose graph \(R_4\), obtained by gluing four circles at a single vertex.}
\label{fig:rose-r4}
\end{figure}

To motivate our approach we begin with two examples. First we consider the case of two outcomes, where the vertices are well-understood. Then we consider the case of three outcomes, where we will see that the vertex structure is more complex.

\begin{ex}\label{ex:contverwedgez2}
The contextual vertices of $\Dist(R_n,{2})$ are precisely those simplicial distributions $p$ such that 
the restriction to each individual circle is either $p_{+}$ or $p_{-}$ (see Equation~\eqref{eq:p +-}), 
and at least one circle is assigned $p_{-}$. This result is proved in \cite[Part (2) of Proposition 12]{kharoof2023topological}.
\end{ex}

%The polytope $\sDist(\vee_{i=1}^n C^{(1)},\Delta_{\ZZ_2})$ is described in \cite[Proposition 11]{kharoof2023topological}. 
%In particular, the contextual vertices are characterized as follows:

%\begin{pro}\label{pro:contverwedgez2}
%(\!\!\cite[Part (2) of Proposition 12]{kharoof2023topological})
%The contextual vertices of 
%$\sDist(\vee_{i=1}^n C^{(1)},\Delta_{\ZZ_2})$ are precisely those simplicial distributions $p$ %such that 
%the restriction to each individual circle is either $p_{+}$ or $p_{-}$ (see Equation~\eqref{eq:p +-}), 
%and at least one circle is assigned $p_{-}$.%
%\end{pro}

Therefore, for two outcomes, {any} contextual vertex on the wedge of circles {must} exhibit a cycle distribution (Definition~\ref{ex:cycle}) on at least one of the circles. As we illustrate in the following example, this pattern does not hold when the number of outcomes is greater than two.
We illustrate this with the following example.
%This result shows that, when there are only two outcomes, any contextual vertex on the wedge of circles must exhibit 
%a cycle distribution (Definition~\ref{Def:CccyclicDis}) on at least one of the circles. 
%However, this pattern no longer holds when the number of outcomes is greater than two. 
%We illustrate this with the following example:

%
%
\begin{ex}\label{ex:Firstex}
Let 
%$p \colon C^{(1)}\vee C^{(1)} \to \Delta_{\ZZ_3}$ 
$p \colon R_2 \to D(\Delta_{\ZZ_3})$ 
be defined by
$$
p_{\sigma_1}=\begin{pmatrix}%
0 & \frac{1}{3} & 0  \\
\frac{1}{3} & 0 & 0 \\
0 & 0 & \frac{1}{3} \\  
\end{pmatrix},~~%
p_{\sigma_2}=\begin{pmatrix}%
0 & 0 & \frac{1}{3} \\
0 & \frac{1}{3} & 0 \\
\frac{1}{3} & 0 & 0  \\ 
\end{pmatrix}.
$$
We will show that this is a contextual vertex using Corollary \ref{cor:insteadoflimit}. First, we examine the vertex supports:
$$
\Vsupp(p_{\sigma_1})=\set{[0,1],[2]}, \;\; \Vsupp(p_{\sigma_2})=\set{[0,2],[1]}.
$$
Now, suppose we take convex combinations 
$\alpha [0,1]+(1-\alpha)[2]$ and $\beta[0,2]+(1-\beta)[1]$, and assume they coincide at the shared node (i.e., they match at the unique vertex $\ast$). This implies that the following equation holds: 
$$
\begin{pmatrix}
\frac{\alpha}{2} \\
\frac{\alpha}{2} \\
1-\alpha
\end{pmatrix}
=
\begin{pmatrix}
\frac{\beta}{2} \\
1-\beta \\
\frac{\beta}{2}
\end{pmatrix}.
$$
So $\beta=\alpha=2-2\beta$, which means that $\alpha=\beta=\frac{2}{3}$. Note that 
$$
\frac{2}{3} [0,1]+ \frac{1}{3} [2]=\begin{pmatrix}%
0 & \frac{1}{3} & 0  \\
\frac{1}{3} & 0 & 0 \\
0 & 0 & \frac{1}{3} \\  
\end{pmatrix}, \;
\text{and} 
\;\; 
\frac{2}{3}[0,2]+\frac{1}{3}[1]=\begin{pmatrix}%
0 & 0 & \frac{1}{3} \\
0 & \frac{1}{3} & 0 \\
\frac{1}{3} & 0 & 0  \\ 
\end{pmatrix}.
$$
This confirms that the convex combinations match with the original definitions of $p_{\sigma_1}$ and $p_{\sigma_2}$, respectively. 
Thus by Corollary~\ref{cor:insteadoflimit},
$p$ is a contextual vertex, even though it does not restrict to a cycle distribution on any of the two circles.
\end{ex}

To handle cases with more than two outcomes, we introduce the following definition.

\begin{defn}
An \emph{average point} in the $m$-cube is a uniform average of some of the following vertices of the $m$-cube:
$$
\begin{pmatrix}
1\\
0\\
0 \\
\vdots \\
0 \\
0
\end{pmatrix},
 \begin{pmatrix}
0 \\
1\\
0 \\
\vdots \\
0 \\
0
\end{pmatrix},
 \dots, 
 \begin{pmatrix}
0\\
0 \\
\vdots \\
0 \\
1 \\
0
\end{pmatrix} 
 ,
 \begin{pmatrix}
0\\
0 \\
\vdots \\
0 \\
0 \\
1
\end{pmatrix}.
$$
%We will denote the set of average points in the $m$-cube by $\Av(m)$. 
\end{defn}
\begin{ex}
The following points are average points in the $6$-cube:
$$
\begin{pmatrix}
1/4\\
1/4\\
0 \\
1/4\\
0 \\
1/4
\end{pmatrix},
 \begin{pmatrix}
0\\
0\\
1/2 \\
0 \\
1/2 \\
0
\end{pmatrix},
 \begin{pmatrix}
1/6\\
1/6 \\
1/6 \\
1/6 \\
1/6 \\
1/6
\end{pmatrix} .
$$
\end{ex}

Given a $k$-order cycle distribution on $(C^{(1)},\Delta_{\ZZ_m})$ (equivalently, on $R_1=C^{(1)}$), its restriction to the unique node of the circle is an average point in the $m$-cube. For example, the restriction of $[0,3,2] \in \Dist(C^{(1)},{5})$ to the node of the circle is the following average point in the
$5$-cube: 
$$
\begin{pmatrix}
1/3 \\
0  \\
{1/3} \\
{1/3} \\
0
\end{pmatrix}.
$$

%\begin{defn}
%For a cyclic permutation $\mu \in \cyc(m)$ (see Definition \ref{def:cyc(m)}), we define the average point $v(\mu)$ in the $m$-cube as: 
%$$
%v(\mu)_i=\begin{cases}
%\frac{1}{|O(\mu)|} & \text{if} \; i \in O(\mu) \\
%0   & \text{otherwise}
%\end{cases}
%$$
%where $O(\mu)$ as defined in Corollary \ref{cor:verticesofC1}.
%\end{defn}
%
%\begin{ex}
%For $[0,3,2] \in \sDist(C^{(1)},\Delta_{\ZZ_5})$, the vector $v([0,3,2])$ is 
%
%$$
%\begin{pmatrix}
%1/3 \\
%0  \\
%{1/3} \\
%{1/3} \\
%0
%\end{pmatrix}
%$$
%
%\end{ex}
%
%
%
%
Next, we state the corresponding special cases of Corollaries \ref{cor:VertofGluing} and \ref{cor:VertexAi} to be used in later sections. {In this case, 
$X^{(i)}=C^{(1)}$ for every $1\leq i \leq n$, and $W=\ast$.}  
%

%\coc{better as corollary}

\begin{cor}\label{cor:VertWedge}
Let $p \in \Dist(R_n ,{m})$. For every $1 \leq i \leq m$, we define the set of average points in the $m$-cube 
$A_i=\set{q|_{\ast}: \; q \in \Vsupp(p_{\sigma_i})}$. If these sets satisfy
\begin{itemize}
    \item for every \(1 \leq i \leq n\), $|A_i|=|\Vsupp(p_{\sigma_i})|$ and the set \(A_i\) is affinely independent; 
    \item the intersection \(\bigcap_{i=1}^n \Conv(A_i)\) consists of a single point;
\end{itemize}
then $p$ is a vertex.
\end{cor}
%
%

%\coc{better as corollary}

\begin{cor}\label{cor:VertWedgee}
Given sets $A_1,\dots,A_{n}$ of average points in the $m$-cube, satisfying
\begin{itemize}
    \item for every $1 \leq i \leq n$, the set $A_i$ is affinely independent;
    \item the intersection $\cap_{i=1}^n \Conv A_i$ consists of a single point;
    \item for every $1 \leq i \leq n$, there exists a face generating $\tilde{A}_i$ contained in $\Dist(C^{(1)},{m})$, such that $s_i^*|_{\tilde{A}_i}$ is a bijection $\tilde{A}_i\to A_i$. 
\end{itemize}
We define the simplicial distribution
$p \colon R_n\to D(\Delta_{\ZZ_m})$ by specifying its restriction to the $j$\textsuperscript{th} circle as
$$
\sum_{q \in \tilde{A}_j }\alpha_{q|_{\ast}} q,
$$
where $\alpha_{q|_{\ast}}$ is the coefficient of $q|_{\ast}$ in the convex combination that represents the unique element in $\cap_{i=1}^n \Conv A_i$, expressed using the elements of $A_j$. Then $p$ is a vertex.
\end{cor}
%
%
%
%
%In the next example we show how Proposition \ref{pro:VertWedgee} is helpful to produce vertices in the scenario of wedge of cycles.  
\begin{ex}
Each of the following sets is an affinely independent set of average points in the $5$-cube: 
$$
A_1=\set{\begin{pmatrix}
1/2\\
1/2\\
0 \\
0 \\
0 
\end{pmatrix},
 \begin{pmatrix}
1/3\\
0\\
1/3 \\
0 \\
1/3 
\end{pmatrix},
 \begin{pmatrix}
0\\
0 \\
1/2 \\
1/2 \\
0
\end{pmatrix} } \;\; \text{and}\;\;  
A_2=\set{\begin{pmatrix}
1/3\\
1/3\\
1/3 \\
0 \\
0 
\end{pmatrix},
 \begin{pmatrix}
1/3\\
0\\
1/3 \\
1/3 \\
0 
\end{pmatrix},
 \begin{pmatrix}
1/5\\
1/5 \\
1/5 \\
1/5 \\
1/5
\end{pmatrix} }.
$$
Moreover, the intersection of $\Conv(A_1)$ and $\Conv(A_2)$ contains exactly one point:
$$
\frac{4}{11}\begin{pmatrix}
1/2\\
1/2\\
0 \\
0 \\
0 
\end{pmatrix}+
\frac{3}{11} \begin{pmatrix}
1/3\\
0\\
1/3 \\
0 \\
1/3 
\end{pmatrix}+
\frac{4}{11} \begin{pmatrix}
0\\
0 \\
1/2 \\
1/2 \\
0
\end{pmatrix} =   
\frac{3}{11} \begin{pmatrix}
1/3\\
1/3\\
1/3 \\
0 \\
0 
\end{pmatrix}+
\frac{3}{11} \begin{pmatrix}
1/3\\
0\\
1/3 \\
1/3 \\
0 
\end{pmatrix}+
\frac{5}{11} \begin{pmatrix}
1/5\\
1/5 \\
1/5 \\
1/5 \\
1/5
\end{pmatrix}.
$$
We define 
$$
\tilde{A}_1=\set{[0,1], [0,2,4],[2,3] } \;\; \text{and}\;\;  
\tilde{A}_2=\set{[0,1,2],[0,2,3],[0,1,2,3,4] }.
$$
Both 
generate faces
%are closed sets of vertices 
in $\Dist(C^{(1)},{5})$, 
%such that 
%the restrictions to the \co{unique} node 
%of the circle 
%induce
and we have induced bijections $\tilde{A_1}\to A_1$ and $\tilde{A_2}\to A_2$.
By Corollary \ref{cor:VertWedgee}, the simplicial distribution $p\colon R_2 \to D(\Delta_{\ZZ_5})$ defined by 
$$
p_{\sigma_1}=\frac{4}{11}[0,1]+ \frac{3}{11}[0,2,4] + \frac{4}{11}[2,3], \quad p_{\sigma_2}=\frac{3}{11}[0,1,2]+\frac{3}{11}[0,2,3]+\frac{5}{11}[0,1,2,3,4]
$$
is a vertex. 
%Note that
 %$$
%\Vsupp(p_{\sigma_1})=\set{[0,1],[0,2,4],[2,3]} \;\; \text{and} \;\; 
%\Vsupp(p_{\sigma_2})=\set{[0,1,2],[0,2,3],[0,1,2,3,4]}.
%$$
%
%
%
%The sets $A_1$ and $A_2$ can produce $31$ additional vertices by changing the internal order of the cycles. For example:
%$$
%<\frac{4}{11}[0,1]+ \frac{3}{11}[0,4,2] + \frac{4}{11}[2,3],\frac{3}{11}[0,2,1]+\frac{3}{11}[0,2,3]+\frac{5}{11}[0,3,1,4,2]>
%$$
%
\end{ex} 

\begin{rem}
Changing the internal order of the cycle distributions of $\tilde{A}_2$ in the previous example can lead to a simplicial distribution that is not a vertex. For example, if we choose 
$\tilde{A}_2=\set{[0,1,2],[0,2,3],[0,1,3,2,4]}$, and define $p \colon R_2 \to D(\Delta_{\ZZ_5})$ by
$$
p_{\sigma_1} = \frac{4}{11}[0,1] + \frac{3}{11}[0,2,4] + \frac{4}{11}[2,3], \quad
p_{\sigma_2} = \frac{3}{11}[0,1,2] + \frac{3}{11}[0,2,3] + \frac{5}{11}[0,1,3,2,4]
$$
then we cannot apply Corollary \ref{cor:VertWedgee} to conclude that $p$ is a vertex, since the set $\tilde{A}_2$ does not generate a face in $\Dist(C^{(1)},{5})$. This is because $[2,3]$ also appears in $\Vsupp(p_{\sigma_2})$. In fact, we can show that $p$ is not a vertex. The simplicial distribution 
$q \colon R_2 \to D(\Delta_{\ZZ_5})$ defined by
\[
q_{\sigma_1} = [2,3], \quad q_{\sigma_2} = [2,3],
\]
satisfies $q \preceq p$. Therefore, by Corollary \ref{cor:vert=min}, $p$ is not a vertex.
\end{rem}

Here is another example of a vertex that can be explained using Corollary \ref{cor:VertWedgee}.

\begin{ex}\label{ex:Thirdex}
Let $p\colon R_3 \to D(\Delta_{\ZZ_4})$ be defined by
$$
p_{\sigma_1}=\begin{pmatrix}%
\frac{1}{5} & \frac{1}{5} & 0  & 0\\
0 & 0 & \frac{1}{5} & 0 \\
\frac{1}{5} & 0 & 0 & \frac{1}{10} \\
0 & 0 & \frac{1}{10} & 0
\end{pmatrix},~~%
p_{\sigma_2}=\begin{pmatrix}%
0 & 0 & \frac{3}{10} &  \frac{1}{10} \\
0 & \frac{1}{5} & 0 & 0  \\
\frac{3}{10} & 0 & 0 & 0   \\
\frac{1}{10} & 0 & 0 & 0   
\end{pmatrix},~~%
p_{\sigma_3}=\begin{pmatrix}%
0 & \frac{1}{10} &  \frac{3}{10} & 0 \\\frac{1}{10} & 0 & 0 &  \frac{1}{10}  \\
\frac{3}{10} & 0 & 0 & 0   \\
0 & \frac{1}{10} & 0 & 0  
\end{pmatrix}.
$$
Then, $p$ is a contextual vertex.
\end{ex}
We will end this section by giving an example of a vertex that does not satisfy 
the affine independence condition of Corollary~\ref{cor:VertWedge}. We will sharpen this vertex detection result to a full characterization in the next subsection.

\begin{ex}\label{ex:counterex}
Let 
$
p \colon R_3 \to D(\Delta_{\mathbb{Z}_4})
$
be defined by
$$
p_{\sigma_1}=
\begin{pmatrix}
\frac16 & \frac16 & 0 & 0 \\
0 & 0 & \frac16 & \frac16 \\
0 & \frac16 & 0 & 0 \\
\frac16 & 0 & 0 & 0
\end{pmatrix},\qquad
p_{\sigma_2}=
\begin{pmatrix}
0 & \frac13 & 0 & 0 \\
\frac13 & 0 & 0 & 0 \\
0 & 0 & 0 & \frac16 \\
0 & 0 & \frac16 & 0
\end{pmatrix},
$$
$$
p_{\sigma_3}=
\begin{pmatrix}
0 & \frac16 & 0 & \frac16 \\
\frac16 & 0 & \frac16 & 0 \\
\frac16 & 0 & 0 & 0 \\
0 & \frac16 & 0 & 0
\end{pmatrix}.
$$
The vertex supports are given by
$$
\Vsupp(p_{\sigma_1})=\{[0],[0,1,3],[1,2]\},\qquad
\Vsupp(p_{\sigma_2})=\{[0,1],[2,3]\},
$$
$$
\Vsupp(p_{\sigma_3})=\{[0,1],[0,1,2],
[0,3,1],[0,3,1,2]\}.
$$
Suppose
$$
\begin{pmatrix}
\alpha_1 & \frac{\alpha_2}{3} & 0 & 0 \\
0 & 0 & \frac{\alpha_3}{2} & \frac{\alpha_2}{3} \\
0 & \frac{\alpha_3}{2} & 0 & 0 \\
\frac{\alpha_2}{3} & 0 & 0 & 0
\end{pmatrix}
\in \Conv(\Vsupp(p_{\sigma_1})),
$$
$$
\begin{pmatrix}
0 & \frac{\beta_1}{2} & 0 & 0 \\
\frac{\beta_1}{2} & 0 & 0 & 0 \\
0 & 0 & 0 & \frac{\beta_2}{2} \\
0 & 0 & \frac{\beta_2}{2} & 0
\end{pmatrix}
\in \Conv(\Vsupp(p_{\sigma_2})),
$$
$$
\begin{pmatrix}
0 & \frac{\gamma_1}{2}+\frac{\gamma_2}{3} & 0 & \frac{\gamma_3}{3}+\frac{\gamma_4}{4} \\
\frac{\gamma_1}{2}+\frac{\gamma_3}{3} & 0 & \frac{\gamma_2}{3}+\frac{\gamma_4}{4} & 0 \\
\frac{\gamma_2}{3}+\frac{\gamma_4}{4} & 0 & 0 & 0 \\
0 & \frac{\gamma_3}{3}+\frac{\gamma_4}{4} & 0 & 0
\end{pmatrix}
\in \Conv(\Vsupp(p_{\sigma_3})),
$$
with
$$
\alpha_1+\alpha_2+\alpha_3=1,\qquad
\beta_1+\beta_2=1,\qquad
\gamma_1+\gamma_2+\gamma_3+\gamma_4=1.
$$
Assume that the restrictions to unique node $\ast$ coincide, i.e., the sums of the
$i$\textsuperscript{th} rows of the three matrices are equal for each $i=1,2,3,4$.  
Then the following system of equations must hold:
$$
\begin{aligned}
\alpha_1+\frac{\alpha_2}{3}
&= \frac{\beta_1}{2}
 = \frac{\gamma_1}{2}+\frac{\gamma_2}{3}
   +\frac{\gamma_3}{3}+\frac{\gamma_4}{4}, \\[4pt]
\frac{\alpha_3}{2}+\frac{\alpha_2}{3}
&= \frac{\beta_1}{2}
 = \frac{\gamma_1}{2}+\frac{\gamma_3}{3}
   +\frac{\gamma_2}{3}+\frac{\gamma_4}{4}, \\[4pt]
\frac{\alpha_3}{2}
&= \frac{\beta_2}{2}
 = \frac{\gamma_2}{3}+\frac{\gamma_4}{4}, \\[4pt]
\frac{\alpha_2}{3}
&= \frac{\beta_2}{2}
 = \frac{\gamma_3}{3}+\frac{\gamma_4}{4}.
\end{aligned}
$$
This system has a unique solution given by
$$
\begin{aligned}
&\alpha_1=\frac16,\qquad
\alpha_2=\frac12,\qquad
\alpha_3=\frac13,\\[6pt]
&\beta_1=\frac23,\qquad
\beta_2=\frac13,\\[6pt]
&\gamma_1=\frac{1-2x}{3},\qquad
\gamma_2=x,\qquad
\gamma_3=x,\qquad
\gamma_4=\frac{2-4x}{3},
\end{aligned}
$$
where  $0\le x\le \frac12$.
In this case we get that
$$
\frac16[0] + \frac12 [0,1,3]
+ \frac13 [1,2] =p_{\sigma_1}, \; \;
\frac23[0,1] + \frac13 [2,3] =p_{\sigma_2}
$$
$$
\frac{1-2x}{3}[0,1] +
x[0,1,2] + 
x[0,3,1] +
\frac{2-4x}{3} [0,3,1,2]=
\begin{pmatrix}
0 & \frac16 & 0 & \frac16 \\
\frac16 & 0 & \frac16 & 0 \\
\frac16 & 0 & 0 & 0 \\
0 & \frac16 & 0 & 0
\end{pmatrix}=p_{\sigma_3}.
$$
By Corollary~\ref{cor:insteadoflimit}, we conclude that $p$ is vertex. Although, the affine indendence condition fails: We have 
$$
A_3=\set{q|_{\ast}: \; q \in \Vsupp(p_{\sigma_3})} =
\set{ \begin{pmatrix}
1/2\\
1/2 \\
0 \\
0 
\end{pmatrix},\begin{pmatrix}
1/3\\
1/3 \\
1/3 \\
0 
\end{pmatrix},\begin{pmatrix}
1/3\\
1/3 \\
0 \\
1/3 
\end{pmatrix},\begin{pmatrix}
1/4\\
1/4 \\
1/4 \\
1/4 
\end{pmatrix}},
$$
and 
$$
2 \cdot\begin{pmatrix}
1/2\\
1/2 \\
0 \\
0 
\end{pmatrix}
-3 \cdot \begin{pmatrix}
1/3\\
1/3 \\
1/3 \\
0 
\end{pmatrix}
-3\cdot \begin{pmatrix}
1/3\\
1/3 \\
0 \\
1/3 
\end{pmatrix} 
+4 \cdot \begin{pmatrix}
1/4\\
1/4 \\
1/4 \\
1/4 
\end{pmatrix}
=\begin{pmatrix}
0 \\
0 \\
0 \\
0
\end{pmatrix}.
$$
This means that the elements of $A_3$ are affinely dependent. 
%Consequently, the vertex $p$
%does not satisfy the conditions of Proposition~\ref{pro:VertWedge}.
\end{ex}

\subsection{Dipole graph}\label{subsec:dipolegraph}

%{We study vertices in the case where the measurement space is obtained by gluing edges at their endpoints. This space will be called a dipole graph. We characterize its vertices in terms of sets of coordinate vectors in the cube satisfying specific geometric properties. This leads to a similar characterization of vertices on the rose graph discussed in the previous subsection.}

The dipole graph is obtained by gluing edges along both of their
endpoints. We give a complete characterization of the vertices of the polytope of distributions on this graph. We then use the collapsing method to
transfer this characterization to the rose graph.

\begin{defn}\label{defn:dipole}
For $n\geq 1$, we define the simplicial set $D_n$ as the colimit 
$$
D_n:= \Delta^1 \cup_{\Delta^0 \sqcup \Delta^0} \Delta^1 \cup_{\Delta^0 \sqcup \Delta^0} \cdots \cup_{\Delta^0 \sqcup \Delta^0} \Delta^1,
$$
obtained by gluing $n$ copies of $\Delta^1$, whose generating simplices are denoted by $\tau_1,\dots,\tau_n$, along their endpoints. We will refer to $D_n$ as the \emph{dipole graph}; see Figure \ref{fig:dipole-d5}.
\end{defn}

\begin{figure}[ht]
\centering
\begin{tikzpicture}[x=1cm,y=1cm]
  \coordinate (A) at (0,0);
  \coordinate (B) at (3.6,0);

  \draw[line width=0.9pt] (A) to[bend left=60] (B);
  \draw[line width=0.9pt] (A) to[bend left=30] (B);
  \draw[line width=0.9pt] (A) -- (B);
  \draw[line width=0.9pt] (A) to[bend right=30] (B);
  \draw[line width=0.9pt] (A) to[bend right=60] (B);

  \node[circle,fill=black,inner sep=1.2pt] at (A) {};
  \node[circle,fill=black,inner sep=1.2pt] at (B) {};
\end{tikzpicture}
\caption{The dipole graph \(D_5\).}
\label{fig:dipole-d5}
\end{figure}
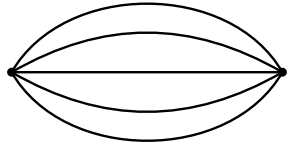

We begin with a particular example studied in \cite[Section 5.3]{kharoof2024extremal}.

%We begin with an example of a vertex that cannot be detected from the wedge.

%
\begin{ex}\label{ex:Forthex}
%Let $X=\left(\sqcup_{i=1}^4 \Delta^1\right) / \sim$ be the measurement space that generated by four edges $\tau_1$, $\tau_2$, $\tau_3$, and $\tau_4$, all glued at their two endpoints. 
The simplicial distribution  $p\colon D_4 \to D(\Delta_{\ZZ_3})$ defined by 
$$
p_{\tau_1}=\begin{pmatrix}%
\frac{1}{4} & 0 & \frac{1}{4} \\
0 & \frac{1}{4} & 0 \\
\frac{1}{4} & 0 & 0 \\
\end{pmatrix},~~%
p_{\tau_2}=\begin{pmatrix}%
\frac{1}{4} & \frac{1}{4} & 0 \\
\frac{1}{4} & 0 & 0 \\
0 & 0 & \frac{1}{4} \\ 
\end{pmatrix}~~%
p_{\tau_3}=\begin{pmatrix}%
\frac{1}{4} & 0 & \frac{1}{4} \\
\frac{1}{4} & 0 & 0 \\
0 & \frac{1}{4} & 0 \\
\end{pmatrix},~~%
p_{\tau_4}=\begin{pmatrix}%
0 & \frac{1}{4} & \frac{1}{4}\\
\frac{1}{4} & 0 & 0 \\
\frac{1}{4} & 0 & 0 \\ 
\end{pmatrix}~~%
$$
is a contextual vertex.
% (see the proof in \cite[Section 5.3]{kharoof2024extremal}). 
%Note that no one of the distributions above is 
%a collapsing distribution (Definition \ref{def:CollapDist}). Therefore, by Proposition \ref{pro:CollapDist1} the given vertex cannot be detected from a wedge of $1$-cycles.
\end{ex}

Our main result is a complete characterization of the vertices of distributions on the dipole graph $D_n$ with arbitrary number of outcomes. This characterization is given in terms of the product-simplex vertices in $\RR^{2m}$, where $m$ is the number of outcomes.

\begin{thm}\label{thm:afinoneintersect}
There is a one-to-one correspondence between the vertices of 
\(\Dist(D_n, {m})\) and the collections of sets \(A_1, A_2, \dots, A_n\) of product-simplex vertices in $\RR^{2m}$, satisfying
\begin{itemize}
    \item for every \(1 \leq i \leq n\), the set \(A_i\) is affinely independent;
    \item the intersection $\bigcap_{i=1}^n \Conv(A_i)$ consists of a single point $u$, and 
    for every $1 \leq i \leq n$, the unique affine representation  
    \[
    u = \sum_{x\in A_i} \alpha_x x
    \]
    has nonzero coefficients $\alpha_x$ for all $x \in A_i$. 
\end{itemize}
Here, for a vertex $p \in \Dist(D_n, {m})$, the associated sets are defined by 
\begin{equation}\label{eq:Adef}
A_i=\set{q|_{\Delta^{0}\sqcup \Delta^{0}} :\; q\in \Vsupp(p_{\tau_i})}=\set{(e_a,e_b)^T:\; p^{ab}_{\tau_i}\neq 0}.    
\end{equation}
\end{thm}
%
%{and $|A_j|=|\Vsupp(p|_{X^{(j)}})|$.} 
\begin{proof}
Consider a vertex $p\in \Dist(D_n, {m})$. For every $1\leq i \leq n$, we choose $A_i$ to be the set of product-simplex vertices 
in $\RR^{2m}$ defined by Equation (\ref{eq:Adef}). 
Since $\Dist(\Delta^1,{m})$ 
is isomorphic to $D(\ZZ_m\times \ZZ_m)$, we conclude that the set $\Vsupp(p_{\tau_i})$ is affinely independent. In addition, obviously $|A_i|=|\Vsupp(p_{\tau_i})|$. Therefore, by part (1) of Corollary \ref{cor:VertofGluihalfsecondir}, the set $A_i$ is affinely independent. By 
part (2) of Corollary \ref{cor:VertofGluihalfsecondir}, the 
intersection $\cap_{i=1}^n \Conv A_i$ consists of a single point. {Furthermore, for each \(1 \leq j \leq n\), this point admits a unique 
affine representation in terms of the elements of \(A_j\), and hence all corresponding coefficients are nonzero.}
Conversely, given a collection of sets $A_1,A_2,\dots,A_n$ of product-simplex vertices in $\RR^{2m}$ satisfying the above conditions, we define 
$$
p\colon D_n \to D(\Delta_{\ZZ_m}),
$$ 
as follows: for each $\tau_i$, let 
$p^{ab}_{\tau_i}$ be the coefficient of 
the product-simplex vertex $(e_a,e_b)^T$ when writing the unique point in 
$\cap_{i=1}^n \Conv A_i$ as a convex combination of the elements in $A_i$. {By the assumption this coefficient is nonzero,}
so $\Vsupp(p_{\tau_i})=\set{E_{a,b}:\; (e_a,e_b)^T \in A_i }$. Therefore
$$
\set{q|_{\Delta^{0}\sqcup \Delta^{0}}: \; q \in \Vsupp(p_{\tau_i})}=A_i.
$$
By Corollary~\ref{cor:VertofGluing}, we obtain that $p$ is a vertex. Alternatively, this follows from Corollary \ref{cor:VertexAi}, where the set $\set{E_{a,b}:\; (e_a,e_b)^T \in A_i }$ serves as $\tilde{A}_i$, which generates a face. 
It is clear that the two processes---moving from a vertex to the corresponding collection of $n$ sets $A_1, A_2, \dots, A_n$ of product-simplex vertices, and the reverse---are inverses of each other.
%Suppose that $\tilde{p}$ is in the limit of Diagram (\ref{dia:limitVsupp})...........
\end{proof}
\begin{ex}\label{ex:simplistforwedgetwopoints}
%Let $X=\Delta^1\sqcup \Delta^1\sqcup \Delta^1\sqcup \Delta^1/\sim$ 
%be the measurement space formed by gluing four edges $\tau_1,\tau_2,\tau_3,\tau_4$, 
%at both of their endpoints. 
Define the sets:  
$$
A_1=\set{\begin{pmatrix}
1 \\
0 \\
0 \\
1 \\
0 \\
0
\end{pmatrix},
\begin{pmatrix}
0 \\
1 \\
0 \\
1 \\
0 \\
0
\end{pmatrix},
\begin{pmatrix}
0 \\
0 \\
1 \\
0 \\
1 \\
0
\end{pmatrix},
\begin{pmatrix}
0 \\
0 \\
1 \\
0 \\
0 \\
1
\end{pmatrix}
}, \;\; A_2=\set{\begin{pmatrix}
1 \\
0 \\
0 \\
0 \\
0 \\
1
\end{pmatrix},
\begin{pmatrix}
0 \\
1 \\
0 \\
1 \\
0 \\
0
\end{pmatrix},
\begin{pmatrix}
0 \\
0 \\
1 \\
1 \\
0 \\
0
\end{pmatrix},
\begin{pmatrix}
0 \\
0 \\
1 \\
0 \\
1 \\
0
\end{pmatrix}
},
$$
$$
A_3=\set{\begin{pmatrix}
1 \\
0 \\
0 \\
1 \\
0 \\
0
\end{pmatrix},
\begin{pmatrix}
0 \\
1 \\
0 \\
0 \\
1 \\
0
\end{pmatrix},
\begin{pmatrix}
0 \\
1 \\
0 \\
0 \\
0 \\
1
\end{pmatrix},
\begin{pmatrix}
0 \\
0 \\
1 \\
1 \\
0 \\
0
\end{pmatrix}
}, \;\; A_4=\set{\begin{pmatrix}
1 \\
0 \\
0 \\
0 \\
1 \\
0
\end{pmatrix},
\begin{pmatrix}
0 \\
1 \\
0 \\
1 \\
0 \\
0
\end{pmatrix},
\begin{pmatrix}
0 \\
1 \\
0 \\
0 \\
0 \\
1
\end{pmatrix},
\begin{pmatrix}
0 \\
0 \\
1 \\
1 \\
0 \\
0
\end{pmatrix}
}.
$$
Every set of the four sets above is an affinely independent set of 
product-simplex vertices in $\RR^{2 \cdot 3}$. Furthermore,
$|\cap_{i=1}^4 \Conv(A_i)|=1$ as shown:
$$
\frac{1}{5}
\begin{pmatrix}
1 \\
0 \\
0 \\
1 \\
0 \\
0
\end{pmatrix}
+ \frac{2}{5}
\begin{pmatrix}
0 \\
1 \\
0 \\
1 \\
0 \\
0
\end{pmatrix}
+ \frac{1}{5}
\begin{pmatrix}
0 \\
0 \\
1 \\
0 \\
1 \\
0
\end{pmatrix} 
+ \frac{1}{5}
\begin{pmatrix}
0 \\
0 \\
1 \\
0 \\
0 \\
1
\end{pmatrix}=\begin{pmatrix}
1/5 \\
2/5 \\
2/5 \\
3/5 \\
1/5 \\
1/5
\end{pmatrix}
=\frac{1}{5}
\begin{pmatrix}
1 \\
0 \\
0 \\
0 \\
0 \\
1
\end{pmatrix}
+ \frac{2}{5}
\begin{pmatrix}
0 \\
1 \\
0 \\
1 \\
0 \\
0
\end{pmatrix}
+ \frac{1}{5}
\begin{pmatrix}
0 \\
0 \\
1 \\
1 \\
0 \\
0
\end{pmatrix} 
+ \frac{1}{5}
\begin{pmatrix}
0 \\
0 \\
1 \\
0 \\
1 \\
0
\end{pmatrix}
$$
$$
\frac{1}{5}\begin{pmatrix}
1 \\
0 \\
0 \\
1 \\
0 \\
0
\end{pmatrix}
+\frac{1}{5}\begin{pmatrix}
0 \\
1 \\
0 \\
0 \\
1 \\
0
\end{pmatrix}
+\frac{1}{5}
\begin{pmatrix}
0 \\
1 \\
0 \\
0 \\
0 \\
1
\end{pmatrix}
+\frac{2}{5}
\begin{pmatrix}
0 \\
0 \\
1 \\
1 \\
0 \\
0
\end{pmatrix}
=\begin{pmatrix}
1/5 \\
2/5 \\
2/5 \\
3/5 \\
1/5 \\
1/5
\end{pmatrix}
=\frac{1}{5}
\begin{pmatrix}
1 \\
0 \\
0 \\
0 \\
1 \\
0
\end{pmatrix}+
\frac{1}{5}
\begin{pmatrix}
0 \\
1 \\
0 \\
1 \\
0 \\
0
\end{pmatrix}+
\frac{1}{5}
\begin{pmatrix}
0 \\
1 \\
0 \\
0 \\
0 \\
1
\end{pmatrix}+
\frac{2}{5}
\begin{pmatrix}
0 \\
0 \\
1 \\
1 \\
0 \\
0
\end{pmatrix}
$$
The corresponding contextual vertex is $p\colon {D_4} \to D(\Delta_{\ZZ_3})$, where:  
$$
p_{\tau_1}=\begin{pmatrix}%
\frac{1}{5} & 0 & 0 \\
 \frac{2}{5} & 0 & 0  \\
0 & \frac{1}{5} &  \frac{1}{5} \\
\end{pmatrix},~~%
p_{\tau_2}=\begin{pmatrix}%
0 & 0 &  \frac{1}{5} \\
\frac{2}{5} & 0 & 0 \\
 \frac{1}{5} &  \frac{1}{5} & 0\\ 
\end{pmatrix},~~%
p_{\tau_3}=\begin{pmatrix}%
\frac{1}{5} & 0 & 0 \\
0 &  \frac{1}{5} & \frac{1}{5} \\
 \frac{2}{5} & 0  & 0 \\
\end{pmatrix},~~%
p_{\tau_4}=\begin{pmatrix}%
0 & \frac{1}{5} & 0 \\
\frac{1}{5} & 0 &  \frac{1}{5} \\
\frac{2}{5} & 0 & 0 \\ 
\end{pmatrix}~~%
$$
\end{ex}

For the vertex that given in Example \ref{ex:Forthex}, the corresponding sets are:
$$
A_1=\set{\begin{pmatrix}
1 \\
0 \\
0 \\
1 \\
0 \\
0
\end{pmatrix},
\begin{pmatrix}
1 \\
0 \\
0 \\
0 \\
0 \\
1
\end{pmatrix},
\begin{pmatrix}
0 \\
1 \\
0 \\
0 \\
1 \\
0
\end{pmatrix}
\begin{pmatrix}
0 \\
0 \\
1 \\
1 \\
0 \\
0
\end{pmatrix}
}, \;
 A_2=\set{\begin{pmatrix}
1 \\
0 \\
0 \\
1 \\
0 \\
0
\end{pmatrix},
\begin{pmatrix}
1 \\
0 \\
0 \\
0 \\
1 \\
0
\end{pmatrix},
\begin{pmatrix}
0 \\
1 \\
0 \\
1 \\
0 \\
0
\end{pmatrix}
\begin{pmatrix}
0 \\
0 \\
1 \\
0 \\
0 \\
1
\end{pmatrix}
}
$$
$$
A_3=\set{\begin{pmatrix}
1 \\
0 \\
0 \\
1 \\
0 \\
0
\end{pmatrix},
\begin{pmatrix}
1 \\
0 \\
0 \\
0 \\
0 \\
1
\end{pmatrix},
\begin{pmatrix}
0 \\
1 \\
0 \\
1 \\
0 \\
0
\end{pmatrix}
\begin{pmatrix}
0 \\
0 \\
1 \\
0 \\
1 \\
0
\end{pmatrix}
}, \;
 A_4=\set{\begin{pmatrix}
1 \\
0 \\
0 \\
0 \\
1 \\
0
\end{pmatrix},
\begin{pmatrix}
1 \\
0 \\
0 \\
0 \\
0 \\
1
\end{pmatrix},
\begin{pmatrix}
0 \\
1 \\
0 \\
1 \\
0 \\
0
\end{pmatrix}
\begin{pmatrix}
0 \\
0 \\
1 \\
1 \\
0 \\
0
\end{pmatrix}
}
$$

We now apply Theorem~\ref{thm:afinoneintersect} to characterize the vertices of the polytope of distributions on the rose graph $R_n$ (see Definition \ref{def:rosegra}). For this, we need an important topological technique used in the theory of simplicial distributions, called the \emph{collapsing method}
(see \cite[Section 4.1]{kharoof2023homotopical}).  We recall this method in the special case of collapsing edges, which is the case relevant to our discussion. See Figures \ref{fig:collpsing K33} and \ref{fig:collpsing2 K33} for examples of such collapsings.

%{To do that we need an important topological technique used in the theory of simplicial distributions, called the collapsing method 
%(see \cite[Section 4.1]{kharoof2023homotopical}). 
%In fact, we need the case where edges are collapsed.}

\begin{defn}\label{def:collapsing}
Let $X$ be a one-dimensional simplicial set, let $\sigma$ be an edge in $X$, {and let 
$s$ be the unique map from $\Delta^1$ to $\Delta^0$}. The quotient $X/\sigma$, obtained by collapsing the edge $\sigma$, is defined as the pushout of the diagram:
%Let $X$ be a $1$-dimensional measurement space, and let $\sigma$ be an edge in $X$. The measurement space $X/\sigma$, obtained by collapsing the edge $\sigma$, is defined as the pushout of the diagram:
%
$$
\begin{tikzcd}[column sep=huge,row sep=large] 
\Delta^1
\arrow[r,"\sigma"] \arrow[d,"s"'] & X \arrow[d] \\
\Delta^0 \arrow[r] & X/\sigma 
\end{tikzcd}
$$
A quotient $X/T$, obtained from $X$ by collapsing a finite set of edges $T$, is called a \emph{collapsing of $X$}. The projection map $\pi \colon X \to X/T$ is called 
the \emph{collapsing map}. Given a collapsing map $\pi: X \to X'$, a simplicial distribution $p\colon X \to D(Y)$ is called a \emph{collapsed simplicial distribution} if it lies in the image of the injective map $\pi^\ast \colon \sDist(X',Y) \to \sDist(X,Y)$ (see Definition \ref{def:pushforward}). 
\end{defn}
%

%For instance, Figure \ref{fig:collpsing K33} describes 

%we collapse the blue edges in the left graph to get a wedge of four $1$-circles.
%
%

%
%\end{ex}
%
%

%
%
%\begin{defn}\label{def:collpseddist}
%Let $\pi\colon X \to X'$ be a collapsing map. A simplicial distribution $p\colon X \to D(Y)$ is called a \emph{collapsed simplicial distribution} if it lies in the image of the injective map
%$\pi^\ast \colon \sDist(X',Y) \to \sDist(X,Y)$ (see Definition \ref{def:pushforward}). 
%\end{defn}
%

When the outcome space is $\Delta_{\ZZ_m}$, we can characterize collapsed simplicial distributions in a simple way.

\begin{defn}\label{def:CollapDist}
A distribution $Q \in D(\ZZ_m^2)$ is called a \emph{collapsed distribution} if $Q^{ab}=0$ whenever $a \neq b$.  
\end{defn}

\begin{pro}\label{pro:CollapDist}
 Let $\pi\colon X \to X'$ be a collapsing map obtained by collapsing $\sigma_1,\dots,\sigma_k$ in $X$.
 \begin{enumerate}
    \item A simplicial distribution $p\colon X' \to D(\Delta_{\ZZ_m})$ is collapsed if and only if $p_{\sigma_i}$ is a collapsed distribution for $1 \leq i \leq k$.
      \item A simplicial distribution $p\colon X' \to D(Y)$ is a contextual vertex if and only if $\pi^\ast(p)$ is a contextual vertex. 
 \end{enumerate} 
\end{pro}

%\begin{pro}\label{pro:CollapDist1}
%Let $X$ be a measurement space, and let $X'$ be the space obtained by collapsing the edges $\sigma_1,\dots,\sigma_k$. Then a simplicial distribution 
%$p\colon X \to D(\Delta_{\ZZ_m})$ is collapsed if and only if each $p_{\sigma_i}$ is a collapsed distribution for $1 \leq i \leq k$.
%\end{pro}
%
%
%\begin{pro}\label{pro:CollapDist2}
%Given a collapsing map $\pi\colon X \to X'$. A simplicial distribution $p\colon X' \to D(Y)$ is a contextual vertex if and only if $\pi^\ast(p)$ is a contextual vertex.   
%\end{pro}
%
\begin{proof}
{Part (1) is clear, and part (2)} follows from \cite[parts $(1)$ and $(3)$ in Theorem 4.4]{kharoof2023homotopical}.    
\end{proof}

%\coc{We may need to point out the relevant reference for part (1) above.}
%\ak{We don't have a reference. However, it is immediate, so I stated "it is clear" in the proof. If you want I can add a short proof.}

%\coc{better as a corollary}

%
\begin{cor}\label{cor:afinoneintersect2222}
There is a one-to-one correspondence between the vertices of 
\(\Dist(R_n, {m})\) and the collections of sets \(A_1, A_2, \dots, A_n\) of product-simplex vertices in $\RR^{2m}$, satisfying 
\begin{itemize}
    \item for every $1 \leq k \leq n$, the set $A_k$ is affinely independent;
    \item the intersection
    \begin{equation}\label{eq:convsum=1}
            \bigcap_{k=1}^n \Conv(A_k)
    \;\cap\;
    \left\{
    \begin{pmatrix}
    \alpha_1 \\
    \vdots \\
    \alpha_m \\
    \rule{0.6em}{0.4pt} \\
    \alpha_1 \\
    \vdots \\
    \alpha_m
    \end{pmatrix}
    :\; \sum_{i=1}^m \alpha_i = 1
    \right\}
        \end{equation} 
    consists of a single point, and for this point the coefficient of every $x \in A_i$ in its affine representation is nonzero.
\end{itemize}
For a vertex $p \in \Dist(R_n, {m})$, the corresponding sets are defined by
\begin{equation}\label{eq:Adef2222}
A_k=\set{(e_a,e_b)^T :\; p^{ab}_{\sigma_k}\neq 0}.
\end{equation}
\end{cor}
\begin{proof}
We define the collapsing map
$$
\pi \colon D_{n+1} \to R_n
$$
by collapsing the edge $\tau_{n+1}$ and sending $\tau_i$ to $\sigma_i$ for every $1\leq i \leq n$. 
By 
%Definition \ref{def:collpseddist} and 
Proposition~\ref{pro:CollapDist} part (2) and {since  $\pi^\ast \colon \Dist(R_{n}, {m}) \to 
\Dist(D_{n+1},{m})$ is injecive}, there is a one-to-one correspondence between vertices 
$p \in \Dist(R_{n},{m})$ and vertices of the form 
$\pi^\ast(p) \in \Dist(D_{n+1}, {m})$. 
By Theorem~\ref{thm:afinoneintersect}, the vertex $\pi^\ast(p)$ corresponds to 
a collection of sets $A_1,\dots,A_{n+1}$ satisfying
\begin{itemize}
    \item for every $1 \leq k \leq n+1$, the set $A_k$ is affinely independent;
    \item
 the intersection $\bigcap_{k=1}^{n+1} \Conv(A_k)$ consists of a single point, 
 {and for this point the coefficient of every $x \in A_i$ in its affine representation is nonzero.}
\end{itemize}
Here
$
A_k=\left\{(e_a,e_b)^T: \; \pi^{\ast}(p)_{\tau_k}^{ab}\neq 0\right\}
$ for each $1 \leq k \leq n+1$. Note that
$$
\pi^\ast(p)_{\tau_k}=p_{\sigma_k} \quad \text{for every } 1\leq k \leq n,
$$
and $\pi^\ast(p)_{\tau_{n+1}}$ is a collapsed distribution (see 
%Definition~\ref{def:CollapDist} and 
{part (1) of} Proposition~\ref{pro:CollapDist}).  
Thus, for every $1 \leq k \leq n$, the set $A_k$ is defined as in \eqref{eq:Adef2222}, while
$$
A_{n+1} \subseteq \set{(e_i,e_i)^T :\; 1\leq i \leq m}.
$$
The set \(A_{n+1}\) is therefore affinely independent, and 
$$
\bigcap_{k=1}^{n+1} \Conv(A_k)= \bigcap_{k=1}^n \Conv(A_k)
    \;\cap\;
    \left\{
    \begin{pmatrix}
    \alpha_1 \\
    \vdots \\
    \alpha_m \\
    \rule{0.6em}{0.4pt} \\
    \alpha_1 \\
    \vdots \\
    \alpha_m
    \end{pmatrix}
    :\; \sum_{i:~(e_i,e_i)^T \in A_{n+1}} \alpha_i = 1
    \right\}.
$$
Since the sets $A_k$ are induced by the same vertex $p$, the index set
$\set{i : \; (e_i,e_i)^T \in A_{n+1}}$ is determined uniquely by any $A_k$. Hence, the final sum may equivalently be written as
$$
\sum_{i=1}^m \alpha_i = 1,
$$
which completes the proof.
\end{proof}

\section{Graph-theoretic characterization of vertices}\label{sec:graphcharac}

{In this section, we use the characterization of vertices on dipole graphs and rose graphs given in the previous section to derive a simpler, equivalent condition for being a vertex on these spaces, using only basic linear algebra and graph theory.}

We begin with the following useful connection between affine independence and linear independence.

\begin{lem}\label{lem:affin=lin}
Let $X\subseteq \{(e_i,e_j)^T : 0 \le i,j \le m-1\}$ be a set of product-simplex vertices in $\mathbb{R}^{2m}$. 
Then the set $X$ is affinely independent if and only if it is linearly independent.    
\end{lem}
\begin{proof}
Firstly, note that linear independence implies affine independence, so we only need to show the converse.
{Now, suppose that the set $X$ is affinely independent, and} let $\{\alpha_v\}_{v\in X}$ be scalars such that
\[
\sum_{v\in X}\alpha_v v = 0.
\]
Since the sum of the coordinates of every vector in $X$ is $2$, we {obtain} $2\sum_v \alpha_v = 0$. By condition (2) of Definition~\ref{def:affinindep}, we get $\alpha_v=0$ for every $v \in X$. 
\end{proof}
%

%\begin{rem}\label{rem:HXXX}
%Given a set $X$ of \co{product-simplex vertices} in the $2m$-cube, we can associate to it a bipartite graph $H(X)$ with \emph{left vertex set} $A$ and \emph{right vertex set} $B$ as subsets of $\{u_0,\dots,u_{m-1}\}$ and $\{w_0,\dots,w_{m-1}\}$, respectively, where there is an edge between $u_i$ and $w_j$ if and only if $(e_i,e_j)^T$ is in the set. This gives a one-to-one correspondence between sets of \co{product-simplex vertices} in the $2m$-cube and such bipartite graphs.
 
%\end{rem}
%\end{cihan}

\begin{defn}\label{def:HXXX}
Let $X\subseteq \{(e_i,e_j)^T : 0 \le i,j \le m-1\}$ be a set of product-simplex vertices in $\RR^{2m}$.
We associate to $X$ a bipartite graph $H(X)$ as follows. The \emph{left vertex set} is
\[
A=\{u_i : \text{there exists } j \text{ such that } (e_i,e_j)^T\in X\},
\]
and the \emph{right vertex set} is
\[
B=\{w_j : \text{there exists } i \text{ such that } (e_i,e_j)^T\in X\}.
\]
There is an edge between $u_i$ and $w_j$ if and only if
$
(e_i,e_j)^T\in X$.
In this way, sets of product-simplex vertices in $\RR^{2m}$ are in one-to-one correspondence with bipartite graphs of this form.
\end{defn}

\begin{pro}\label{pro:affinHXB}
Let $X$ be a set of product-simplex vertices in $\RR^{2m}$. 
Then, $X$ is affinely independent if and only if the bipartite graph $H(X)$ has no cycles. 
\end{pro}
\begin{proof}
%\noindent\emph{($\Rightarrow$)}  
($\implies$)
Assume that $H(X)$ contains a cycle
$$
u_{i_1} \to w_{j_1} \to u_{i_2} \to w_{j_2} \to \dots \to u_{i_n} \to w_{j_n} \to u_{i_1}.
$$
Then
$$
(e_{i_1},e_{j_1})^T,(e_{i_2},e_{j_1})^T,(e_{i_2},e_{j_2})^T, \dots,
(e_{i_n},e_{j_n})^T,(e_{i_1},e_{j_n})^T \in X,
$$
and hence
$$
(e_{i_1},e_{j_1})^T-(e_{i_2},e_{j_1})^T+(e_{i_2},e_{j_2})^T-\cdots
+(e_{i_n},e_{j_n})^T-(e_{i_1},e_{j_n})^T=0,
$$
which shows that the set $X$ is affinely dependent.

%\medskip
%\noindent\emph{($\Leftarrow$)} 
($\impliedby$)
Assume that $H(X)$ has no cycles. Consider a linear relation
$$
\sum_{(e_i,e_j)^T\in X} \lambda_{ij}(e_i,e_j)^T=\vec{0},
\qquad \lambda_{ij}\in\mathbb{R}.
$$
This implies that
\begin{align}
\sum_{j:~(e_i,e_j)^T\in X} \lambda_{ij} &= 0 \quad \text{for all } i,
\label{eq:first}\\
\sum_{i:~(e_i,e_j)^T\in X} \lambda_{ij} &= 0 \quad \text{for all } j.
\label{eq:second}
\end{align}
Since $H(X)$ has no cycles, it has a leaf vertex, i.e., a vertex of degree one. Without loss of generality, suppose $u_i$ is a leaf.
Then, there exists exactly one edge $(u_i,w_j)\in H(X)$ incident to $u_i$, meaning that $(e_i,e_j)^T \in X$ is the unique vector with first coordinate $e_i$.
Equation~\eqref{eq:first} for this $i$ reduces to
$$
\lambda_{ij}=0.
$$
Removing this edge from $H(X)$ and repeating the argument inductively shows that all coefficients $\lambda_{ij}$ vanish.
Hence the only linear relation is the trivial one, and $X$ is linearly (and therefore affinely) independent.
\end{proof}

\begin{lem}\label{lem:spanXUX}
Let $X$ be a set of product-simplex vertices in $\RR^{2m}$ such that the associated
bipartite graph $H(X)$ is a tree.  
Define
$$
I_X := \set{i:\; \exists j\; \text{such that} \; (e_i,e_j)^T\in X},
\qquad
J_X := \set{j:\; \exists i\; \text{such that} \; (e_i,e_j)^T\in X}.
$$
Then, $\operatorname{span}(X)$, the subspace spanned by the vectors in $X$, is equal to  
\begin{equation}\label{eq:UX}
U(X)
:=
\Bigl\{
(\alpha,\beta)^T \in \RR^{2m}
: \;
\sum_{k=0}^{m-1}\alpha_k = \sum_{k=0}^{m-1}\beta_k,\;
\alpha_i = 0 \ \forall i \notin I_X,\;
\beta_j = 0 \ \forall j \notin J_X
\Bigr\}.
\end{equation}
\end{lem}
\begin{proof} 
Every product-simplex vertex in $X$ clearly belongs to $U(X)$, hence
$\operatorname{span}(X) \subseteq U(X)$.
Now, the number of vectors in $X$ equals the number of edges
of the tree $H(X)$, namely $|I_X| + |J_X| - 1$.
By Proposition~\ref{pro:affinHXB} the set $X$ is affinely independent, so by Lemma~\ref{lem:affin=lin} it is linearly independent.
Therefore, 
the dimension of $\operatorname{span}(X)$ is $|I_X|+|J_X|-1$, which is exactly the dimension of $U(X)$. Hence, 
$\operatorname{span}(X) = U(X)$.
\end{proof}

\begin{defn}\label{def:QA1An}
Let $H_1,\dots,H_n$ be bipartite graphs, each with
$\{u_0,\dots,u_{m-1}\}$ as the left vertex set and
$\{w_0,\dots,w_{m-1}\}$ as the right vertex set. Let
$C_1,\dots,C_N$ denote the connected components of the graphs $H_i$,
$i=1,\dots,n$. We define $N \times 2m$ matrix $Q(H_1,\dots,H_n)$ by
$$
Q(H_1,\dots,H_n)_{k,a} =
\begin{cases}
+1 & v_a \in C_{k} \text{ and } 0 \le a \le m-1,\\
-1 & w_{a-m} \in C_{k} \text{ and } m \le a \le 2m-1,\\
0 & \text{otherwise}.
\end{cases}
$$
For sets $A_1,\dots,A_n$ of product-simplex vertices in $\RR^{2m}$, we define 
$$
Q(A_1,\dots,A_n):=Q\left(H(A_1),\dots,H(A_n)\right).
$$
\end{defn}
%

%\coc{figure will be updated to tikzpicture}
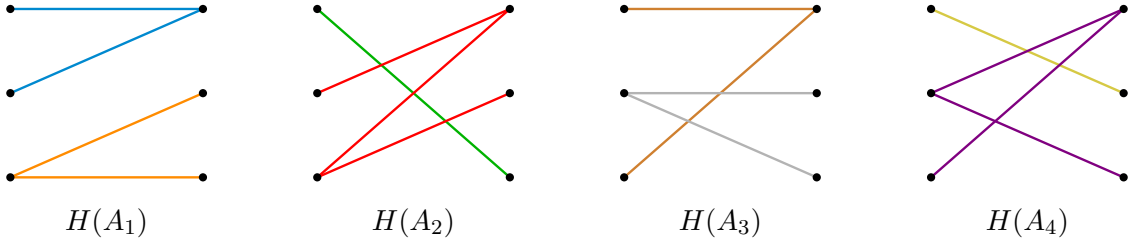
\begin{figure}[ht]
\centering
\begin{tikzpicture}[x=1cm,y=1cm,scale=.8]

% ---------------- H(A_1) ----------------
\begin{scope}[shift={(0,0)}]
  \coordinate (L1) at (0,  1.4);
  \coordinate (L2) at (0,  0);
  \coordinate (L3) at (0, -1.4);
  \coordinate (R1) at (3.2,  1.4);
  \coordinate (R2) at (3.2,  0);
  \coordinate (R3) at (3.2, -1.4);

  \draw[line width=0.9pt, cyan!70!blue] (L1)--(R1);
  \draw[line width=0.9pt, cyan!70!blue] (L2)--(R1);

  \draw[line width=0.9pt, orange!90!yellow] (L3)--(R2);
  \draw[line width=0.9pt, orange!90!yellow] (L3)--(R3);

  \foreach \v in {L1,L2,L3,R1,R2,R3}{
    \node[circle,fill=black,inner sep=1.1pt] at (\v) {};
  }

  \node at (1.6,-2.15) {$H(A_1)$};
\end{scope}

% ---------------- H(A_2) ----------------
\begin{scope}[shift={(5.1,0)}]
  \coordinate (L1) at (0,  1.4);
  \coordinate (L2) at (0,  0);
  \coordinate (L3) at (0, -1.4);
  \coordinate (R1) at (3.2,  1.4);
  \coordinate (R2) at (3.2,  0);
  \coordinate (R3) at (3.2, -1.4);

  \draw[line width=0.9pt, green!70!black] (L1)--(R3);

  \draw[line width=0.9pt, red] (L2)--(R1);
  \draw[line width=0.9pt, red] (L3)--(R1);
  \draw[line width=0.9pt, red] (L3)--(R2);

  \foreach \v in {L1,L2,L3,R1,R2,R3}{
    \node[circle,fill=black,inner sep=1.1pt] at (\v) {};
  }

  \node at (1.6,-2.15) {$H(A_2)$};
\end{scope}

% ---------------- H(A_3) ----------------
\begin{scope}[shift={(10.2,0)}]
  \coordinate (L1) at (0,  1.4);
  \coordinate (L2) at (0,  0);
  \coordinate (L3) at (0, -1.4);
  \coordinate (R1) at (3.2,  1.4);
  \coordinate (R2) at (3.2,  0);
  \coordinate (R3) at (3.2, -1.4);

  \draw[line width=0.9pt, brown!75!orange] (L1)--(R1);
  \draw[line width=0.9pt, brown!75!orange] (L3)--(R1);

  \draw[line width=0.9pt, gray!60] (L2)--(R2);
  \draw[line width=0.9pt, gray!60] (L2)--(R3);

  \foreach \v in {L1,L2,L3,R1,R2,R3}{
    \node[circle,fill=black,inner sep=1.1pt] at (\v) {};
  }

  \node at (1.6,-2.15) {$H(A_3)$};
\end{scope}

% ---------------- H(A_4) ----------------
\begin{scope}[shift={(15.3,0)}]
  \coordinate (L1) at (0,  1.4);
  \coordinate (L2) at (0,  0);
  \coordinate (L3) at (0, -1.4);
  \coordinate (R1) at (3.2,  1.4);
  \coordinate (R2) at (3.2,  0);
  \coordinate (R3) at (3.2, -1.4);

  \draw[line width=0.9pt, yellow!80!black] (L1)--(R2);

  \draw[line width=0.9pt, violet] (L2)--(R1);
  \draw[line width=0.9pt, violet] (L3)--(R1);
  \draw[line width=0.9pt, violet] (L2)--(R3);

  \foreach \v in {L1,L2,L3,R1,R2,R3}{
    \node[circle,fill=black,inner sep=1.1pt] at (\v) {};
  }

  \node at (1.6,-2.15) {$H(A_4)$};
\end{scope}

\end{tikzpicture}
\caption{The bipartite graphs \(H(A_1)\), \(H(A_2)\), \(H(A_3)\), \(H(A_4)\) from Example~\ref{ex:simplistforwedgetwopoints}, with connected components indicated by different colors.}
\label{Fig:Bipartitegrphs1}
\end{figure}

%\begin{figure}[h!] 
%  \centering
%  \includegraphics[width=.7\linewidth]{bipartitegraphs1.png}
%\caption{The bipartite graphs $H(A_1), H(A_2), H(A_3), H(A_4)$ from Example~\ref{ex:simplistforwedgetwopoints}, with connected components indicated by different colors.}
%\label{Fig:Bipartitegrphs1}
%\end{figure}   

\begin{ex}
Consider the sets $A_1,A_2,A_3,A_4$ from
Example~\ref{ex:simplistforwedgetwopoints}. The graphs
\[H(A_1),H(A_2),H(A_3),H(A_4)\] 
are shown in
Figure~\ref{Fig:Bipartitegrphs1}, with their connected components colored
differently. 
%We order the components as follows:
%$$
%C_1=\text{blue}, \quad
%C_2=\text{orange}, \quad
%C_3=\text{green}, \quad
%C_4=\text{red},
%$$
%$$
%C_5=\text{brown}, \quad
%C_6=\text{gray}, \quad
%C_7=\text{yellow}, \quad
%C_8=\text{purple}.
%$$
%
Then 
%$Q(A_1,A_2,A_3,A_4)$ is the following $8\times 6$ matrix:
\begin{equation}\label{eq:matrQA1}
Q(A_1,A_2,A_3,A_4)=\begin{pmatrix}
1 & 1 & 0 & -1 & 0 & 0 \\
0 & 0 & 1 & 0 & -1 & -1 \\
1 & 0 & 0 & 0 & 0 & -1 \\
0 & 1 & 1 & -1 & -1 & 0 \\
1 & 0 & 1 & -1 & 0 & 0 \\
0 & 1 & 0 & 0 & -1 & -1 \\
1 & 0 & 0 & 0 & -1 & 0 \\
0 & 1 & 1 & -1 & 0 & -1
\end{pmatrix}.
\end{equation}
\end{ex}

\begin{pro}\label{pro:spaAiQA}
Let $A_1,\dots,A_n$ be sets of product-simplex vertices in $\RR^{2m}$. Assume that
for every $1 \le i \le n$, 
%the elements of 
 $A_i$ is affinely independent. Then
\begin{equation}\label{eq:capspanAi}
\bigcap_{i=1}^n \operatorname{span}(A_i)=\set{x \in \RR^{2m}: \; Q(A_1,\dots,A_n) x =0 } .
\end{equation}
\end{pro}
\begin{proof}
We first show that for every \(1 \le j \le n\),
$$
\operatorname{span}(A_j)
=
\bigl\{ x \in \RR^{2m} :\; Q(A_j)\,x = 0 \bigr\}.
$$
Let $C_{j_1},\dots,C_{j_k}$ be the connected components of $H(A_j)$, and let
$X_{j_1},\dots,X_{j_k}$ be the corresponding subsets of $A_j$.
By Proposition~\ref{pro:affinHXB}, each $C_{j_\ell}$ is a tree.
By Lemma~\ref{lem:affin=lin}, the set $A_j$ is linearly independent.
Using this fact together with Lemma~\ref{lem:spanXUX}, we obtain
$$
\begin{aligned}
\operatorname{span}(A_j)
&=
\operatorname{span}(X_{j_1}) \oplus \cdots \oplus \operatorname{span}(X_{j_k}) \\
&=
U(X_{j_1}) \oplus \cdots \oplus U(X_{j_k}) \\
&=
\bigl\{ x \in \RR^{2m} :\; Q(A_j)\,x = 0 \bigr\}.
\end{aligned}
$$
See Equation~\eqref{eq:UX}. Observe that
$$
Q(A_1,\dots,A_n)
=
\begin{pmatrix}
Q(A_1) \\
\vdots \\
Q(A_n)
\end{pmatrix}.
$$
Combining these equalities yields Equation~\eqref{eq:capspanAi}.
\end{proof}

\begin{thm}\label{thm:bybipartgraphs}
Let $D_n$ denote the dipole graph with $n$ edges.
Given a simplicial distribution $p\colon D_n \to D(\Delta_{\ZZ_m})$, we define associated bipartite graphs $H_1, H_2, \dots, H_n$, each with
\begin{itemize}
  \item a left vertex set given by a subset of $\{u_0,\dots,u_{m-1}\}$, 
  \item a right vertex set given by a subset of $\{w_0,\dots,w_{m-1}\}$,
  \item an edge between $u_a$ and $w_b$ in $H_i$ if and only if $p_{\tau_i}^{ab} \neq 0$.  
\end{itemize}
Then $p$ is a vertex if and only if the following conditions are satisfied:
\begin{itemize}
    \item For every $1 \leq i \leq n$, the graph of $H_i$ contains no cycle;
    \item the rank of $Q(H_1,\dots,H_n)$ is equal to $2m-1$.
\end{itemize}    
\end{thm}
\begin{proof}
For $1 \le j \le n$, let $A_j$ be defined as in~\eqref{eq:Adef}.
By Theorem~\ref{thm:afinoneintersect}, $p$ is a vertex if and only if the sets
$A_1,\dots,A_n$ satisfy the following conditions:
\begin{itemize}
    \item for every $1 \le i \le n$, 
    %the elements of 
    $A_i$ is affinely independent;
    \item the intersection $\bigcap_{i=1}^n \Conv(A_i)$ consists of a single point,
    and in its affine representation the coefficient of every $x\in A_j$ is nonzero.
\end{itemize}
Note that $H_j$ coincides with $H(A_j)$ as defined in
Definition~\ref{def:HXXX}. By Proposition~\ref{pro:affinHXB}, the first condition
is equivalent to the requirement that each $H_j$ contains no cycle.

By Proposition~\ref{pro:spaAiQA}, the the rank of $Q(H_1,\dots,H_n)$ is equal to $2m-1$ if and only if
the space 
$\bigcap_{i=1}^n \operatorname{span}(A_i)$  is one-dimensional. Observe that
$p|_{\Delta^0 \sqcup \Delta^0} \in \bigcap_{i=1}^n \Conv(A_i)$, and hence
$p|_{\Delta^0 \sqcup \Delta^0} \in \bigcap_{i=1}^n \operatorname{span}(A_i)$.
Moreover, the coefficient of $(e_a,e_b)^T \in A_j$ in the affine representation of
$p|_{\Delta^0 \sqcup \Delta^0}$ is exactly $p_{\tau_j}^{ab}$, which is nonzero
(see Equation~\eqref{eq:Adef}).
Thus, it remains to prove that under the assumption that for every $1 \le j \le n$, the set $A_j$ is affinely independent
$$
\bigcap_{i=1}^n \Conv(A_i)
=
\{p|_{\Delta^0 \sqcup \Delta^0}\}
\quad\Longleftrightarrow\quad
\bigcap_{i=1}^n \operatorname{span}(A_i)
=
\operatorname{span}\!\bigl(p|_{\Delta^0 \sqcup \Delta^0}\bigr).
$$

Assume first that
$
\bigcap_{i=1}^n \operatorname{span}(A_i)
=
\operatorname{span}\!\bigl(p|_{\Delta^0 \sqcup \Delta^0}\bigr)
$. 
If $q \in \bigcap_{i=1}^n \Conv(A_i)$, then
$q \in \bigcap_{i=1}^n \operatorname{span}(A_i)$,
and hence $q$ is a scalar multiple of $p|_{\Delta^0 \sqcup \Delta^0}$.
Since every vector in $\bigcap_{i=1}^n \Conv(A_i)$ lies in
$D(\ZZ_m)\times D(\ZZ_m)$ and therefore has total coordinate sum equal to $2$. This implies that the scalar multiple is $1$, and hence 
%it follows that 
$q=p|_{\Delta^0 \sqcup \Delta^0}$.

Conversely, suppose there exists a nonzero
$\tilde{x} \in \bigcap_{i=1}^n \operatorname{span}(A_i)$ that is not a scalar multiple of
$p|_{\Delta^0 \sqcup \Delta^0}$.
For any $x\in  \bigcap_{i=1}^n \operatorname{span}(A_i)$, $1 \le j \le n$, and $a,b\in\ZZ_m$, denote by 
$
\gamma^{a,b}_{j,x}
$
the coefficient of $(e_a,e_b)^T$ in the representation of $x$ as a linear
combination of the elements of $A_j$. In particular,
$$
\gamma^{a,b}_{j,p|_{\Delta^0 \sqcup \Delta^0}}
=
p_{\tau_j}^{a,b}
\ge 0,
$$
with strict inequality whenever $(e_a,e_b)^T \in A_j$.
For every $1 \leq j \leq n$, there exists $t_j>0$ such that
$$
\gamma^{a,b}_{j,\tilde{x}}+t_j\gamma^{a,b}_{j,p|_{\Delta^0 \sqcup \Delta^0}} \ge 0,
$$
with strict inequality whenever
$\gamma^{a,b}_{j,p|_{\Delta^0 \sqcup \Delta^0}}>0$.
Let $t=\max\{t_1,\dots,t_n\}$ and define
$$
z:=\tilde{x}+t\,p|_{\Delta^0 \sqcup \Delta^0}.
$$
Then $z \in \bigcap_{i=1}^n \operatorname{span}(A_i)$, and we have
$$
\gamma^{a,b}_{j,z}
=
\gamma^{a,b}_{j,\tilde{x}}
+
t\,\gamma^{a,b}_{j,p|_{\Delta^0 \sqcup \Delta^0}}
\ge 0.
$$
Write $z=(\alpha,\beta)^T$, so that
$\sum_{i=1}^m \alpha_i=\sum_{i=1}^m \beta_i$.
Define
$$
z':=\frac{1}{\sum_{i=1}^m \alpha_i}\,z.
$$
Then
$$
z'
=
\sum_{(e_a,e_b)^T\in A_j}
\frac{\gamma^{a,b}_{j,z}}{\sum_{i=1}^m \alpha_i}
\,(e_a,e_b)^T,
$$
and we have 
$$
\sum_{(e_a,e_b)^T \in A_j} \frac{\gamma^{a,b}_{j,z}}{\sum_{i=1}^m \alpha_i}=1,
$$
since 
$$
\sum_{(e_a,e_b)^T \in A_j} \gamma^{a,b}_{j,z}= \sum_{i=0}^{m-1}\sum_{b:\, (e_i,e_b)^T \in A_j} \gamma^{i,b}_{j,z}=
\sum_{i=0}^{m-1}\alpha_{{i}}.
$$
Hence $z' \in \bigcap_{i=1}^n \Conv(A_i)$.
Since $\tilde{x}$ is not a scalar multiple of $p|_{\Delta^0 \sqcup \Delta^0}$, neither is $z'$.
Thus, $\bigcap_{i=1}^n \Conv(A_i)$ contains more than one point, completing the proof. 
\end{proof}
\begin{ex}
In the bipartite graphs shown in Figure~\ref{Fig:Bipartitegrphs1}, there are no
cycles. Moreover, the rank of the matrix in Equation~\eqref{eq:matrQA1} is $5$.
Therefore, by Theorem~\ref{thm:bybipartgraphs}, the simplicial distribution $p$
from Example~\ref{ex:simplistforwedgetwopoints} is a vertex.
\end{ex}

\begin{defn}\label{def:tildeQA1An}
Let $H_1,\dots, H_n$ be bipartite graphs, each with $\{u_0,\dots,u_{m-1}\}$ as the left vertex set and
$\{w_0,\dots,w_{m-1}\}$ as the right vertex set.
Let $C_{1},\dots,C_{N}$ denote the connected components of the graphs $H_i$ as $i = 1, \dots, n$.
We define the $(N+m) \times 2m$ matrix $\tilde{Q}(H_1,\dots,H_n)$ by 
$$
\tilde{Q}(H_1,\dots,H_n)_{k,a} =
\begin{cases}
+1 & 1\leq k \leq N, \; v_a \in C_{k}, \text{ and } 0 \leq a \leq m-1,\\
-1 & 1\leq k \leq N, \; w_{a-m} \in C_{k} \text{ and } m \leq a \leq 2m-1,\\
+1  & N+1\leq k \leq N+M \text{ and } a=k-N ,\\
-1  & N+1\leq k \leq N+M \text{ and } a=k-N+m, \\
0 & \text{otherwise}.
\end{cases}
$$
For sets $A_1,\dots,A_n$ of product-simplex vertices in $\RR^{2m}$, we define 
$$
\tilde{Q}(A_1,\dots,A_n):=\tilde{Q}\left(H(A_1),\dots,H(A_n)\right).
$$
\end{defn}
Note that 
\begin{equation}\label{eq:QHtildeQH}
\tilde{Q}(H_1,\dots,H_n)
=
\begin{pmatrix}
Q(H_1,\dots,H_n) \\
I_m\;\;\;\;\;\;-I_m
\end{pmatrix}.
\end{equation}
See Definition \ref{def:QA1An}.
%

%\coc{stated as a corollary}

\begin{cor}\label{cor:spaAitildeQA}
Let $A_1,\dots,A_n$ be sets of product-simplex vertices in $\RR^{2m}$. Assume that
for every $1 \le i \le n$, 
%the elements of 
$A_i$ is affinely independent. Then
\begin{equation}\label{eq:spanintersecttildeQ}
\bigcap_{i=1}^n \operatorname{span}(A_i)\;\cap\;
    \left\{
    \begin{pmatrix}
    \alpha_1 \\
    \vdots \\
    \alpha_m \\
    \rule{0.6em}{0.4pt} \\
    \alpha_1 \\
    \vdots \\
    \alpha_m
    \end{pmatrix}
    :\; \sum_{i=1}^m \alpha_i = 1
    \right\}=\set{x \in \RR^{2m}: \; \tilde{Q}(A_1,\dots,A_n) x =\vec{0} } 
\end{equation}
\end{cor}

\begin{proof}
Equation (\ref{eq:spanintersecttildeQ}) follows from Proposition \ref{pro:spaAiQA} and Equation (\ref{eq:QHtildeQH}). 
\end{proof}
%

%Now, using Propositions \ref{pro:afinoneintersect2222}, \ref{pro:affinHXB}, and \ref{pro:spaAitildeQA} we can proof the following result similar to the proof of Theorem \ref{thm:bybipartgraphs}.

Now, using Corollary \ref{cor:afinoneintersect2222}, Proposition \ref{pro:affinHXB}, and {Corollary \ref{cor:spaAitildeQA}, we give a graph-theoretic characterization of {extremal simplicial} distributions on rose graphs, analogous to Theorem \ref{thm:bybipartgraphs}.

%\coc{changed to a corollary}

\begin{cor}\label{cor:characgraphs2}
Let $R_n$ denote the rose graph with $n$ circles.
Given a simplicial distribution $p\colon R_n \to D(\Delta_{\ZZ_m})$, we define associated bipartite graphs $H_1, H_2, \dots, H_n$, each with
\begin{itemize}
  \item a left vertex set given by a subset of $\{u_0,\dots,u_{m-1}\}$, 
  \item a right vertex set given by a subset of $\{w_0,\dots,w_{m-1}\}$,
  \item an edge between $u_a$ and $w_b$ in $H_i$ if and only if $p_{\sigma_i}^{ab} \neq 0$.  
\end{itemize}
Then $p$ is a vertex if and only if the following conditions are satisfied:
\begin{itemize}
    \item For every $1 \leq i \leq n$, the graph of $H_i$ contains no cycle;
    \item the rank of $\tilde{Q}(H_1,\dots,H_n)$ is equal to $2m-1$.
\end{itemize}    
\end{cor} 

\begin{figure}[ht]
\centering
\begin{tikzpicture}[x=1cm,y=1cm,scale=.8]

% common coordinates for each panel:
% left side  : y =  2.1,  0.7, -0.7, -2.1
% right side : y =  2.1,  0.7, -0.7, -2.1

% ---------------- first graph ----------------
\begin{scope}[shift={(0,0)}]
  \coordinate (L1) at (0,  2.1);
  \coordinate (L2) at (0,  0.7);
  \coordinate (L3) at (0, -0.7);
  \coordinate (L4) at (0, -2.1);
  \coordinate (R1) at (3.2,  2.1);
  \coordinate (R2) at (3.2,  0.7);
  \coordinate (R3) at (3.2, -0.7);
  \coordinate (R4) at (3.2, -2.1);

  % blue component
  \draw[line width=0.9pt, cyan!70!blue] (L1)--(R1);
  \draw[line width=0.9pt, cyan!70!blue] (L1)--(R2);
  \draw[line width=0.9pt, cyan!70!blue] (L3)--(R1);
  \draw[line width=0.9pt, cyan!70!blue] (L3)--(R4);

  % orange component
  \draw[line width=0.9pt, orange!90!yellow] (L2)--(R3);
  \draw[line width=0.9pt, orange!90!yellow] (L4)--(R3);

  \foreach \v in {L1,L2,L3,L4,R1,R2,R3,R4}{
    \node[circle,fill=black,inner sep=1.1pt] at (\v) {};
  }
\end{scope}

% ---------------- second graph ----------------
\begin{scope}[shift={(5.8,0)}]
  \coordinate (L1) at (0,  2.1);
  \coordinate (L2) at (0,  0.7);
  \coordinate (L3) at (0, -0.7);
  \coordinate (L4) at (0, -2.1);
  \coordinate (R1) at (3.2,  2.1);
  \coordinate (R2) at (3.2,  0.7);
  \coordinate (R3) at (3.2, -0.7);
  \coordinate (R4) at (3.2, -2.1);

  % green component
  \draw[line width=0.9pt, green!70!black] (L1)--(R3);
  \draw[line width=0.9pt, green!70!black] (L1)--(R4);

  % red component
  \draw[line width=0.9pt, red] (L2)--(R2);

  % brown component
  \draw[line width=0.9pt, brown!75!orange] (L3)--(R1);
  \draw[line width=0.9pt, brown!75!orange] (L4)--(R1);

  \foreach \v in {L1,L2,L3,L4,R1,R2,R3,R4}{
    \node[circle,fill=black,inner sep=1.1pt] at (\v) {};
  }
\end{scope}

% ---------------- third graph ----------------
\begin{scope}[shift={(11.6,0)}]
  \coordinate (L1) at (0,  2.1);
  \coordinate (L2) at (0,  0.7);
  \coordinate (L3) at (0, -0.7);
  \coordinate (L4) at (0, -2.1);
  \coordinate (R1) at (3.2,  2.1);
  \coordinate (R2) at (3.2,  0.7);
  \coordinate (R3) at (3.2, -0.7);
  \coordinate (R4) at (3.2, -2.1);

  % gray component
  \draw[line width=0.9pt, gray!60] (L1)--(R2);
  \draw[line width=0.9pt, gray!60] (L1)--(R3);
  \draw[line width=0.9pt, gray!60] (L4)--(R2);

  % yellow component
  \draw[line width=0.9pt, yellow!85!black] (L2)--(R1);
  \draw[line width=0.9pt, yellow!85!black] (L2)--(R4);
  \draw[line width=0.9pt, yellow!85!black] (L3)--(R1);

  \foreach \v in {L1,L2,L3,L4,R1,R2,R3,R4}{
    \node[circle,fill=black,inner sep=1.1pt] at (\v) {};
  }
\end{scope}

\end{tikzpicture}
\caption{Bipartite graphs associated to the simplicial distribution from Example~\ref{ex:Thirdex}.}
\label{Fig:Bipartitegrphs2}
\end{figure}
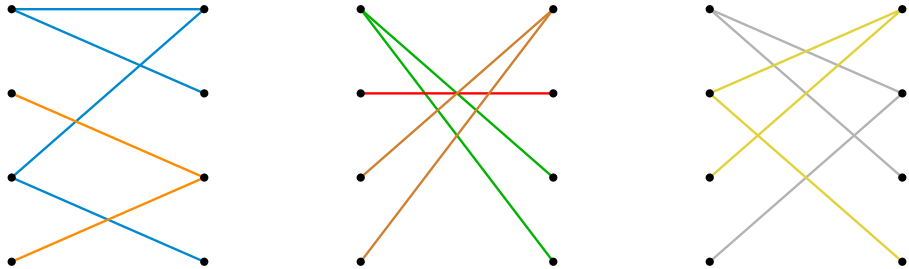
%\begin{figure}[h!] 
%  \centering
%  \includegraphics[width=.7\linewidth]{bipartitegraphs2.png}
%\caption{Bipartite graphs associated to the simplicial distribution from Example~\ref{ex:Thirdex}.}
%\label{Fig:Bipartitegrphs2}
%\end{figure} 

\begin{ex}
Using Corollary~\ref{cor:characgraphs2}, we can prove 
%of the fact
 that the simplicial distribution $p\colon R_3 \to D(\Delta_{\ZZ_4})$ from Example~\ref{ex:Thirdex} is a vertex.
%We explain why the simplicial distribution
%$p\colon C^{(1)} \vee C^{(1)} \vee C^{(1)} \to D(\Delta_{\ZZ_4})$
%from Example~\ref{ex:Thirdex} is a vertex.
The corresponding bipartite graphs are shown in
Figure~\ref{Fig:Bipartitegrphs2}. Clearly, none of the bipartite graphs contains a cycle.
%contains a cycle.
Moreover, 
%$\tilde{Q}(H_1,H_2,H_3)$ is the following 
we obtain the
$11\times 8$ matrix:
$$
\tilde{Q}(H_1,H_2,H_3)=\begin{pmatrix}
1 & 0 & 1 & 0 & -1 & -1 & 0 & -1\\
0 & 1 & 0 & 1 & 0 & 0 & -1 & 0 \\
1 & 0 & 0 & 0 & 0 &  0 & -1 & -1 \\
0 & 1 & 0 & 0 & 0 & -1 & 0 & 0 \\
0 & 0 & 1 & 1 & -1 & 0 & 0 & 0 \\
1 & 0 & 0 & 1 & 0 & -1 & -1 & 0 \\
0 & 1 & 1 & 0 & -1 & 0 & 0 & -1\\
1 & 0 & 0 & 0 & -1 & 0 & 0 & 0 \\
0 & 1 & 0 & 0 & 0 & -1 & 0 & 0 \\
0 & 0 & 1 & 0 & 0 & 0 & -1 & 0 \\
0 & 0 & 0 & 1 & 0 & 0 & 0 & -1 
\end{pmatrix}.
$$
Its rank is $7$ hence, by Corollary~\ref{cor:characgraphs2}, $p$ is a vertex.
\end{ex}

\begin{ex}
Next, we use Corollary~\ref{cor:characgraphs2} to give an alternative proof of the fact that the simplicial distribution $p\colon R_3 \to D(\Delta_{\ZZ_4})$ from Example~\ref{ex:counterex} is a vertex. 
The corresponding bipartite graphs are shown in
Figure~\ref{Fig:Bipartitegrphs3}. Again, these graphs do not contain any cycles,
%there are no cycles,
and
%$\tilde{Q}(H_1,H_2,H_3)$ is the following
we obtain the $12\times 8$ matrix:
$$
\tilde{Q}(H_1,H_2,H_3)=
\begin{pmatrix}
1 & 0 & 1 & 1 & -1 & -1 & 0 & 0 \\
0 & 1 & 0 & 0 & 0 & 0 & -1 & -1 \\
1 & 0 & 0 & 0 & 0 &  -1 & 0 & 0 \\
0 & 1 & 0 & 0 & -1 & 0 & 0 & 0 \\
0 & 0 & 1 & 0 & 0 & 0 & 0 & -1 \\
0 & 0 & 0 & 1 & 0 & 0 & -1 & 0 \\
1 & 0 & 0 & 1 & 0 & -1 & 0 & -1\\
0 & 1 & 1 & 0 & -1 & 0 & -1 & 0 \\
1 & 0 & 0 & 0 & -1 & 0 & 0 & 0 \\
0 & 1 & 0 & 0 & 0 & -1 & 0 & 0 \\
0 & 0 & 1 & 0 & 0 & 0 & -1 & 0 \\
0 & 0 & 0 & 1 & 0 & 0 & 0 & -1 
\end{pmatrix}.
$$
Its rank is $7$, so $p$ is a vertex.   
\end{ex}

\begin{figure}[ht]
\centering
\begin{tikzpicture}[x=1cm,y=1cm,scale=.8]

% common coordinates for each panel:
% left side  : y =  2.1,  0.7, -0.7, -2.1
% right side : y =  2.1,  0.7, -0.7, -2.1

% ---------------- first graph ----------------
\begin{scope}[shift={(0,0)}]
  \coordinate (L1) at (0,  2.1);
  \coordinate (L2) at (0,  0.7);
  \coordinate (L3) at (0, -0.7);
  \coordinate (L4) at (0, -2.1);
  \coordinate (R1) at (3.2,  2.1);
  \coordinate (R2) at (3.2,  0.7);
  \coordinate (R3) at (3.2, -0.7);
  \coordinate (R4) at (3.2, -2.1);

  % blue component
  \draw[line width=0.9pt, cyan!70!blue] (L1)--(R1);
  \draw[line width=0.9pt, cyan!70!blue] (L1)--(R2);
  \draw[line width=0.9pt, cyan!70!blue] (L3)--(R2);
  \draw[line width=0.9pt, cyan!70!blue] (L4)--(R1);

  % orange component
  \draw[line width=0.9pt, orange!90!yellow] (L2)--(R3);
  \draw[line width=0.9pt, orange!90!yellow] (L2)--(R4);

  \foreach \v in {L1,L2,L3,L4,R1,R2,R3,R4}{
    \node[circle,fill=black,inner sep=1.1pt] at (\v) {};
  }
\end{scope}

% ---------------- second graph ----------------
\begin{scope}[shift={(5.8,0)}]
  \coordinate (L1) at (0,  2.1);
  \coordinate (L2) at (0,  0.7);
  \coordinate (L3) at (0, -0.7);
  \coordinate (L4) at (0, -2.1);
  \coordinate (R1) at (3.2,  2.1);
  \coordinate (R2) at (3.2,  0.7);
  \coordinate (R3) at (3.2, -0.7);
  \coordinate (R4) at (3.2, -2.1);

  % upper pair
  \draw[line width=0.9pt, green!70!black] (L1)--(R2);
  \draw[line width=0.9pt, red]              (L2)--(R1);

  % lower pair
  \draw[line width=0.9pt, brown!75!orange] (L3)--(R4);
  \draw[line width=0.9pt, gray!60]         (L4)--(R3);

  \foreach \v in {L1,L2,L3,L4,R1,R2,R3,R4}{
    \node[circle,fill=black,inner sep=1.1pt] at (\v) {};Example 4.20
  }
\end{scope}

% ---------------- third graph ----------------
\begin{scope}[shift={(11.6,0)}]
  \coordinate (L1) at (0,  2.1);
  \coordinate (L2) at (0,  0.7);
  \coordinate (L3) at (0, -0.7);
  \coordinate (L4) at (0, -2.1);
  \coordinate (R1) at (3.2,  2.1);
  \coordinate (R2) at (3.2,  0.7);
  \coordinate (R3) at (3.2, -0.7);
  \coordinate (R4) at (3.2, -2.1);

  % yellow component
  \draw[line width=0.9pt, yellow!85!black] (L1)--(R2);
  \draw[line width=0.9pt, yellow!85!black] (L1)--(R4);
  \draw[line width=0.9pt, yellow!85!black] (L4)--(R2);

  % purple component
  \draw[line width=0.9pt, violet] (L2)--(R1);
  \draw[line width=0.9pt, violet] (L2)--(R3);
  \draw[line width=0.9pt, violet] (L3)--(R1);

  \foreach \v in {L1,L2,L3,L4,R1,R2,R3,R4}{
    \node[circle,fill=black,inner sep=1.1pt] at (\v) {};
  }
\end{scope}

\end{tikzpicture}
\caption{Bipartite graphs associated to the simplicial distribution from Example~\ref{ex:counterex}.}
\label{Fig:Bipartitegrphs3}
\end{figure}
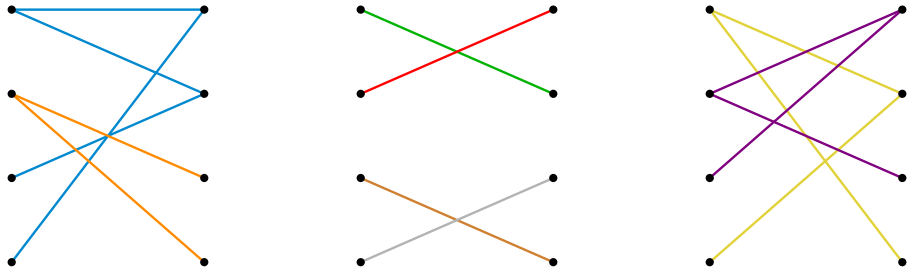
%\begin{figure}[h!] 
%  \centering
%  \includegraphics[width=.7\linewidth]{bipartitegraphs3.png}
%\caption{Bipartite graphs associated to the simplicial distribution from Example~\ref{ex:counterex}.}
%\label{Fig:Bipartitegrphs3}
%\end{figure} 

In the bipartite graphs arising from Example~\ref{ex:Thirdex} (Figure~\ref{Fig:Bipartitegrphs2}), we have
$\operatorname{rank} Q(H_1,H_2,H_3)=5<2m-1=7$. Likewise, for the bipartite graphs in Example~\ref{ex:counterex}
(Figure~\ref{Fig:Bipartitegrphs3}), $\operatorname{rank} Q(H_1,H_2,H_3)=6<7$. Hence, by
Theorem~\ref{thm:bybipartgraphs}, these collections fail the required rank condition and therefore cannot correspond to
vertices of $\Dist(D_3,{4})$. In contrast, when we replace $Q$ by the rose-adapted matrix $\tilde{Q}$,
both examples satisfy the sharp rank condition $\operatorname{rank}\!\bigl(\tilde{Q}(H_1,H_2,H_3)\bigr)=2m-1=7$
(as verified in Examples~\ref{ex:Thirdex} and~\ref{ex:counterex}), thereby confirming that they do define vertices of
$\Dist(R_3,{4})$.

\section{Counting the number of vertices}\label{sec:1dim}

%
%\begin{ex}
%In Figure \ref{fig:collpsing2 K33}, we collapse the blue edges in the left graph to get 
%a gluing of five edges along their endpoints.
%
%

%\begin{cihan}{purple}{to be revised}
In this section, we demonstrate how the identification of vertices on the rose graph and the dipole graph can be used to detect vertices on connected $1$-dimensional measurement spaces, including physically relevant
bipartite Bell scenarios.

%\end{cihan}

\subsection{Collapsing to the rose graph}

A connected directed (multi)graph can be represented by a one-dimensional simplicial set $X$ (see Example \ref{ex:one dim sdist}).
By collapsing a fixed spanning tree, we obtain a wedge of circles (i.e., the rose scenario).
We denote by $T(X)$ the number of spanning trees in $X$.
According to Kirchhoff's theorem~\cite{kirchhoff1958solution},
this number is given by
\[
T(X)= \frac{1}{k}\lambda_1\cdot \lambda_2 \dots \cdot \lambda_k,
\]
where $k$ is the number of nodes and $\lambda_1,\dots,\lambda_k$ are the nonzero eigenvalues of the Laplacian matrix of the graph.

\begin{defn}
We define $\kappa^{(1)}(n,m)$ as the number of vertices in
$\Dist(R_n,{m})$, and define $\tilde{\kappa}^{(1)}(n,m)$ as the number of contextual vertices in
$\Dist(R_n,{m})$ that do not contain any collapsed distribution (Definition~\ref{def:CollapDist}) on any circle.
\end{defn}
\begin{ex}
There are only two deterministic distributions in
$\Dist(R_n,{2})$, and by Example~\ref{ex:contverwedgez2} there are $2^n-1$ contextual vertices.
Thus $\kappa^{(1)}(n,2)=2^n+1$. In addition, we have $\tilde{\kappa}^{(1)}(n,2)=1$.
\end{ex}

\begin{ex}\label{ex:K33ZZ3wedgecycles}
Using a computer computation we obtained
%algorithm we computed the following:
%\ak{$\kappa^{(1)}(2,3)$, $\kappa^{(1)}(3,3)$, and $\kappa^{(1)}(4,3)$ computed by Talay:}
\[
\kappa^{(1)}(2,3)=56 \; , \; \kappa^{(1)}(3,3)=488 \; , \; \kappa^{(1)}(4,3)=4088.
%\; , \; \kappa^{(1)}(5,3)= 32528
\]
Using these results, we can compute $\tilde{\kappa}^{(1)}(4,3)$.
First, we count the vertices of $\Dist(R_4,{3})$ having at least one collapsed distribution.
Let us denote the generating $1$-simplices of $R_4$ by $\sigma_1,\sigma_2,\sigma_3,\sigma_4$.
We have the following:
\begin{enumerate}
\item For $1\leq i \leq 4$, the number of vertices in $\Dist(R_4,{3})$ whose restriction to $\sigma_i$
is a collapsed distribution is equal to $\kappa^{(1)}(3,3)=488$.
\item For $1\leq i<j \leq 4$, the number of vertices in $\Dist(R_4,{3})$ whose restrictions to $\sigma_i$ and $\sigma_j$
are collapsed distributions is equal to $\kappa^{(1)}(2,3)=56$.
\item For $1\leq i<j<k \leq 4$, the number of vertices in $\Dist(R_4,{3})$ whose restrictions to $\sigma_i$, $\sigma_j$, and $\sigma_k$
are collapsed distributions is equal to $\kappa^{(1)}(1,3)$, which according to~\cite[Corollary~4.5]{kharoof2024extremal} satisfies
\begin{equation}\label{eq:mu13}
\kappa^{(1)}(1,3) = \sum_{k=1}^3 \binom{3}{k} (k-1)! = 8.
\end{equation}
\item The vertices in $\Dist(R_4,{3})$ with four collapsed distributions
are exactly the three deterministic distributions.
\end{enumerate}
Using the inclusion--exclusion principle, there are
\[
\binom{4}{1} \cdot 488 - \binom{4}{2} \cdot 56 + \binom{4}{3} \cdot 8 -\binom{4}{4} \cdot 3
=1952-336+32-3=1645
\]
vertices of $\Dist(R_4,{3})$ with at least one collapsed distribution. Therefore
\[
\tilde{\kappa}^{(1)}(4,3)=\kappa^{(1)}(4,3)-1645=4088-1645=2443.
\]
\end{ex}

\begin{pro}\label{pro:minvertwedge}
Let $X=(V,E)$ be a directed graph. Then there are at least \[\tilde{\kappa}^{(1)}(|E|-|V|+1,m)\cdot T(X)\]
contextual vertices in $\Dist(X,{m})$.
\end{pro}
\begin{proof}
Let $T_1$ be a spanning tree of $X$. By collapsing the $|V|-1$ edges of $T_1$, all nodes are identified, and we obtain the 
%collapsing 
collapsed
space
$X/T_1$ as the rose scenario $R_{|E|-(|V|-1)} = R_{|E|-|V|+1}$ (see Definition~\ref{def:collapsing} and Figure~\ref{fig:collpsing K33}).
Let $\pi_1 \colon X \to X / T_1$ be the corresponding collapsing map.
By part~(2) of Proposition~\ref{pro:CollapDist}, for every contextual vertex
$p$ in $\Dist(X / T_1, {m})$, the simplicial distribution $\pi_1^\ast(p)$ is a contextual vertex in $\Dist(X,{m})$.
Since the map $\pi_1^\ast$ is injective, and by part~(1) of Proposition~\ref{pro:CollapDist}, we obtain
$\tilde{\kappa}^{(1)}(|E|-|V|+1,m)$ contextual vertices $q$ in $\Dist(X,{m})$ such that, for every edge $\tau$ in $X$, the distribution
$q_{\tau}$  is collapsed if and only if  $\tau \in T_1$.

Now, for another spanning tree $T_2 \neq T_1$, we obtain another set of $\tilde{\kappa}^{(1)}(|E|-|V|+1,m)$ contextual vertices in 
$\Dist(X,{m})$.
Therefore, there are at least $\tilde{\kappa}^{(1)}(|E|-|V|+1,m)\cdot T(X)$ contextual vertices in $\Dist(X,{m})$.
\end{proof}

\begin{figure}[ht]
\centering

\begin{minipage}{0.52\textwidth}
\centering
\begin{tikzpicture}[x=1cm,y=1cm]
  % vertices for a deformed drawing of K_{3,3}
  \coordinate (v1) at (0,  1.2);
  \coordinate (v2) at (0, -1.2);
  \coordinate (v3) at (2.0,  1.2);
  \coordinate (v4) at (2.0, -1.2);
  \coordinate (v5) at (4.3,  2.0);
  \coordinate (v6) at (4.3, -2.0);

  % black edges (the four edges surviving after collapse)
  \draw[line width=0.9pt] (v1)--(v2);
  \draw[line width=0.9pt] (v1) .. controls (1.4,1.9) and (3.0,2.2) .. (v5);
  \draw[line width=0.9pt] (v2) .. controls (1.4,-1.9) and (3.0,-2.2) .. (v6);
  \draw[line width=0.9pt] (v5) .. controls (4.9,0.9) and (4.9,-0.9) .. (v6);

  % blue spanning tree to be collapsed
  \draw[line width=1.1pt, cyan!70!blue] (v1)--(v3);
  \draw[line width=1.1pt, cyan!70!blue] (v2)--(v4);
  \draw[line width=1.1pt, cyan!70!blue] (v3)--(v4);
  \draw[line width=1.1pt, cyan!70!blue] (v3) .. controls (3.1,0.8) and (4.0,-0.2) .. (v6);
  \draw[line width=1.1pt, cyan!70!blue] (v4) .. controls (3.1,-0.8) and (4.0,0.2) .. (v5);

  \foreach \v in {v1,v2,v3,v4,v5,v6}{
    \node[circle,fill=black,inner sep=1.1pt] at (\v) {};
  }
\end{tikzpicture}

\smallskip

(a)
\end{minipage}
\hfill
\begin{minipage}{0.34\textwidth}
\centering
\begin{tikzpicture}[x=1cm,y=1cm]
  \coordinate (O) at (0,0);

  % top petal
  \draw[line width=0.9pt]
    (O) .. controls (-0.08,0.14) and (-0.55,0.95) .. (-0.22,1.42)
        .. controls (-0.08,1.62) and ( 0.08,1.62) .. ( 0.22,1.42)
        .. controls ( 0.55,0.95) and ( 0.08,0.14) .. (O);

  % right petal
  \draw[line width=0.9pt]
    (O) .. controls (0.14, 0.08) and (0.95, 0.55) .. (1.42, 0.22)
        .. controls (1.62, 0.08) and (1.62,-0.08) .. (1.42,-0.22)
        .. controls (0.95,-0.55) and (0.14,-0.08) .. (O);

  % bottom petal
  \draw[line width=0.9pt]
    (O) .. controls ( 0.08,-0.14) and ( 0.55,-0.95) .. ( 0.22,-1.42)
        .. controls ( 0.08,-1.62) and (-0.08,-1.62) .. (-0.22,-1.42)
        .. controls (-0.55,-0.95) and (-0.08,-0.14) .. (O);

  % left petal
  \draw[line width=0.9pt]
    (O) .. controls (-0.14,-0.08) and (-0.95,-0.55) .. (-1.42,-0.22)
        .. controls (-1.62,-0.08) and (-1.62, 0.08) .. (-1.42, 0.22)
        .. controls (-0.95, 0.55) and (-0.14, 0.08) .. (O);

  \node[circle,fill=black,inner sep=1.1pt] at (O) {};
\end{tikzpicture}

\smallskip

(b)
\end{minipage}

\caption{(a) The complete bipartite graph \(K_{3,3}\). (b) The rose graph \(R_4\). We collapse the blue edges in \(K_{3,3}\) to obtain \(R_4\).}
\label{fig:collpsing K33}
\end{figure}
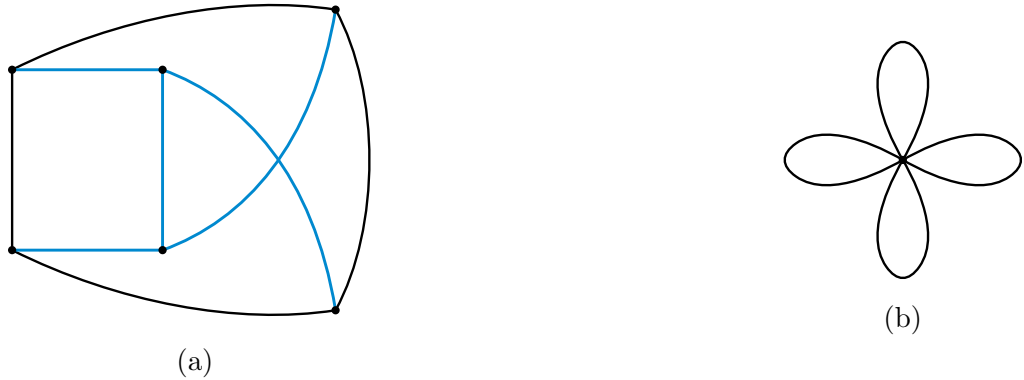

%\begin{figure}[h!]
%\centering
%\begin{subfigure}{.44\textwidth}
%  \centering
%  \includegraphics[width=.6\linewidth]{collpsing K33 left.png}
%  \caption{}
%  \label{fig:collpsing K33 left}
%\end{subfigure}%
%\begin{subfigure}{.32\textwidth}
%  \centering
%  \includegraphics[width=.6\linewidth]{collpsing K33 right.png}
%  \caption{}
%  \label{fig:collpsing K33 right}
%\end{subfigure}
%\caption{{(a) The bipartite graph $K_{3,3}$. (b) The rose graph $R_4$. We collapse the blue edges in $K_{3,3}$ to obtain $R_4$.}  
%}
%\label{fig:collpsing K33}
%\end{figure} 

\begin{cor}\label{cor:bipartiwedge}
Let $K_{n_1,n_2}$ denote the complete bipartite graph.
There are at least 
\[\tilde{\kappa}^{(1)}\!\left((n_1-1)(n_2-1),m\right)\cdot n_1^{n_2-1}n_2^{n_1-1}\]
contextual vertices in $\Dist(K_{n_1,n_2},{m})$.
% (see Definition~\ref{def:Kn1n2}).
\end{cor}
\begin{proof}
The complete bipartite graph $K_{n_1,n_2}$ has $n_1+n_2$ nodes and $n_1n_2$ edges, so $|E|-|V|+1$ is equal to
\[
n_1n_2-n_1-n_2+1=(n_1-1)(n_2-1).
\]
Furthermore, there are exactly $n_1^{n_2-1}n_2^{n_1-1}$ spanning trees in $K_{n_1,n_2}$ (see~\cite{harary1969graph}).
The result now follows from Proposition~\ref{pro:minvertwedge}.
\end{proof}
\begin{ex}
By Corollary~\ref{cor:bipartiwedge} and Example~\ref{ex:K33ZZ3wedgecycles}, there are at least
\[
\tilde{\kappa}^{(1)}(4,3)\cdot 3^2 3^2=197883
\]
contextual vertices in $\Dist(K_{3,3},{3})$.
\end{ex}

\subsection{Collapsing to the dipole graph}

Next, we apply the collapsing method to show how knowing the number of vertices in the distribution polytope of the dipole graph
yields a lower bound on the number of contextual vertices in bipartite scenarios.

\begin{defn}
We define $\kappa^{(2)}(n,m)$ as the number of vertices of $\Dist(D_n,{m})$, and
$\tilde{\kappa}^{(2)}(n,m)$ as the number of contextual vertices of $\Dist(D_n,{m})$ that do not contain any collapsed distribution on any edge.
\end{defn}

\begin{ex}
We have $\kappa^{(2)}(1,m) = m^2$ and $\tilde{\kappa}^{(2)}(1,m)=0$, since there are no contextual vertices in 
$\Dist(\Delta^1,{m})$.
\end{ex}

\begin{ex}\label{ex:K33ZZ3edges}
Using computer computation we obtain
%\ak{$\kappa^{(2)}(3,3)$ computed by Talay}
\[
\kappa^{(2)}(3,3)=561.
\]
To compute $\tilde{\kappa}^{(2)}(3,3)$, note first that for each $i \in \{1,2,3\}$, according to part~(2) of Proposition~\ref{pro:CollapDist},
the number of vertices of $\Dist(D_3,{3})$ whose distribution on $\tau_i$
is collapsed equals $\kappa^{(1)}(2,3)=56$ (see Example~\ref{ex:K33ZZ3wedgecycles}).
On the other hand, for $1\leq i<j \leq 3$, the number of vertices with collapsed distributions on both $\tau_i$ and $\tau_j$
is equal to $\kappa^{(1)}(1,3)=8$ (see Equation~\ref{eq:mu13}).
In addition, the vertices of $\Dist(D_3,{3})$ whose distributions on all three edges
are collapsed are precisely three deterministic distributions (those assigning a common outcome to the two nodes).

Therefore, by the inclusion--exclusion principle, the number of vertices with at least one collapsed distribution is
\[
\binom{3}{1} \cdot 56 - \binom{3}{2} \cdot 8 + \binom{3}{3} \cdot 3 = 168 - 24 + 3 = 147.
\]
Hence, the number of vertices without any collapsed distribution is
\[
\kappa^{(2)}(3,3) - 147 = 561 - 147 = 414.
\]
Finally, subtracting the number of deterministic distributions without any collapsed distribution, we obtain
\[
\tilde{\kappa}^{(2)}(3,3) = 414 - 3 \cdot 2 = 408.
\]
\end{ex}

\begin{figure}[ht]
\centering

\begin{minipage}{0.48\textwidth}
\centering
\begin{tikzpicture}[x=1cm,y=1cm]
  % standard K_{3,3} layout
  \coordinate (L1) at (0,  1.6);
  \coordinate (L2) at (0,  0);
  \coordinate (L3) at (0, -1.6);
  \coordinate (R1) at (4.0,  1.6);
  \coordinate (R2) at (4.0,  0);
  \coordinate (R3) at (4.0, -1.6);

  % blue forest whose contraction gives the two vertices of D_5
  \draw[line width=1.1pt, cyan!70!blue] (L1)--(R1);
  \draw[line width=1.1pt, cyan!70!blue] (L1)--(R2);
  \draw[line width=1.1pt, cyan!70!blue] (L2)--(R3);
  \draw[line width=1.1pt, cyan!70!blue] (L3)--(R3);

  % remaining five edges
  \draw[line width=0.9pt] (L1)--(R3);
  \draw[line width=0.9pt] (L2)--(R1);
  \draw[line width=0.9pt] (L2)--(R2);
  \draw[line width=0.9pt] (L3)--(R1);
  \draw[line width=0.9pt] (L3)--(R2);

  \foreach \v in {L1,L2,L3,R1,R2,R3}{
    \node[circle,fill=black,inner sep=1.1pt] at (\v) {};
  }
\end{tikzpicture}

\smallskip

(a)
\end{minipage}
\hfill
\begin{minipage}{0.34\textwidth}
\centering
\begin{tikzpicture}[x=1cm,y=1cm]
  \coordinate (A) at (0,0);
  \coordinate (B) at (3.6,0);

  \draw[line width=0.9pt] (A) to[bend left=60] (B);
  \draw[line width=0.9pt] (A) to[bend left=30] (B);
  \draw[line width=0.9pt] (A) -- (B);
  \draw[line width=0.9pt] (A) to[bend right=30] (B);
  \draw[line width=0.9pt] (A) to[bend right=60] (B);

  \node[circle,fill=black,inner sep=1.1pt] at (A) {};
  \node[circle,fill=black,inner sep=1.1pt] at (B) {};
\end{tikzpicture}

\smallskip

(b)
\end{minipage}

\caption{(a) The complete bipartite graph \(K_{3,3}\). (b) The dipole graph \(D_5\). We collapse the blue edges in \(K_{3,3}\) to obtain \(D_5\).}
\label{fig:collpsing2 K33}
\end{figure}
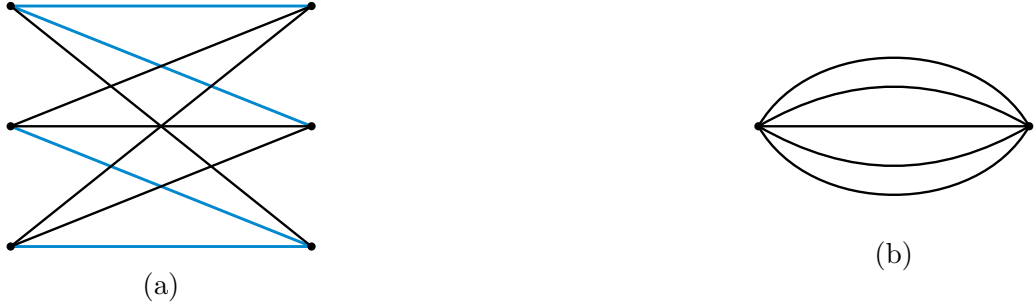
%\begin{figure}[h!]
%\centering
%\begin{subfigure}{.33\textwidth}
%  \centering
%  \includegraphics[width=.6\linewidth]{Collapsing2 K33 left.png}
%  \caption{}
%  \label{fig:Collapsing2 K33 left}
%\end{subfigure}%
%\begin{subfigure}{.32\textwidth}
%  \centering
%  \includegraphics[width=.6\linewidth]{5 edges.png}
%  \caption{}
%  \label{fig:5 edges}
%\end{subfigure}
%\caption{{(a) The bipartite graph $K_{3,3}$. (b) the dipole graph $D_5$. We collapse the blue edges in $K_{3,3}$ to obtain $D_5$.
%} 
%}
%\label{fig:collpsing2 K33}
%\end{figure}

\begin{pro}\label{pro:bipartiedges}
There are at least
\[
\sum_{\substack{0 < k < n_1 \\ 0 < r < n_2}}
\binom{n_1}{k} \cdot \binom{n_2}{r} \cdot
\tilde{\kappa}^{(2)}\!\left(k(n_2 - r) + (n_1 - k)r, m\right)
\]
contextual vertices in $\Dist(K_{n_1,n_2},{m})$.
\end{pro}

 \begin{proof}
Let $X = K_{A,B}$ be a complete bipartite graph, where $|A| = n_1$ and $|B| = n_2$.
Consider subsets $A_1, A_2 \subseteq A$ and $B_1, B_2 \subseteq B$ such that:
\begin{itemize}
    \item $A = A_1 \sqcup A_2$ and $B = B_1 \sqcup B_2$;
    \item $|A_1| = k$, $|B_1| = r$, with $0 < k < n_1$ and $0 < r < n_2$.
\end{itemize}

Now collapse all edges between $A_1$ and $B_1$, and between $A_2$ and $B_2$.
This yields a new graph $X'$, formed by gluing together
\[
|A_1|\cdot |B_2| + |A_2|\cdot |B_1| = k(n_2 - r) + (n_1 - k)r
\]
edges through two points (see Figure~\ref{fig:collpsing2 K33}), i.e., $X' \cong D_{k(n_2-r)+(n_1-k)r}$.
Let $\pi \colon X \to X'$ denote the corresponding collapsing map.

By part~(2) of Proposition~\ref{pro:CollapDist}, for each contextual vertex
$p \in \Dist(X', {m})$, the pullback $\pi^\ast(p)$
is a contextual vertex in $\Dist(X, {m})$.
Since $\pi^\ast$ is injective, part~(1) of Proposition~\ref{pro:CollapDist} implies that there are at least
\[
\tilde{\kappa}^{(2)}\!\left(k(n_2 - r) + (n_1 - k)r, m\right)
\]
contextual vertices $q$ in $\Dist(X, {m})$ satisfying the property that,
for every edge $\tau$ in $X$,
$q_{\tau}$ is a collapsed distribution if and only if 
$\tau$ is an edge between  $A_1$  and $B_1$
 or between $A_2$  and  $B_2$.
Summing over all valid choices of $(k,r)$, and noting that the number of such choices is
$\binom{n_1}{k}\binom{n_2}{r}$, the result follows.
\end{proof}

\begin{ex}
By Proposition~\ref{pro:bipartiedges} and Example~\ref{ex:K33ZZ3edges},
we find that in $\Dist(K_{3,2}, {3})$ there are at least
\[
 \binom{3}{1} \cdot \binom{2}{1} \cdot \tilde{\kappa}^{(2)}(3,3)
 +  \binom{3}{2} \cdot \binom{2}{1} \cdot \tilde{\kappa}^{(2)}(3,3)
 = 12 \cdot \tilde{\kappa}^{(2)}(3,3)
 = 12 \cdot 408 = 4896
\]
contextual vertices.
\end{ex}

\bibliography{bib.bib}

@article{bermejo2017contextuality,
  title={Contextuality as a resource for models of quantum computation with qubits},
  author={Bermejo-Vega, Juan and Delfosse, Nicolas and Browne, Dan E and Okay, Cihan and Raussendorf, Robert},
  journal={Physical review letters},
  volume={119},
  number={12},
  pages={120505},
  year={2017},
  publisher={APS}
}

@article{howard2014contextuality,
  title={Contextuality supplies the ‘magic’for quantum computation},
  author={Howard, Mark and Wallman, Joel and Veitch, Victor and Emerson, Joseph},
  journal={Nature},
  volume={510},
  number={7505},
  pages={351--355},
  year={2014},
  publisher={Nature Publishing Group UK London}
}

@article{budroni2022kochen,
  title={Kochen-specker contextuality},
  author={Budroni, Costantino and Cabello, Ad{\'a}n and G{\"u}hne, Otfried and Kleinmann, Matthias and Larsson, Jan-{\AA}ke},
  journal={Reviews of Modern Physics},
  volume={94},
  number={4},
  pages={045007},
  year={2022},
  publisher={APS}
}

@article{raussendorf2009contextuality,
  title = {Contextuality in measurement-based quantum computation},
  author = {Raussendorf, Robert},
  journal = {Phys. Rev. A},
  volume = {88},
  issue = {2},
  pages = {022322},
  numpages = {7},
  year = {2013},
  month = {Aug},
  publisher = {American Physical Society},
  doi = {10.1103/PhysRevA.88.022322},
  url = {https://link.aps.org/doi/10.1103/PhysRevA.88.022322}
}

@book{brualdi2006combinatorial,
  title={Combinatorial matrix classes},
  author={Brualdi, Richard A},
  volume={13},
  year={2006},
  publisher={Cambridge University Press}
}

@article{bolker1976simplicial,
  title={Simplicial geometry and transportation polytopes},
  author={Bolker, Ethan D},
  journal={Transactions of the American Mathematical Society},
  volume={217},
  pages={121--142},
  year={1976}
}

@inproceedings{prusa2013universality,
  title={Universality of the local marginal polytope},
  author={Prusa, Daniel and Werner, Tomas},
  booktitle={Proceedings of the IEEE Conference on Computer Vision and Pattern Recognition},
  pages={1738--1743},
  year={2013}
}

@article{okay2025polyhedral,
  title={Polyhedral Classical Simulators for Quantum Computation},
  author={Okay, Cihan},
  journal={arXiv preprint arXiv:2510.07540},
  year={2025}
}

@article{de2013combinatorics,
  title={Combinatorics and geometry of transportation polytopes: An update.},
  author={De Loera, Jes{\'u}s A and Kim, Edward D},
  journal={Discrete geometry and algebraic combinatorics},
  volume={625},
  pages={37--76},
  year={2013}
}

@book{bertsimas1997introduction,
  title={Introduction to linear optimization},
  author={Bertsimas, Dimitris and Tsitsiklis, John N},
  volume={6},
  year={1997},
  publisher={Athena scientific Belmont, MA}
}

@article{kharoof2024extremal,
  title={Extremal simplicial distributions on cycle scenarios with arbitrary outcomes},
  author={Kharoof, Aziz and Ipek, Selman and Okay, Cihan},
  journal={Journal of Physics A: Mathematical and Theoretical},
  year={2024}
}

@article{barbosa2023bundle,
  title={A bundle perspective on contextuality: Empirical models and simplicial distributions on bundle scenarios},
  author={Barbosa, Rui Soares and Kharoof, Aziz and Okay, Cihan},
  journal={arXiv preprint arXiv:2308.06336},
  year={2023}
}

@article{kharoof2023homotopical,
title = {Homotopical characterization of strongly contextual simplicial distributions on cone spaces},
journal = {Topology and its Applications},
volume = {352},
pages = {108956},
year = {2024},
issn = {0166-8641},
doi = {https://doi.org/10.1016/j.topol.2024.108956},
url = {https://www.sciencedirect.com/science/article/pii/S016686412400141X},
author = {Aziz Kharoof and Cihan Okay}
}

@article{okay2024twisted,
  title={Twisted simplicial distributions},
  author={Okay, Cihan and Stern, Walker H},
  journal={arXiv preprint arXiv:2403.19808},
  year={2024}
}

@article{kharoof2023topological,
  title={Topological methods for studying contextuality: N-cycle scenarios and beyond},
  author={Kharoof, Aziz and Ipek, Selman and Okay, Cihan},
  journal={Entropy},
  volume={25},
  number={8},
  pages={1127},
  year={2023},
  publisher={MDPI}
}

@book{riehl2017category,
  title={Category theory in context},
  author={Riehl, Emily},
  year={2017},
  publisher={Courier Dover Publications}
}

@article{kharoof2022simplicial,
  author       = {Aziz Kharoof and Cihan Okay},
  title        = {Simplicial distributions, convex categories and contextuality},
  journal      = {Theory and Applications of Categories},
  volume       = {44},
  number       = {13},
  pages        = {372--409},
  year         = {2025}
}

@article{okay2022simplicial,
  doi = {10.22331/q-2023-05-22-1009},
  url = {https://doi.org/10.22331/q-2023-05-22-1009},
  title = {Simplicial quantum contextuality},
  author = {Okay, Cihan and Kharoof, Aziz and Ipek, Selman},
  journal = {{Quantum}},
  issn = {2521-327X},
  publisher = {{Verein zur F{\"{o}}rderung des Open Access Publizierens in den Quantenwissenschaften}},
  volume = {7},
  pages = {1009},
  month = may,
  year = {2023}
}

@article{barrett2005nonlocal,
  title = {Nonlocal correlations as an information-theoretic resource},
  author = {Barrett, Jonathan and Linden, Noah and Massar, Serge and Pironio, Stefano and Popescu, Sandu and Roberts, David},
  journal = {Physical Review A},
  volume = {71},
  issue = {2},
  pages = {022101},
  numpages = {11},
  year = {2005},
  month = {Feb},
  publisher = {APS},
  nonote={doi: \href{https://link.aps.org/doi/10.1103/PhysRevA.71.022101}{10.1103/PhysRevA.71.022101}. arXiv: \href{https://arxiv.org/abs/quant-ph/0404097}{quant-ph/0404097}},
  nodoi = {10.1103/PhysRevA.71.022101},
  nourl = {https://link.aps.org/doi/10.1103/PhysRevA.71.022101},
  noeprint={quant-ph/0404097}
}

@book{goerss2009simplicial,
  title={Simplicial homotopy theory},
  author={Goerss, Paul G and Jardine, John F},
  year={2009},
  publisher={Springer Science \& Business Media}
}

@article{friedman2008elementary,
  title   = {An Elementary Illustrated Introduction to Simplicial Sets},
  author  = {Friedman, Greg},
  journal = {The Rocky Mountain Journal of Mathematics},
  volume  = {42},
  number  = {2},
  pages   = {353--423},
  year    = {2012}
}

@article{abramsky2011sheaf,
  title={The sheaf-theoretic structure of non-locality and contextuality},
  author={Abramsky, Samson and Brandenburger, Adam},
  journal={New Journal of Physics},
  volume={13},
  number={11},
  pages={113036},
  year={2011},
  publisher={IOP Publishing},
  nonote={doi: \href{https://doi.org/10.1088/1367-2630/13/11/113036}{10.1088/1367-2630/13/11/113036}. arXiv: \href{https://arxiv.org/abs/1102.0264}{1102.0264}},
  nodoi={10.1088/1367-2630/13/11/113036},
  noeprint={1102.0264},
}

@article{abramsky2016possibilities,
  title={Possibilities determine the combinatorial structure of probability polytopes},
  author={Abramsky, Samson and Barbosa, Rui Soares and Kishida, Kohei and Lal, Raymond and Mansfield, Shane},
  journal={Journal of Mathematical Psychology},
  volume={74},
  pages={58--65},
  year={2016},
  publisher={Elsevier},
  nonote={doi: \href{https://doi.org/10.1016/j.jmp.2016.03.006}{10.1016/j.jmp.2016.03.006}. arXiv: \href{https://arxiv.org/abs/1603.07735}{1603.07735}},
  nodoi={10.1016/j.jmp.2016.03.006},
  nourl={https://doi.org/10.1016/j.jmp.2016.03.006},
  noeprint={1603.07735}
}

@article{brunner2014bell,
  title={Bell nonlocality},
  author={Brunner, Nicolas and Cavalcanti, Daniel and Pironio, Stefano and Scarani, Valerio and Wehner, Stephanie},
  journal={Review of Modern Physics},
  volume={86},
  pages={419},
  year={2014},
  publisher={APS},
  nonote={doi: \href{https://link.aps.org/doi/10.1103/RevModPhys.86.419}{10.1103/RevModPhys.86.419}. arXiv: \href{https://arxiv.org/abs/1303.2849}{1303.2849}},
  nodoi = {10.1103/RevModPhys.86.419},
  nourl = {https://link.aps.org/doi/10.1103/RevModPhys.86.419},
  noeprint={1303.2849}
}

@book{pitowsky1989quantum,
	author = {Itamar Pitowsky},
	year = {1989},
	title = {Quantum Probability Quantum Logic},
	publisher = {Springer}
}

@article{kirchhoff1958solution,
  title={On the solution of the equations obtained from the investigation of the linear distribution of galvanic currents},
  author={Kirchhoff, Gustav},
  journal={IRE transactions on circuit theory},
  volume={5},
  number={1},
  pages={4--7},
  year={1958},
  publisher={IEEE}
}

@book{harary1969graph,
  title     = {Graph Theory},
  author    = {Harary, Frank},
  publisher = {Addison-Wesley},
  address   = {Reading, Massachusetts},
  year      = {1969}
}

@article{kharoof2026geometry,
  title={The geometry of simplicial distributions on suspension scenarios},
  author={Kharoof, Aziz},
  journal={Journal of Applied and Computational Topology},
  volume={10},
  number={1},
  pages={2},
  year={2026},
  publisher={Springer}
}

@book{rockafellar2015convex,
  title     = {Convex Analysis},
  author    = {Rockafellar, Ralph Tyrell},
  series    = {Princeton Mathematical Series},
  volume    = {28},
  publisher = {Princeton University Press},
  year      = {2015}
}

@article{popescu1994quantum,
  title={Quantum nonlocality as an axiom},
  author={Popescu, Sandu and Rohrlich, Daniel},
  journal={Foundations of Physics},
  volume={24},
  number={3},
  pages={379--385},
  year={1994},
  publisher={Springer}
}

@article{wainwright2008graphical,
  title={Graphical models, exponential families, and variational inference},
  author={Wainwright, Martin J and Jordan, Michael I and others},
  journal={Foundations and Trends{\textregistered} in Machine Learning},
  volume={1},
  number={1--2},
  pages={1--305},
  year={2008},
  publisher={Now Publishers, Inc.}
}
\bibliographystyle{ieeetr}

\normalfont

\end{document}